\documentclass[a4paper,11pt]{amsart}
\usepackage{dsfont}
\usepackage{graphicx}
\usepackage{amsmath}
\usepackage{amssymb}
\usepackage[active]{srcltx}
\usepackage{hyperref}
\usepackage{enumerate}
\usepackage{bbm}
\usepackage{amsthm}
\usepackage{mathrsfs}

\newtheorem{thm}{Theorem}
\newtheorem{lem}[thm]{Lemma}
\newtheorem{prop}[thm]{Proposition}
\newtheorem{cor}[thm]{Corollary}
\newtheorem*{cor*}{Corollary}
\newtheorem*{lem*}{Lemma}

\newtheorem{defn}{Definition}[section]

\theoremstyle{remark}
\newtheorem{remark}{Remark} 

\theoremstyle{plain}

\def\CC{{\mathbb C}}
\def\EE{{\mathbb E}}

\def\HH{{\mathbb H}}
\def\NN{{\mathbb N}}

\def\RR{{\mathbb R}}

\def\ZZ{{\mathbb Z}}

\def\veca{{\text{\boldmath$a$}}}
\def\vecb{{\text{\boldmath$b$}}}

\def\vece{{\text{\boldmath$e$}}}

\def\vecu{{\text{\boldmath$u$}}}
\def\vecv{{\text{\boldmath$v$}}}

\def\scrA{{\mathcal A}}

\def\scrC{{\mathcal C}}
\def\scrD{{\mathcal D}}
\def\scrE{{\mathcal E}}
\def\scrF{{\mathcal F}}
\def\scrH{{\mathcal H}}

\def\scrI{{\mathcal I}}
\def\scrJ{{\mathcal J}}
\def\scrK{{\mathcal K}}

\def\scrP{{\mathcal P}}
\def\scrR{{\mathcal R}}
\def\scrS{{\mathcal S}}

\def\scrY{{\mathcal Y}}

\def\Re{\operatorname{Re}}
\def\Im{\operatorname{Im}}

\def\GL{\operatorname{GL}}

\def\Res{\operatorname{Res}}

\def\SL{\operatorname{SL}}

\def\SO{\operatorname{SO}}

\def\PSL{\operatorname{PSL}}

\def\sgn{\operatorname{sgn}}

\def\Proj{\operatorname{Proj}}

\def\Onder#1#2#3#4#5{#1 \setbox0=\hbox{$#1$}\setbox1=\hbox{$#2$}
       \dimen0=.5\wd0 \dimen1=\dimen0 \dimen2=\dp0 \dimen3=\dimen2
       \advance\dimen0 by .5\wd1 \advance\dimen0 by -#4
       \advance\dimen1 by -.5\wd1 \advance\dimen1 by -#4
       \advance\dimen2 by -#3 \advance\dimen2 by \ht1
       \advance\dimen2 by 0.3ex \advance\dimen3 by #5
        \kern-\dimen0\raisebox{-\dimen2}[0ex][\dimen3]{\box1}
       \kern\dimen1}

\newcommand{\GaG}{\Gamma\backslash G}
\newcommand{\GaH}{\Gamma\backslash \mathbb{H}}

\newcommand{\sfrac}[2]{{\textstyle \frac {#1}{#2}}}

\newcommand{\smatr}[4]{\left( \begin{smallmatrix} #1 & #2 \\ #3 & #4 \end{smallmatrix} \right) }

\newcommand{\wty}{\widetilde{y}}
\newcommand{\wtk}{\widetilde{k}}

\newcommand{\kthe}{k_{\theta}}

\newcommand{\sYGa}{\mathcal{Y}_{\Gamma}}
\begin{document}
\title[Eisenstein series and analytic continuation of representations]{Renormalization of integrals of products of Eisenstein series and analytic continuation of representations}

\author{Samuel C. Edwards}

\address{Department of Mathematics, Yale University, New Haven, CT, USA}
\address{Department of Mathematics, Uppsala University, Uppsala, Sweden}
\email{samuel.edwards@yale.edu}


\date{\today}

\subjclass[2000]{}


\begin{abstract}
We combine Zagier's theory of \emph{renormalizable} automorphic functions on the hyperbolic plane with the analytic continuation of representations of $\mathrm{SL}(2,\mathbb{R})$ due to Bernstein and Reznikov to study triple products of Eisenstein series of arbitrary (in particular, non-arithmetic) non-compact finite-volume hyperbolic surfaces. 
\end{abstract}

\maketitle
\tableofcontents
\section{Introduction}\label{intro}
Let $\Gamma < \PSL(2,\RR)$ be a non-cocompact cofinite Fuchsian group, which acts by isometries on the hyperbolic upper half plane $\HH=\lbrace z\in\CC\,:\, \Im(z)>0\rbrace$ by way of M\"obius transformations:
\begin{equation*}
g\cdot z:= \sfrac{az+b}{cx+d}\qquad\forall g=\pm \smatr a b c d \in \PSL(2,\RR),\;z\in \HH.
\end{equation*}
Recall that this action also preserves the  hyperbolic area element $d\mu(x+iy)=\frac{dx\,dy}{y^2}$. Since the hyperbolic surface $\GaH$ is non-compact and of finite volume, it has a finite number $\kappa$ ($1\leq \kappa <\infty$) of cusps. We choose a set of representatives $\lbrace\eta_1,\eta_2,\ldots \eta_{\kappa}\rbrace\subset\partial_{\infty}\HH= \RR\cup\lbrace \infty\rbrace$ for the cusps; $\lbrace\eta_1,\eta_2,\ldots \eta_{\kappa}\rbrace$ is thus a maximal set of $\Gamma$-inequivalent parabolic fixed points of $\Gamma$. 

To each $\eta_j$ and $z\in\HH$, $s\in\CC$ we associate an \emph{Eisenstein series} $E_j(z,s)$ (see Section \ref{Eispolessec}, \eqref{Eisdef} for the definition). Recall that for fixed $z\in\HH$, $E_j(z,s)$ is a meromorphic function in $s$, and for fixed $s\in\CC$ that is not a pole of the Eisenstein series, $E_j(z,s)$ is an automorphic form. Furthermore, for $s$ which is not a pole, each Eisenstein series is an eigenfunction of the hyperbolic Laplacian $\Delta= y^{-2}(\frac{\partial^2}{\partial x^2}+\frac{\partial^2}{\partial y^2})$; $-\Delta E_j(x+i y, s)= s(1-s)E_j(x+i y, s)$.   

For each representative $\eta_k$ of a cusp of $\GaH$, there exists a re-parametrization of $\HH$ given by $z=x+iy \leftrightarrow x_k+iy_k$ (see Section \ref{prelims} for a rigorous definition) such that each Eisenstein series $E_j(z,s)$ has the following Fourier decomposition with respect to $x_k+i y_k$:
\begin{equation}\label{EjkFourier}
E_j(z,s)= \delta_{j, k} y_k^s + \varphi_{j,k}(s) y_k^{1-s} + \sum_{\underset{m\neq 0}{m\in\ZZ}} \psi_{m,k}^{(j)}(s)\sqrt{y_k} K_{s-\frac{1}{2}}(2 \pi |m| y_k) e( m x_k),
\end{equation}
where (as is standard) $e(m x_k)=e^{2\pi i m x_k}$, $K_u(r)$ is the $K$-Bessel function (cf.\ \cite[Chapter 8.4]{GR}), and the coefficients  $\varphi_{j,k}(s)$, $\psi_{m,k}^{(j)}(s)$ are meromorphic functions in $s$. The purpose of this article is to study the asymptotic behaviour of the coefficients $ \psi_{m,k}^{(j)}(s)$ as $|m|\rightarrow \infty$.

We fix $k_0\in \lbrace 1,\ldots,\kappa\rbrace$ and $s_0=\sigma_0+i t_0\in \CC$ with $\sigma_0\geq \frac{1}{2}$, and write $E(z)=E_{k_0}(z,s_0)$. For each $k\in\lbrace 1,\ldots,\kappa\rbrace$ we can define a \emph{Rankin-Selberg $L$-function} $L_k(|E|^2,s)$ by
\begin{equation}\label{Lkdef}
L_k(|E|^2,s):=\sum_{\underset{m\neq 0}{m\in\ZZ}} \frac{|\psi_{m,k}^{(k_0)}(s_0)|^2}{|m|^s}, \qquad \Re(s) > 2\sigma_0.
\end{equation}
The sum is absolutely convergent for $s$ as above due to the following bound on sums of $|\psi_{m,l}^{(j)}(s_0)|^2$ (see \cite[Lemma 2.7]{AndJensKron}):
\begin{equation}\label{RanSelbd}
\sum_{\underset{0<|m|\leq M}{m\in\ZZ}} |\psi_{m,l}^{(j)}(s_0)|^2 \ll_{\Gamma,s_0} \begin{cases} M \log 2 M\qquad&\mathrm{if}\,\sigma_0=\sfrac{1}{2}\\ M^{2\sigma_0}\qquad &\mathrm{if}\,\sigma_0>\sfrac{1}{2}\end{cases}\qquad \forall j,l\in\lbrace1,\ldots,\kappa\rbrace.
\end{equation}
In the case $s_0\in\frac{1}{2}+i\RR$, we prove the following:
\begin{thm}\label{thm1} For $T>1$, $s_0\in \frac{1}{2}+i\RR$, and $\epsilon>0$
\begin{equation*}
\int_0^{T} |L_k(|E|^2,\sfrac{1}{2}+it)|^2\,dt\ll_{\Gamma,s_0,\epsilon} T^{6+\epsilon}.
\end{equation*}
\end{thm}
By a standard argument due to Good \cite{Good} (see also \cite{Petridis95}), Theorem \ref{thm1} may be used to obtain more precise asymptotics for the sums of the $|\psi_{m,l}^{(j)}(s_0)|^2$. For each $j\in\lbrace 1,\ldots,\kappa\rbrace$, let $\scrE_j\subset \CC$ denote the set of poles for $E_j(z,s)$. One then has $\#(\scrE_j\cap\lbrace s\in\CC\,:\,\Re(s)\geq \frac{1}{2}\rbrace)=\#(\scrE_j\cap (\frac{1}{2},1])<\infty$ (for all $j$); see Section \ref{Eispolessec} for a summary of (and references to) of the known properties of $\scrE_j$ which we will need. We can now state the following
\begin{cor}\label{Fccor} Given $t_0\in\RR\setminus\lbrace 0\rbrace$, there exist $c_1$, $c_{1+2it_0}$, $c_{\zeta}$ ($\zeta\in \scrE_k \cap(\frac{1}{2},1)$) in $\CC$ such that for all $\epsilon>0$,
\begin{align*}
\sum_{\underset{0<|m|\leq M}{m\in\ZZ}}& |\psi_{m,k}^{(k_0)}(\sfrac 1 2 +i t_0)|^2= \frac{16\cosh(\pi t_0)}{\mu(\GaH)\pi}M\log M\\[-13pt]&\qquad\qquad\qquad\qquad+\frac{8\cosh(\pi t_0)}{\mu(\GaH)\pi}\left(c_1\mu(\GaH)+2\log (4\pi)-2-2\Re\left(\sfrac{\Gamma'(\frac{1}{2}+i t_0)}{\Gamma(\frac{1}{2}+i t_0)} \right) \right)M\\&\qquad\qquad\qquad\qquad+\Re\big(M^{1+2 i t_0} c_{1+2 i t_0}\big)+\!\!\sum_{\zeta\in (\frac{1}{2},1)\cap\scrE_k}\!\! c_{\zeta}M^{\zeta}+O_{\Gamma, t_0,\epsilon}\big(M^{\frac{6}{7}+\epsilon}\big)
\end{align*}
as $M\rightarrow\infty$. The numbers $c_1$, $c_{1+2 i t_0}$ and $c_{\zeta}$ may be explicitly computed in terms of $t_0$ and $\Gamma$, and are given by the formulas \eqref{c1def}, \eqref{c12it0}, and \eqref{czeta}, respectively.
\end{cor}

\begin{remark}
As noted above, Corollary \ref{Fccor} follows from Theorem \ref{thm1} by more-or-less standard arguments from analytic number theory (involving contour integrals and Cauchy's residue theorem). For the sake of completeness, we give a detailed proof in Appendix \ref{AppC}. 
\end{remark}

\begin{remark}
The restriction $t_0\neq 0$ in Corollary \ref{Fccor} is to give a relatively simple statement; for $t_0\neq 0$, $L_k(|E|^2,s)$ has simple poles at $s=1\pm 2 i t_0$ and a pole of order two at $s=1$. The residues at these points give rise to the main terms in the expression in Corollary \ref{Fccor}. As such, when $t_0=0$, a different expression is required. It appears that our method may also be generalized without any major complications in order to prove analogues of Theorem \ref{thm1} for $s_0$ with $\Re(s_0)>\frac{1}{2}$.
\end{remark}

\begin{remark}
Note that Corollary \ref{Fccor} gives rise to the following bound on the individual Fourier coefficients of $E_{k_0}(z,\frac{1}{2}+i t_0)$:
\end{remark}
\begin{cor}\label{Fccor1} For $t_0,\,m\neq 0$, $\epsilon>0$
\begin{equation*}
|\psi_{m,k}^{(k_0)}(\sfrac 1 2 +i t_0)|\ll_{\Gamma,t_0,\epsilon} |m|^{\frac{3}{7}+\epsilon}.
\end{equation*}
\end{cor}

\begin{remark}\label{Ramanujan}
Recall that a \emph{Maass cusp form} $\phi$ of $\Gamma$ also has a Fourier expansion similar to \eqref{EjkFourier} at the cusp $\eta_k$ given by
\begin{equation*}
\phi(x+iy)= \sum_{\underset{m\neq 0}{m\in\ZZ}} a_m\sqrt{y_k} K_{s-\frac{1}{2}}(2 \pi |m| y_k) e( m x_k),
\end{equation*}
where $s\in \CC$ is such that $-\Delta\phi=s(1-s)\phi$. For arithmetic $\Gamma$, the Ramanujan-Petersson conjecture for Mass cusp forms is the (conjectured) bound $|a_m| \ll_{\epsilon} |m|^{\epsilon}$ for all $\epsilon>0$. To date, the best result in this direction is $|a_m|\ll |m|^{5/28+\epsilon}$, due to Bump, Duke, Hoffstein, and Iwaniec \cite{BDHI}. In \cite{Sarnak}, Sarnak gave the first improvement over the Hecke bound $|a_m|\ll |m|^{1/2}$ for \emph{all} lattices $\Gamma$, in particular, also the non-arithmetic ones (related work by Petridis \cite{Petridis95, Petridis99} further improved these bounds). Key to these results is the use of analytic continuation to bound integrals over $\GaH$ of triple products of Maass forms. Bernstein and Reznikov \cite{Bern} subsequently interpreted the analytic continuation in terms of representation theory and combined this with a theory of $\SL(2,\RR)$-invariant Sobolev norms to obtain (essentially) the optimal bound for the triple product method: an error term $O\big(M^{\frac{2}{3}+\epsilon}\big)$ in the sum corresponding to Corollary \ref{Fccor} for cusp forms-this gives the bound $|a_m|\ll_{\epsilon} |m|^{1/3+\epsilon}$ on the individual Fourier coefficients. These results (together with certain numerical investigations and hueristic arguments, cf., e.g., \cite{Hej}) suggest that the Ramanujan-Petersson conjecture might also hold for the Maass cusp forms of non-arithmetic $\Gamma$.
\end{remark}

\begin{remark}
Returning our attention to Eisenstein series, we note that for $\Gamma$ that are congruence subgroups, we have a much greater understanding of the Fourier coefficients of Eisenstein series than those of Maass cusp forms. Indeed, for $\Gamma=\SL(2,\ZZ)$, a ``classical" computation gives $\psi_{m,1}^{(1)}(s)= \frac{2\pi^s|m|^{s-1/2} }{\Gamma(s)}\frac{\sigma_{1-2 s}(|m|)}{\zeta(2s)}$, where as usual $\sigma_{z}(n)=\sum_{d|n} d^z$ ($n\in\NN$, $z\in \CC$). Standard bounds on $\sigma_z$ then give $|\psi_{m,1}^1(s)|\ll_{\epsilon} |m|^{\epsilon+\Re(s)-\frac{1}{2}}$ for $\Re(s)\geq \frac{1}{2}$, i.e. the Ramanujan-Petersson conjecture holds for the Eisenstein series of $\SL(2,\ZZ)$ (similar computations may be done for other congruence lattices). 

Similarly to the case of Maass cusp forms, less is known for non-arithmetic lattices. Using \eqref{RanSelbd}, one obtains the convexity bound $|\psi_{j,m}^{(k)}(s)|\ll_{\epsilon} |m|^{\Re(s)+\epsilon}$ ($\Re(s)\geq \frac{1}{2}$) of Eisenstein series of any lattice $\Gamma$. Corollary \ref{Fccor1} thus provides the first (as far as we are aware) subconvexity bound for the Fourier coefficients of Eisenstein series of non-arithmetic $\Gamma$. As for cusp forms, one might conjecture that the Ramanujan-Petersson bound $|\psi_{j,m}^{(k)}(s)|\ll_{\epsilon} |m|^{\Re(s)-\frac{1}{2}+\epsilon}$ ($\Re(s)\geq \frac{1}{2}$) also holds for the Eisenstein series of non-arithmetic lattices, i.e.\ it is a general property of lattices that does originate in number theory. 
\end{remark}

\begin{remark}
It is conjectured (the Phillips-Sarnak conjecture, cf.\ e.g. \cite{PS85a,PS85b,PS92}) that for a ``generic" lattice $\Gamma$, $\GaH$ has only finitely many Maass cusp forms. Thus, the Eisenstein series are of great importance for the spectral theory of such lattices $\Gamma$; they are expected to contribute ``almost all" of the spectrum of $\Delta$ in $L^2(\GaH)$. 
\end{remark}

\begin{remark}\label{Eisentripprod}
Eisenstein series are not in $L^2(\GaH)$. The integrals of the products that one needs to consider when using the triple product method are thus not finite. In order to get round this, we follow Zagier \cite{Zagier}, who developed the Rankin-Selberg method for a certain class of functions that are not of rapid decay in the cusps of $\GaH$. The idea behind this is to give meaning to, or \emph{renormalize}, certain divergent integrals in a way that captures the information required for the Rankin-Selberg method. Theorem \ref{thm1} follows (cf.\ Section \ref{implication}) from a more general (and our main) result, Theorem \ref{thm2}, which is a bound on integrals of renormalized integrals of triple products of Eisenstein series.
\end{remark}

\begin{remark}
The main work of this article consists of combining the ideas of Bernstein and Reznikov with those of Zagier described, respectively, in Remarks \ref{Ramanujan} and \ref{Eisentripprod} above. We start by obtaining expressions for renormalized integrals of triple products of Eisenstein series in terms of (standard) integrals of linear combinations of these triple products with other Eisenstein series. This enables us to construct a function in $L^2(\GaG)$ (here $G=\SL(2,\RR)$ and $\Gamma$ is identified with its inverse image under the map $g\mapsto \pm g$ from $G$ to $\PSL(2,\RR)$) whose ``coefficients" in the direct integral decomposition of $L^2(\GaG)$ into irreducible representations are given by precisely the renormalized integrals of triple products that we are interested in. Interpreting Eisenstein series as intertwining operators from principal series representations of $G$ to $C^{\infty}(\GaG)$ then allows us to use the analytic continuation of representations as in \cite{Bern} to obtain subconvexity bounds for the Fourier coefficients of Eisenstein series.

Our initial hope was that at this stage we would be able to use the results of \cite{Bern} regarding $G$-invariant Sobolev norms in order to improve the error term in Corollary \ref{Fccor} to  $O\big(M^{\frac{2}{3}+\epsilon} \big)$. However, due to the slightly more complicated nature of renormalized integrals (compared with standard integrals), we have so far not been able to see how to do so.
\end{remark}

\begin{remark}
Unpublished work of Str\"ombergsson \cite{Anotes} suggests that by more closely following the method of \cite{Sarnak}, one should be able to obtain Corollary \ref{Fccor} with an error term $O\big(M^{\frac{3}{4}+\epsilon}\big)$. In \cite{Avelin}, Avelin carried out numerical investigations of the statistical properties of the Fourier coefficients of Eisenstein series, including Corollary \ref{Fccor}, for certain cases of non-arithmetic $\Gamma$ with $\kappa=1$.
\end{remark}

\subsection{Acknowledgements}
This research was partially funded by Swedish Research Council Grant 2016-03360, as well as by a postdoctoral scholarship from the Knut and Alice Wallenberg Foundation. I would like to thank my supervisor Andreas Str\"ombergsson for suggesting this problem, sharing his notes from his previous work \cite{Anotes}, many useful and enlightening discussions, and carefully reading this manuscript.

\section{Renormalization}\label{renorm}
We now follow Zagier \cite{Zagier} and develop a theory of \emph{renormalizable integrals} over $\GaH$. In \cite{Zagier}, Zagier deals only with the case $\Gamma=\PSL(2,\ZZ)$. Further generalisations to congruence subgroups and the adelic setting have been carried out by a number of authors, cf.\ \cite{Ben,DuttaGupta,DuttaGupta2,Michel}. Here, however, we deal with arbitrary (in particular, non-arithmetic) $\Gamma<\PSL(2,\RR)$.

\subsection{Preliminaries}\label{prelims}
 We start by choosing a \emph{canonical fundamental domain} $\scrF\subset \HH$ for $\Gamma$, see \cite[p. 268]{Hejhal}. We denote the vertices of $\scrF$ along $\partial_{\infty}\HH=\RR\cup\lbrace \infty\rbrace$ by $\eta_1,\ldots,\eta_{\kappa}$ (recall that we have assumed that $\kappa \geq 1$). Note that since $\scrF$ is canonical, the set $\lbrace \eta_1,\ldots,\eta_{\kappa}\rbrace$ is a maximal set of $\Gamma$-inequivalent parabolic fixed points. For each $k\in \lbrace 1,\ldots \kappa\rbrace$ we choose an element $h_k\in\PSL(2,\RR)$ such that $h_k\cdot \eta_k=\infty$ and $h_k\Gamma_{\eta_k}h_k^{-1}= \pm \smatr{1}{\ZZ}{0}{1}$ ($\Gamma_{\eta_k}$ being the stabilizer of $\eta_k$ in $\Gamma$). Since $\scrF$ is canonical, we may assume without loss of generality that each $h_k$ has been chosen so that
\begin{equation*}
h_k\cdot \scrF \cap \lbrace z\in \HH\,:\, \Im(z)\geq B\rbrace =\lbrace x+i y\in \HH\,:\, 0\leq x\leq 1,\, y\geq B\rbrace 
\end{equation*}
for any $B$ greater than a constant $B_0=B_0(\Gamma)>1$, which we fix once and for all. Given $B\geq B_0$, we define the \emph{cuspidal region at} $\eta_k$, $\scrC_{k,B}$, by
\begin{equation*}
\scrC_{k,B}=h_k^{-1}\cdot\lbrace x+i y\in \HH\,:\, 0\leq x\leq 1,\, y\geq B\rbrace \subset \scrF,
\end{equation*}
and also a bounded region $\scrF_B$ by
\begin{equation*}
\scrF_B=\scrF\setminus \bigcup_{k=1}^{\kappa} \scrC_{k,B}.
\end{equation*}
For each $k\in\lbrace 1,\ldots,\kappa\rbrace$ and $z \in \HH$, we write
\begin{equation}\label{cusppara}
z_k=x_k+i y_k=h_k\cdot z.
\end{equation}
This parametrization may be used to express integrals over cusp regions $\scrC_{k,B}$ as follows:
\begin{equation}\label{cuspint}
\int_{\scrC_{k,B}} F(z)\,d\mu(z)=\int_{h_k^{-1}\cdot[0,1]\times i[B,\infty)} F(z)\,d\mu(z)=\int_{B}^{\infty}\int_0^1 F(z)\,\frac{dx_k\,dy_k}{y_k^2},
\end{equation}
were in the last expression we interpret $z$ through $z=h_k^{-1}\cdot z_k=h_k^{-1}\cdot(x_k+i y_k)$. In connection with our choice of cusps and elements $h_k$, we define the \emph{invariant height function} $\scrY_{\Gamma}$ by
\begin{equation}\label{YGadef}
\sYGa(z):=\max_{k\in\lbrace 1,\ldots,\kappa\rbrace} \max_{\gamma\in \Gamma}\, \Im(h_k \gamma\cdot z)
\end{equation}
(cf.\ \cite[(3.8)]{Iwaniec}). We note that $\sYGa$ is a continuous, $\Gamma$-invariant function on $\HH$, and for each $k\in\lbrace 1,\ldots,\kappa\rbrace$, we have the bound
\begin{equation}\label{sYkjbd}
\sYGa(z)\leq \max\lbrace y_k,\,y_k^{-1}\rbrace.
\end{equation}
This is seen by writing $h_j \gamma h_k^{-1}=\smatr a b c d$, where $j,k\in \lbrace 1,\ldots,\kappa\rbrace$, $\gamma\in \Gamma$. Then $\Im(h_j\gamma \cdot z)= \Im(\smatr a b c d \cdot z_k)=\Im( \frac{a z_k + b}{c z_k +d})=\frac{y_k}{|c z_k +d|^2}$. Now, either $|c|\geq 1$ or $\smatr a b c d = \smatr 1 * 0 1$ (cf.\ \cite[Lemma 2.3]{Strom1}), hence $\frac{y_k}{|c z_k +d|^2}\leq \max\lbrace y_k,y_k^{-1}\rbrace$. Let $\delta_B>0$ be such that for all $k\in\lbrace 1,\ldots,\kappa\rbrace$, $y_k\geq \delta_B$ for all $z\in\scrF_B\cup \scrC_{k,B}$. We will write $y\gg1$ for $y\geq \delta_B$. Unless stated otherwise, all implied constants are dependent on $\Gamma$.

\subsection{Poles of Eisenstein series}\label{Eispolessec}
We now recall some standard facts regarding the poles of the Eisenstein series that will be needed. Our references for this section are principally \cite[Chapter 11]{Borel} and \cite[Chapters 6.9, 6.11, and 6 ]{Hejhal} (in particular, see the following: \cite[Claim 9.6 (p.\ 78), Theorem 11.8 (p.\ 130), Theorem 11.11 (p.\ 140), and p.\ 284 (item 12), p.\ 297 (E), p.\ 298 (line 6)]{Hejhal}.

Recall that the Eisenstein series $E_k(z,s)$ is defined through the meromorphic continuation (in $s$) of
\begin{equation}\label{Eisdef}
E_k(z,s)=\sum_{\gamma \in \Gamma_{\eta_k}\backslash \Gamma} \Im( h_k\gamma\cdot z)^s \qquad z\in \HH,\,\Re(s)>1,
\end{equation} 
and that we have defined $\scrE_k\subset \CC$ to be the set of poles of $E_k(z,s)$. We now let $\scrE=\bigcup_{k=1}^{\kappa}\scrE_k$. As stated in Section \ref{intro}, $\scrE\cap\lbrace s\in\CC\,:\,\Re(s)\geq \frac{1}{2}\rbrace=\scrE\cap(\frac{1}{2},1]$. Moreover, all the poles in $\scrE\cap(\frac{1}{2},1]$ are simple. The poles of an Eisenstein series $E_j(z,s)$ are closely linked to the poles of the functions $\varphi_{j,k}(s)$. Each $\varphi_{j,k}(s)$ is a meromorphic function (in $s$) and $E_j(z,s)$ has a pole of order less than or equal to $n$ at $s=\zeta$ if and only if all $\varphi_{j,k}$, $k=1,\ldots,\kappa$ also have poles of order less than or equal to $n$ at $s=\zeta$. In summary: all $E_j(z,s)$ and $\varphi_{j,k}(s)$ are holomorphic along the line $\frac{1}{2}+i\RR$, and for all $s$ with $\Re(s)>\frac{1}{2}$ apart from at a finite number of simple poles in $(\frac{1}{2},1]$.

We have the following result regarding the behaviour of the Eisenstein series at $s=1$:
\begin{lem}\label{Eis1pole} All Eisenstein series $E_j(z,s)$ and $\varphi_{j,k}(s)$ ($j,k\in\lbrace 1,\ldots,\kappa\rbrace$) have simple poles at $s=1$. Moreover,
\begin{equation*}
\Res_{s=1} E_j(z,s)=\Res_{s=1}\varphi_{j,k}(s)= \frac{1}{\mu(\GaH)}\qquad \qquad\forall j,k\in\lbrace 1,\ldots,\kappa\rbrace.
\end{equation*}
Consequently, for all $j,k\in\lbrace 1,\ldots,\kappa\rbrace$ there exist  functions $\widetilde{E}_j:\HH\times\CC\rightarrow\CC$ and $\widetilde{\varphi}_{j,k}:\CC\rightarrow\CC$ such that
\begin{equation*}
E_j(z,s)=\frac{\mu(\GaH)^{-1}}{s-1}+\widetilde{E}_j(z,s),
\end{equation*}
\begin{equation*}
\varphi_{j,k}(s)=\frac{\mu(\GaH)^{-1}}{s-1}+\widetilde{\varphi}_{j,k}(s),
\end{equation*}
where $s\mapsto\widetilde{E}_j(z,s)$, $s\mapsto\widetilde{\varphi}_{j,k}(s)$ are holomorphic in a neighbourhood of $1$, $\widetilde{E}_j(z,s)$ is jointly continuous in $(z,s)$ and an automorphic function in $z$ for those $s$ that are not poles of $\widetilde{E}_j$.
\end{lem}

\subsection{Renormalizable functions and renomalized integrals}
\begin{defn}\label{Renormdef1}
A continuous, $\Gamma$-invariant function $F(z)$ on $\HH$ is said to be \underline{renormalizable} if, for each $k\in \lbrace 1,\ldots,\kappa\rbrace$, 
\begin{equation}\label{renormform1}
F(z)= \Xi_k(y_k)+O(y_k^{-A}) \qquad (\forall A\in \RR)\;\mathrm{as}\; y_k\rightarrow\infty,
\end{equation}
where each $\Xi_k(y)$ is a function of the form
\begin{equation}\label{Xikdef}
\Xi_k(y)= \sum_{i=1}^{I_k} \frac{c_{ki}}{n_{ki}!}  y^{\alpha_{ki}}\log^{n_{ki}}y
\end{equation}
(here $c_{ki}\in\CC\setminus\lbrace 0\rbrace,\,\alpha_{ki}\in \CC,\,I_k,\,n_{ki}\in\NN$, and $(\alpha_{ki},n_{ki})\neq (\alpha_{k'i'},n_{k'i'})$ for $(k,i)\neq(k',i')$).
\end{defn}
We note (though will not need) that this definition is in fact independent of the choices we've made (viz.\ $\scrF$, the elements $h_k$, and the cusps $\eta_k$). Observe also that the set of renormalizable functions is an algebra over $\CC$; if $F_1(z)$ and $F_2(z)$ are renormalizable, then so is $F_1(z)F_2(z)$.
\begin{defn}\label{Renormdef2}
We define the \underline{renormalized integral}, denoted $R.N.\left(\int_{\GaH} F(z)\,d\mu(z)\right)$, of a renormalizable function $F(z)$ by
\begin{align}\label{RNintdef}
R.N.\left(\int_{\GaH} F(z)\,d\mu(z)\right):=\int_{\scrF_B} F(z)\,d\mu(z)+\sum_{k=1}^{\kappa}\left( \int_{\scrC_{k,B}} \left(F(z)-\Xi_k(y_k)\right)\,d\mu(z)-\widehat{\Xi}_k(B)\right),
\end{align}
where $B\geq B_0$ is arbitrary, and
\begin{equation}\label{Xihatdef}
\widehat{\Xi}_k(y)=\sum_{\substack{1 \leq i \leq I_k \\ 
\alpha_{ki} \neq 1}} c_{ki} \sum_{m=0}^{n_{ki}}
\frac{(-1)^{n_{ki}-m}}{m!} \cdot 
\frac{y^{\alpha_{ki}-1} \log^m y}{(\alpha_{ki}-1)^{n_{ki}-m+1}}+\sum_{\substack{1 \leq i \leq I_k \\ 
\alpha_{ki} = 1}} c_{ki}
\frac{\log^{n_{ki}+1} y}{(n_{ki}+1)!}
\end{equation}
($c_{ki},\,\alpha_{ki},\,I_k,\,n_{ki}$ being as in \eqref{Xikdef}).
\end{defn}
That $B\geq B_0$ is permitted to be arbitrary in Definition \ref{Renormdef2} is due to the fact that
\begin{equation*}
\frac{d}{dy} \widehat{\Xi}_k(y)=y^{-2}\Xi_k(y),
\end{equation*}
which (when combined with \eqref{cuspint}) shows that this definition does in fact make sense. From the definition, we see that $F\mapsto R.N.\left( \int_{\GaH} F(z)\,d\mu(z)\right)$ is a linear map from the space of renormalizable functions to $\CC$. 
\begin{lem}\label{intlem}
Let $F\in L^1(\GaH)$ be renormalizable. Then
\begin{equation*}
R.N.\left(  \int_{\GaH} F(z)\,d\mu(z)\right)= \int_{\GaH} F(z)\,d\mu(z).
\end{equation*}
\end{lem}
\begin{proof}
Since $F$ is integrable, \eqref{cuspint} implies that $\alpha_{ki}< 1$ for all choices of $k$ and $i$ in \eqref{Xikdef}. Using this, we have
\begin{align*}
R.N.&\left(\int_{\GaH} F(z)\,d\mu(z)\right)=\int_{\scrF_B} F(z)\,d\mu(z)+\sum_{k=1}^{\kappa}\left( \int_{\scrC_{k,B}} F(z)-\Xi_k(y_k)\,d\mu(z)-\widehat{\Xi}_k(B)\right)\\&=\int_{\scrF_B} F(z)\,d\mu(z)+\sum_{k=1}^{\kappa} \int_{\scrC_{k,B}} F(z)\,d\mu(z)-\sum_{k=1}^{\kappa} \int_{\scrC_{k,B}} \Xi_k(y_k)\,d\mu(z)-\sum_{k=1}^{\kappa} \widehat{\Xi}_k(B)
\\&=\int_{\GaH} F(z)\,d\mu(z)+\sum_{k=1}^{\kappa} \widehat{\Xi}_k(B)-\sum_{k=1}^{\kappa} \widehat{\Xi}_k(B).
\end{align*}
\end{proof}

\subsection{Renormalization of integrals of products of Eisenstein series}
Our principal source of renormalizable functions are Eisenstein series; that Eisenstein series are indeed renormalizable follows from entering the bound (see \cite[pp. 295-297]{Hejhal}, \cite[6.20]{Iwaniec})
\begin{equation}\label{Eisexp}
 \left|\sum_{\underset{m\neq 0}{m\in\ZZ}} \psi_{m,k}^{(j)}(s)\sqrt{y_k} K_{s-\frac{1}{2}}(2 \pi |m| y_k) e( m x_k) \right|= O_s(e^{-2 \pi y_k})
\end{equation}
into \eqref{EjkFourier}. The implied constant in this bound is uniform for $y_k\gg 1$ and all $s$ in any compact subset of $\CC\setminus \scrE_j$; this will be of importance later. As previously observed, pointwise multiplication of renormalizable function gives a new renormalizable function. We therefore obtain further examples of renormalizable functions by multiplying Eisenstein series with each other. We now prove a series of results regarding renormalized integrals of products of Eisenstein series, and state our main theorem. 
\begin{lem}\label{maassel}
Let $k_1,k_2\in \lbrace 1,\ldots ,\kappa\rbrace$ and $s_1,s_2\in \CC\setminus\scrE$ and $\overline{s_2}\neq s_1,1-s_1$. Then
\begin{equation}\label{masselrel}
R.N.\left( \int_{\GaH} E_{k_1}(z,s_1)\overline{E_{k_2}(z,s_2)}\,d\mu(z)\right)=0.
\end{equation}
\end{lem}
\begin{proof}
This is essentially a reformulation of the \emph{Maass-Selberg relations}, cf.\ \cite{Casselman}, \cite[pp.\ 427-429]{Zagier}. We recall the set-up: for $B\geq B_0$, let $E_k^B(z,s)$ be the $\Gamma$-invariant function that is defined through
\begin{equation*}
E_k^B(z,s)= E_k(z,s)-\sum_{i=1}^{\kappa} \mathbbm{1}_{\scrC_{i,B}}(z)\left(\delta_{k,i} y_i^s +\varphi_{k,i}(s) y_i^{1-s} \right)\qquad z\in\scrF
\end{equation*}
($E_k^B(z,s)$ is usually called a \emph{truncated} Eisenstein series). Using \eqref{Eisexp}, we see that $E_k^B(z,s)\in L^p(\GaH)$ for all $p\geq 1$. The Maass-Selberg relations then give an expression for the inner product (in $L^2(\GaH)$) of two truncated Eisenstein series:
\begin{align}\label{MaasSelexp}
\int_{\GaH} E_{k_1}^B(z,s_1)\overline{E_{k_2}^B(z,s_2)}\,d\mu(z)&=\sum_{i=1}^{\kappa} \bigg(\frac{\delta_{k_1, i}\delta_{k_2, i}B^{s_1+\overline{s_2}-1}-\varphi_{k_1,i}(s_1)\overline{\varphi_{k_2,i}(s_2)}B^{1-s_1-\overline{s_2}}}{s_1+\overline{s_2}-1}\\\notag&\qquad\qquad+ \frac{\delta_{k_1, i}\overline{\varphi_{k_2,i}(s_2)}B^{s_1-\overline{s_2}}-\varphi_{k_1,i}(s_1)\delta_{k_2 ,i}B^{\overline{s_2}-s_1}}{s_1-\overline{s_2}}\bigg)\\\notag&= \frac{\delta_{k_1 ,k_2} B^{s_1+\overline{s_2}-1}}{s_1+\overline{s_2}-1}+\frac{\overline{\varphi_{k_2,k_1}(s_2)}B^{s_1-\overline{s_2}}}{s_1-\overline{s_2}}+\frac{\varphi_{k_1,k_2}(s_1)B^{\overline{s_2}-s_1}}{\overline{s_2}-s_1}\\\notag&\qquad\qquad-\left(\sum_{i=1}^{\kappa} \varphi_{k_1,i}(s_1)\overline{\varphi_{k_2,i}(s_2)} \right)\frac{B^{1-s_1-\overline{s_2}}}{s_1+\overline{s_2}-1},
\end{align}
 cf. \cite[Chapter 2.3]{Kubota}. On the other hand, viewing $F(z)=E_{k_1}(z,s_1)\overline{E_{k_2}(z,s_2)}$ as a renormalizable function, from \eqref{EjkFourier} and \eqref{Eisexp}, we have
\begin{align*}
F(z)=&\left(  \delta_{k_1,i} y_i^{s_1} + \varphi_{k_1,i}(s_1) y_{i}^{1-s_1} + O(e^{-2\pi y_i}) \right)\\&\qquad\cdot\left(\overline{ \delta_{k_2,i} y_i^{s_2} + \varphi_{k_2,i}(s_2) y_{i}^{1-s_2} } + O(e^{-2\pi y_i}) \right)\qquad \mathrm{as}\, y_i\rightarrow\infty\\=& \delta_{k_1,i}\delta_{k_2,i}y_i^{s_1+\overline{s_2}}+\delta_{k_2,i}\varphi_{k_1,i}(s_1)y_i^{1+\overline{s_2}-s_1}+\delta_{k_1,i}\overline{\varphi_{k_2,i}(s_2)}y_i^{1+s_1-\overline{s_2}}\\&\qquad\qquad\qquad\qquad\qquad+\varphi_{k_1,i}(s_1)\overline{\varphi_{k_2,i}(s_2)}y_i^{2-s_1-\overline{s_2}}+ O(y_i^{-A}) \qquad (\forall A)\,\mathrm{as}\, y_i\rightarrow\infty,
\end{align*}
 giving
 \begin{align*}
 \Xi_i(y) =&\delta_{k_1,i}\delta_{k_2,i}y^{s_1+\overline{s_2}}+\delta_{k_2,i}\varphi_{k_1,i}(s_1)y^{1+\overline{s_2}-s_1}\\&\qquad\qquad+\delta_{k_1,i}\overline{\varphi_{k_2,i}(s_2)}y^{1+s_1-\overline{s_2}}+\varphi_{k_1,i}(s_1)\overline{\varphi_{k_2,i}(s_2)}y^{2-s_1-\overline{s_2}}.
 \end{align*}
Using \eqref{Xihatdef} and \eqref{MaasSelexp}, we observe that
\begin{equation*}
\sum_{i=1}^{\kappa} \widehat{\Xi}_i(B)=\int_{\GaH} E_{k_1}^B(z,s_1)\overline{E_{k_2}^B(z,s_2)}\,d\mu(z),
\end{equation*}
hence
\begin{align*}
0&=\int_{\GaH} E_{k_1}^B(z,s_1)\overline{E_{k_2}^B(z,s_2)}\,d\mu(z)-\sum_{i=1}^{\kappa} \widehat{\Xi}_i(B)\\&=\int_{\scrF} \left( E_{k_1}(z,s_1)-\sum_{i=1}^{\kappa} \mathds{1}_{\scrC_{i,B}}(z)\left(\delta_{k_1,i} y_i^{s_1} +\varphi_{k_1,i}(s_1) y_i^{1-s_1} \right)\right)\\&\qquad\qquad\times \left( \overline{E_{k_2}(z,s_2)}-\sum_{i=1}^{\kappa} \mathds{1}_{\scrC_{i,B}}(z)\left(\delta_{k_2,i} y_i^{\overline{s_2}} +\overline{\varphi_{k,i}(s_2)} y_i^{1-\overline{s_2}} \right)\right)\,d\mu(z)-\sum_{i=1}^{\kappa} \widehat{\Xi}_i(B)\\&=\int_{\scrF_B}F(z)\,d\mu(z) - \sum_{i=1}^{\kappa} \widehat{\Xi}_i(B)+\sum_{i=1}^{\kappa} \int_{\scrC_{i,B}} \left(E_{k_1}(z,s_1)-\delta_{k_1,i} y_i^{s_1}-\varphi_{k_1,i}(s_1) y_i^{1-s_1} \right)\\&\qquad\qquad\qquad\qquad\qquad\qquad\qquad\qquad\qquad\times \left(\overline{E_{k_2}(z,s_2)}-\delta_{k_2,i} y_i^{\overline{s_2}} -\overline{\varphi_{k_2,i}}(s_2) y_i^{1-\overline{s_2}} \right)\,d\mu(z)\\&=\int_{\scrF_B}F(z)\,d\mu(z) - \sum_{i=1}^{\kappa} \widehat{\Xi}_i(B)+\sum_{i=1}^{\kappa} \int_{\scrC_{i,B}} \bigg(F(z)-\Xi_i(y_i)\\&\qquad\qquad\qquad-\left( \delta_{k_1,i} y_i^{s_1}-\varphi_{k_1,i}(s_1) y_i^{1-s_1}\right)\left(\overline{E_{k_2}(z,s_2)}-\delta_{k_2,i} y_i^{\overline{s_2}} -\overline{\varphi_{k_2,i}}(s_2) y_i^{1-\overline{s_2}} \right)\\&\qquad\qquad\qquad-\left( \delta_{k_2,i} y_i^{\overline{s_2}} -\overline{\varphi_{k_2,i}}(s_2) y_i^{1-\overline{s_2}} \right)\left(E_{k_1}(z,s_1)-\delta_{k_1,i} y_i^{s_1}-\varphi_{k_1,i}(s_1) y_i^{1-s_1} \right)\bigg)\,d\mu(z).
\end{align*}
Once again using \eqref{Eisexp}, we have as $ y_i\rightarrow \infty$,
\begin{align}\label{prodbdd}
&\left( \delta_{k_1,i} y_i^{s_1}-\varphi_{k_1,i}(s_1) y_i^{1-s_1}\right)\left(\overline{E_{k_2}(z,s_2)}-\delta_{k_2,i} y_i^{\overline{s_2}} -\overline{\varphi_{k_2,i}}(s_2) y_i^{1-\overline{s_2}} \right)=O_{s_1,s_2,\epsilon} ( e^{-2(\pi-\epsilon) y_i})\\\notag&\left( \delta_{k_2,i} y_i^{\overline{s_2}} -\overline{\varphi_{k_2,i}}(s_2) y_i^{1-\overline{s_2}} \right)\left(E_{k_1}(z,s_1)-\delta_{k_1,i} y_i^{s_1}-\varphi_{k_1,i}(s_1) y_i^{1-s_1} \right)=O_{s_1,s_2,\epsilon} ( e^{-2(\pi-\epsilon) y_i}),
\end{align}
hence
\begin{align*}
0=&\int_{\scrF_B}F(z)\,d\mu(z)+\sum_{i=1}^{\kappa} \int_{\scrC_{i,B}} \big(F(z)-\Xi_i(y_i)\big)\,d\mu(z) - \sum_{i=1}^{\kappa} \widehat{\Xi}_i(B)\\&-\sum_{i=1}^{\kappa} \int_{\scrC_{i,B}}\left( \delta_{k_1,i} y_i^{s_1}-\varphi_{k_1,i}(s_1) y_i^{1-s_1}\right)\left(\overline{E_{k_2}(z,s_2)}-\delta_{k_2,i} y_i^{\overline{s_2}} -\overline{\varphi_{k_2,i}}(s_2) y_i^{1-\overline{s_2}} \right)\,d\mu(z) \\&-\sum_{i=1}^{\kappa} \int_{\scrC_{i,B}}\left( \delta_{k_2,i} y_i^{\overline{s_2}} -\overline{\varphi_{k_2,i}}(s_2) y_i^{1-\overline{s_2}} \right)\left(E_{k_1}(z,s_1)-\delta_{k_1,i} y_i^{s_1}-\varphi_{k_1,i}(s_1) y_i^{1-s_1} \right)\,d\mu(z).
\end{align*}
Using \eqref{prodbdd} in lines two and three of the previous expression yields
\begin{equation*}
R.N.\left(\int_{\GaH} F(z)\,d\mu(z) \right)=O_{s_1,s_2,\epsilon}(e^{-(2\pi-\epsilon) B});
\end{equation*}
since the left-hand side of this expression does not depend on $B$, letting $B\rightarrow\infty$ concludes the proof.
\end{proof}
We will now use Lemma \ref{maassel} to show how one can compute certain renormalized integrals over $\GaH$ by way of ordinary integrals.
We first make an auxiliary definition:
\begin{defn}\label{Psirsdef}
Given $j,\,k\in\lbrace 1,\ldots,\kappa\rbrace$ and $r,\,s\in \CC$, let $\Phi_{j,k}^{r,s}(z)$ be defined by
\begin{align}\label{Psirsdefeq}
\Phi_{j,k}^{r,s}(z):&= E_{j}(z,r)E_{k}(z,s)-\delta_{j,k}E_{k}(z,r+s)\\\notag&\qquad-\varphi_{k,j}(s)E_{j}(z,r+1-s)-\varphi_{j,k}(r)E_{k}(z,1-r+s)\\\notag&\qquad\qquad-\sum_{i=1}^{\kappa}\varphi_{j,i}(r)\varphi_{k,i}(s)E_i(z,2-r-s).
\end{align}
\end{defn}
Note that for fixed $r$ and $s$ such that none of the Eisenstein series in \eqref{Psirsdefeq} has a pole, $\Phi_{j,k}^{r,s}(z)$ is an automorphic function in $z$. Furthermore, fixing $z$ and one of $r,s$ yields a meromorphic function in the remaining parameter ($\Phi_{j,k}^{r,s}(z)$ being the product and linear combination of meromorphic functions). As we shall see, $\Phi_{j,k}^{r,s}(z)$ is holomorphic for more values of $r$ and $s$ than \eqref{Psirsdefeq} would initially suggest; this is due to a cancellation of poles of Eisenstein series occurring in \eqref{Psirsdefeq}. 

\begin{prop}\label{Phiprops}
The map $(r,s)\mapsto\Phi_{j,k}^{r,s}(z)$ is holomorphic for all $j,\,k\in\lbrace 1,\ldots,\kappa\rbrace$ and $(r,s)\in (\frac{1}{2}+i \RR)\times(\frac{1}{2}+i \RR)$. Furthermore, $\Phi_{j,k}^{r,s}\in L^p(\GaH)$ for all $p<\infty$ for all such $r,s$, and
\begin{equation}\label{tripleprodint}
R.N.\left(\int_{\GaH} E_j(z,r)E_k(z,s)\overline{E_l(z, \sfrac{1}{2}+it)}\,d\mu(z) \right)=\int_{\GaH} \Phi_{j,k}^{r,s}(z)\overline{E_l( z,\sfrac{1}{2}+it)}\,d\mu(z)
\end{equation}
for all $l\in\lbrace 1,\ldots,\kappa\rbrace$ and $t\in \RR$.
\end{prop}
\begin{proof}
We choose an arbitrary $T>0$, and prove the statement for all $r,s\in \frac{1}{2}+i(-T,T)$. Since the poles of all the Eisenstein series and all $\varphi_{j,k}(s)$ with real part greater or equal to $\frac{1}{2}$ are finitely many and lie in the half-open interval $(\frac{1}{2},1]$, there exists $\delta=\delta_T>0$ such that all the terms in \eqref{Psirsdefeq} are well-defined (i.e.\ finite) for all $(r,s)\in\big([\frac{1}{2}-2\delta,\frac{1}{2}+2\delta]+i[-T-\delta,T+\delta]\big) \times \big([\frac{1}{2}-2\delta,\frac{1}{2}+2\delta]+i[-T-\delta,T+\delta]\big)$ apart from at $(r,s)=(r,r)$ or $(r,s)=(r,1-r)$. Being jointly meromorphic in $(r,s)$, the function $(r,s)\mapsto \Phi_{j,k}^{r,s}(z)$ is therefore holomorphic at all $(r,s)\in \big( (\frac{1}{2}-\delta,\frac{1}{2}+\delta)+i(-T,T)\big)\times  \big( (\frac{1}{2}-\delta,\frac{1}{2}+\delta)+i(-T,T)\big)$ except for possibly the previously mentioned exceptions. We thus fix $r\in (\frac{1}{2}-\delta,\frac{1}{2}+\delta)+i(-T,T)$ and consider $s$ in a neighbourhood of $r$ or $1-r$. Since $s\mapsto \Phi_{j,k}^{r,s}(z)$ is meromorphic in $s$, in order to show that $s\mapsto \Phi_{j,k}^{r,s}(z)$ is in fact holomorphic at $s=r$ or $s=1-r$, it suffices to show that $s\mapsto \Phi_{j,k}^{r,s}(z)$ remains \emph{bounded} as $s\rightarrow r$ or $1-r$. By symmetry, the same will hold for $r\mapsto \Phi_{j,k}^{r,s}(z)$, allowing us to use Hartogs' theorem to ascertain that $(r,s)\mapsto \Phi_{j,k}^{r,s}(z)$ is (jointly) holomorphic on $ \big( \frac{1}{2}+i(-T,T)\big)\times  \big(\frac{1}{2}+i(-T,T)\big)$.

There are three cases to consider: \textit{i)} $s=1-r$, $r\neq \frac{1}{2}$ \textit{ii)} $s=r$, $r\neq \frac{1}{2}$, \textit{iii)} $r=\frac{1}{2}$. We will consider only case \textit{i)}, which is of most interest to us (the other cases are dealt with in a similar fashion). Letting $s=1-r+w$, with $w\neq 0$ in a small neighbourhood of $0$, we have
\begin{align*}
\Phi_{j,k}^{r,1-r+w}(z)&= E_{j}(z,r)E_{k}(z,1-r+w)-\delta_{j,k}E_{k}(z,1+w)\\\notag&\qquad-\varphi_{k,j}(1-r+w)E_{j}(z,2r-w)-\varphi_{j,k}(r)E_{k}(z,2-2r+w)\\\notag&\qquad\qquad-\sum_{i=1}^{\kappa}\varphi_{j,i}(r)\varphi_{k,i}(1-r+w)E_i(z,1-w).
\end{align*}
We now recall (see \cite{Borel,Hejhal,Iwaniec}) that the \emph{scattering matrix} $\big( \varphi_{i,l}(r)\big)_{1\leq i,l,\leq\kappa}$ satisfies 
\begin{enumerate}
\item$\big( \varphi_{i,l}(r)\big)_{i,l}=\big( \varphi_{l,i}(r)\big)_{i,l}$ (i.e.\ it is symmetric).
\item We have $\big( \varphi_{i,l}(r)\big)_{i,l}\big( \varphi_{i,l}(1-r)\big)_{i,l}=I$ for all $r,1-r\in\CC\setminus\scrE$ ($I$ being the $\kappa\times\kappa$ identity matrix).
\item $r\mapsto\varphi_{i,l}(r)$ is holomorphic on $[\frac{1}{2}-2\delta,\frac{1}{2}+2\delta]+i[-T-\delta,T+\delta]$. 
\end{enumerate}
These facts give
\begin{equation*}
\sum_{i=1}^{\kappa} \varphi_{j,i}(r)\varphi_{k,i}(1-r+w)=\sum_{i=1}^{\kappa} \varphi_{j,i}(r)\varphi_{i,k}(1-r+w)=\delta_{j,k}+O(|w|).
\end{equation*}
Using this together with Lemma \ref{Eis1pole}, we obtain
\begin{align*}
\Phi_{j,k}^{r,1-r+w}(z)&= E_{j}(z,r)E_{k}(z,1-r+w)-\delta_{j,k}\left( \frac{\mu(\GaH)^{-1}}{w}+\widetilde{E}_{k}(z,1+w)\right)\\\notag&\qquad-\varphi_{k,j}(1-r+w)E_{j}(z,2r-w)-\varphi_{j,k}(r)E_{k}(z,2-2r+w)\\\notag&\qquad\qquad-\sum_{i=1}^{\kappa}\varphi_{j,i}(r)\varphi_{k,i}(1-r+w)\left( \frac{\mu(\GaH)^{-1}}{-w}+\widetilde{E}_{i}(z,1-w)\right)\\&= E_{j}(z,r)E_{k}(z,1-r+w)-\delta_{j,k}\widetilde{E}_{k}(z,1+w)\\\notag&\qquad-\varphi_{k,j}(1-r+w)E_{j}(z,2r-w)-\varphi_{j,k}(r)E_{k}(z,2-2r+w)\\\notag&\qquad\qquad-\sum_{i=1}^{\kappa}\varphi_{j,i}(r)\varphi_{k,i}(1-r+w)\widetilde{E}_{i}(z,1-w)+ O(|w|)\frac{\mu(\GaG)^{-1}}{w}.
\end{align*}
Since $r\in \big(\frac{1}{2}-\delta,\frac{1}{2}+\delta)+i(-T,T)\big)\setminus\lbrace\frac{1}{2}\rbrace$, all terms in this expression remain bounded as $w\rightarrow 0$, as desired.

In order to prove that $\Phi_{j,k}^{r,s}\in L^p(\GaH)$ for all $r,s \in \frac{1}{2}+i(T,T)$ and $p<\infty$, we first assume that $r\in \frac{1}{2}+i(-T,T)\big)$, $s\in (\frac{1}{2}-\delta,\frac{1}{2}+\delta)+i(-T-\delta,T+\delta)$, $s\neq r,1-r$, and use the Fourier expansion \eqref{EjkFourier} in a cusp $\eta_l$ together with the bound \eqref{Eisexp} in the definition \eqref{Psirsdefeq} of $\Phi_{j,k}^{r,s}$, giving
\begin{align*}
\Phi_{j,k}^{r,s}(z)= -\delta_{j,k} \varphi_{k,l}(r+s)y_l^{1-r-s}&-\varphi_{k,j}(s)\varphi_{j,l}y_l^{s-r}-\varphi_{j,k}(r)\varphi_{k,l}y_l^{r-s}\\&-\sum_{i=1}^{\kappa}\varphi_{j,i}(r)\varphi_{k,i}(s)\varphi_{i,l}(2-r-s) y_l^{r+s-1}+O_{r,s,\epsilon}(e^{-2(\pi-\epsilon)y_l}),
\end{align*}
valid for all $y_l\gg 1$. From this, we see that if $s,r\in \frac{1}{2}+i\RR$, $s\neq r,1-r$, then $\Phi_{j,k}^{r,s}(z)=O_{r,s}(1)$, hence $\Phi_{j,k}^{r,s}\in L^{\infty}(\GaH)$. Moreover, since the implied constant in \eqref{Eisexp} is uniform over compact subsets of $\CC\setminus\scrE$, fixing $r\in \frac{1}{2}+i(-T,T)$ and letting $K\subset [\frac{1}{2}-2\delta,\frac{1}{2}+2\delta]+i[-T-\delta,T+\delta]$ be a compact subset such that $r,1-r\not\in K$, we have
\begin{align}\label{PhiKbdd}
\Phi_{j,k}^{r,s}(z)= -\delta_{j,k} \varphi_{k,l}(r+s)y_l^{1-r-s}&-\varphi_{k,j}(s)\varphi_{j,l}y_l^{s-r}-\varphi_{j,k}(r)\varphi_{k,l}y_l^{r-s}\\\notag&-\sum_{i=1}^{\kappa}\varphi_{j,i}(r)\varphi_{k,i}(s)\varphi_{i,l}(2-r-s) y_l^{r+s-1}+O_{r,K,\epsilon}(e^{-2(\pi-\epsilon)y_l})
\end{align}
for all $s\in K$ and $y_l\gg 1$. For $\delta>\xi>0$, let $D_{r,\xi}\subset \CC$ denote the following set: $\lbrace \zeta\in \CC\,:\, |r-\zeta|<\xi\mathrm{\;or\;}|1-r-\zeta|<\xi\rbrace$. Having proved that $ (r,s)\mapsto\Phi_{j,k}^{r,s}(z)$ is holomorphic on $(\frac{1}{2}+i\RR)\times(\frac{1}{2}+i\RR)$, $s\mapsto  \Phi_{j,k}^{r,s}(z)$ is holomorphic on $D_{r,\xi}$. By the maximum modulus principle, $|\Phi_{j,k}^{r,s}(z)|\leq\max_{\zeta\in\partial D_{r,\xi}} |\Phi_{j,k}^{r,\zeta}(z)|$, which, after applying \eqref{PhiKbdd} with $K=\partial D_{r,\xi}$, gives 
\begin{equation}\label{PhiDbd}
|\Phi_{j,k}^{r,s}(z)|\ll_{r,\xi} y_l^{\xi}\qquad \forall s\in D_{r,\xi},\,y_l\gg 1, \xi>0.
\end{equation}
In particular, this shows that $\Phi_{j,k}^{r,r}$, $\Phi_{j,k}^{r,1-r}\in L^p(\GaH)$ for all $r\in\frac{1}{2}+i\RR$ and $p<\infty$. 

It remains to prove \eqref{tripleprodint}. We fix $r\in\frac{1}{2}+i \RR$, and view both sides of \eqref{tripleprodint} as functions of $s$. Considering first the case $s\neq r,\,1-r$, we let $\alpha_1=r+s$, $\alpha_2=r+1-s$, $\alpha_3=1-r+s$, and $\alpha_4=2-r-s$. Since $\Re(\alpha_i)=1$, $\overline{\frac{1}{2}+it}\neq \alpha_i$ or $1-\alpha_i$ for all $i$. Furthermore, $\alpha_i\neq 1$, hence none of the $\alpha_i$s are poles of the Eisenstein series. This allows us to use Lemma \ref{maassel}:
\begin{equation*}
R.N\left( \int_{\GaH} E_m(z,\alpha_i)\overline{E_{l}(z,\sfrac{1}{2}+it)}\,d\mu(z)\right)=0\qquad \forall i\in\lbrace 1,\ldots,4\rbrace, \,l,m\in\lbrace 1,\ldots,\kappa\rbrace.
\end{equation*} 
The linearity of renormalized integrals then gives
\begin{align*}
&R.N.\left( \int_{\GaH} E_{j}(z,r)E_{k}(z,s)\overline{E_{l}(z,\sfrac{1}{2}+it)}\,d\mu(z)\right)\\&\qquad\qquad\qquad\qquad\qquad=R.N.\left(\int_{\GaH} \Phi_{j,k}^{r,s}(z)\overline{E_{l}(z,\sfrac{1}{2}+it)}\,d\mu(z)\right),
\end{align*} 
and so by Lemma \ref{intlem}, all that remains is to show that $  \Phi_{j,k}^{r,s}(z)\overline{E_{j}(z,\sfrac{1}{2}+it)}$ is integrable. This is indeed the case: $\overline{E_{j}(z,\sfrac{1}{2}+it)}\in L^{2-\epsilon}(\GaH)$ for all $\epsilon>0$, and we have previously shown that $\Phi_{j,k}^{r,s}\in L^{\infty}(\GaG)$, so their product is in $L^1(\GaH)$. We will now show that both sides of \eqref{tripleprodint} are continuous with respect to $s$ at $s=r$ and $s=1-r$, and hence also agree at these points.

Starting with the right-hand side of \eqref{tripleprodint}, and letting $s=r$ or $1-r$, we have 
\begin{equation*}
\Phi_{j,k}^{r,s}(z)\overline{E_{j}(z,\sfrac{1}{2}+it)}=\lim_{w\rightarrow s}\Phi_{j,k}^{r,w}(z)\overline{E_{j}(z,\sfrac{1}{2}+it)}.
\end{equation*}
Recalling the bound \eqref{PhiDbd}, we let $w\in D_{r,\xi}$, where (e.g.) $\xi < \frac{1}{3}$, giving
\begin{equation*}
|\Phi_{j,k}^{r,w}(z)\overline{E_{j}(z,\sfrac{1}{2}+it)}|\ll_{r,t} y_l^{\frac{1}{3}} y_l^{\frac{1}{2}}\qquad \forall z\in \scrF_B\cup \scrC_{l,B};
\end{equation*}
this allows is to use Lebesgue's dominated convergence theorem to conclude that the left-hand side of \eqref{tripleprodint} is indeed continuous with respect to $s$ at $s=r$ and $s=1-r$.

Turning our attention to the left-hand side of \eqref{tripleprodint}, again we let $w\rightarrow s$, where $s=r$ or $1-r$. From the definition of renormalized integrals \eqref{RNintdef}, we have
\begin{align}\label{rntrip}
\lim_{w\rightarrow s} R.N.&\left( \int_{\GaH} E_{j}(z,r)E_{k}(z,w)\overline{E_{l}(z,\sfrac{1}{2}+it)}\,d\mu(z)\right)\\\notag&\qquad\qquad=\lim_{w\rightarrow s} \int_{\scrF_B}  E_{j}(z,r)E_{k}(z,w)\overline{E_{l}(z,\sfrac{1}{2}+it)}\,d\mu(z)\\\notag&\qquad\qquad\quad+\sum_{i=1}^{\kappa} \lim_{w\rightarrow s} \int_{\scrC_{i,B}} \left(E_{j}(z,r)E_{k}(z,w)\overline{E_{l}(z,\sfrac{1}{2}+it)}-\Xi_i(y_i)\right)\,d\mu(z)\\\notag&\qquad\qquad\qquad-\sum_{i=1}^{\kappa}\lim_{w\rightarrow s}  \widehat{\Xi}_i(B),
\end{align}
where
\begin{equation*}
\Xi_i(y)=(\delta_{j,i}y^{r}+\varphi_{j,i}(r) y^{1-r})(\delta_{k,i}y^{w}+\varphi_{k,i}(w) y^{1-w})(\delta_{l,i}y^{\frac{1}{2}-i t}+\varphi_{l,i}(\sfrac{1}{2}-it) y^{\frac{1}{2}+i t}).
\end{equation*}
The uniformity of the bound \eqref{Eisexp} (in $s$) over compact subsets of $\CC$ again allows us to use the dominated convergence theorem to exchange the limits with the integrals in rows two and three of \eqref{rntrip}. One sees from the definition \eqref{Xihatdef} that $w \mapsto \widehat{\Xi}_i(B)$ is continuous (in particular, note that all the powers of $y$ in $\Xi_i(y)$ have real part $\frac{3}{2}$), completing the proof.
\end{proof}
We now state our main result:
\begin{thm}\label{thm2} For $r,s\in \frac{1}{2}+i\RR$ and $j,k,l\in\lbrace 1,\ldots,\kappa\rbrace$, 
\begin{equation*}
\int_0^T \left| R.N.\left(\int_{\GaH} E_j(z,r)E_k(z,s)\overline{E_l(z, \sfrac{1}{2}+it)}\,d\mu(z) \right)\right|^2e^{\pi t}\,dt \ll_{r,s,\epsilon}T^{4+\epsilon}\qquad \forall T>1.
\end{equation*}
\end{thm}

\section{The Rankin-Selberg transform}
In this section we will define the \textit{Rankin-Selberg transform} of a renormalizable function and connect this with the renormalized integrals of the previous section. Here we essentially just follow Str\"ombergsson \cite{Anotes} (unpublished), who generalized the corresponding results of Zagier \cite{Zagier} for $\Gamma=\PSL(2,\ZZ)$ to general lattices $\Gamma$. 

\subsection{The Rankin-Selberg transform} Let $F(z)$ be a renormalizable function with $\alpha_{j,i}$ and $\Xi_k$ as in \eqref{Xikdef}. For notational purposes, let $M_0=\max_{k,i} |\Re(\alpha_{k,i})|$ (using the convention that $M_0=0$ if $I_k=0$ for all $k\in\lbrace 1,\ldots,\kappa\rbrace$). 
\begin{defn}\label{RSdef}
For a renormalizable function $F(z)$ and $k\in\lbrace 1,\ldots \kappa\rbrace$, let
\begin{equation*}
a_0^{(k)}(y)=\int_0^1 F\big(h_k^{-1}\cdot(x+i y)\big)\,dx,
\end{equation*}
and for $s\in \CC$ with $ \Re(s)> 1+ M_0$, define the \underline{Rankin-Selberg transform} $R_k(F,s)$ of $F$ by 
\begin{equation}\label{RSint}
R_k(F,s):= \int_0^{\infty} \left( a_0^{(k)}(y)-\Xi_k(y)\right)y^{s-2}\,dy.
\end{equation}
\end{defn}

\begin{lem}\label{RSprops} The following hold for $s\in \CC$ with $\Re(s)> M_0+1$:
\begin{enumerate} 
\item The integral in \eqref{RSint} defining $R_k(F,s)$ is absolutely convergent.
\item $R_k(F,s)=R.N.\left( \int_{\GaH} F(z)E_k(z,s)\,d\mu(z)\right)$. 
\end{enumerate}
\end{lem}
\begin{proof}
Starting with \textit{(1)}, we choose and fix $\epsilon>0$ so that $\Re(s)> M_0+1+\epsilon$, and note that by \eqref{renormform1} and \eqref{Xikdef}, for $z\in \scrF$ we have $|F(z)|\ll \sYGa(z)^{M_0+\epsilon}$ (cf.\ \eqref{YGadef}). Since both $F$ and $\sYGa$ are $\Gamma$-invariant, this inequality then holds for all $z\in\HH$. Fixing $B\geq B_0$, for all $y\in [B,\infty)$, we have
\begin{equation*}
a_0^{(k)}(y)-\Xi_k(y)=\int_0^1 F\big(h_k^{-1}\cdot(x+iy)\big)\,dx-\Xi_k(y)=\int_0^1\left( \Xi_k(y)+O(y^{-A})\right)\,dx-\Xi_k(y)=O(y^{-A})
\end{equation*}
(here $A\in\RR$ is arbitrary). We thus have
\begin{equation*}
\int_B^{\infty} |a_0^{(k)}(y)-\Xi_k(y)| y^{\Re(s)-2}\,dy\ll \int_B^{\infty} y^{\Re(s)-2-A} <\infty;
\end{equation*}
valid for all $A>\Re(s)+1$.
Since $F$ and $\Xi_k$ are continuous, $\int_1^B |a_0^{(k)}(y)-\Xi_k(y)| y^{\Re(s)-2}\,dy<\infty$, for all $s\in \CC$. For $0<y\leq 1$, we use the bound $F\big(h_k^{-1}\cdot(x+iy)\big)\ll \sYGa\big(h_k^{-1}\cdot(x+iy)\big)^{M_0+\epsilon}\leq \max \lbrace y,y^{-1}\rbrace^{M_0+\epsilon}=y^{-M_0-\epsilon}$ (here we used \eqref{sYkjbd} for the second inequality). This bound thus also holds for $a_0^{(k)}(y)$, as well as for $\Xi_k$, giving
\begin{equation*}
\int_0^1 |a_0^{(k)}(y)-\Xi_k(y)|y^{\Re(s)-2}\,dy\ll\int_0^1y^{\Re(s)-2-M_0-\epsilon}\,dy<\infty,
\end{equation*}
since $\Re(s)>M_0+1+\epsilon$. From this proof, we see that in fact the integral in \eqref{RNintdef} converges absolutely and uniformly for any $s\in [A_1,A_2]+i\RR$, where $[A_1,A_2]\subset (M_0+1,\infty)$.

Turning to \textit{(2)}, we assume (without loss of generality) that $\eta_k=\infty$ and $h_k=\pm\smatr 1 0 0 1$; hence $\Gamma_{\infty}=\pm \smatr{1}{\ZZ}{0}{1}$. Let $\scrR_B$ denote the rectangle $[0,1]+i(0,B)$, and $G(z)$ denote the $\Gamma_{\infty}$-invariant function $G(z)=F(z)y^s$. Observe that $\scrR_B$ is a fundamental domain for action of $\Gamma_{\infty}$ on $\RR+i(0,B)$. We now construct another such fundamental domain: 
\begin{equation*}
\scrD=\scrF\setminus\scrC_{k,B}\cup\bigcup_{\underset{\gamma\not\in\Gamma_{\infty}}{\gamma\in \Gamma_{\infty}\backslash\Gamma} }\gamma \cdot \scrF.
\end{equation*}
Note that since $\Re(s)>M_0+1$, we have $\int_{\scrR_B}|G(z)|\,d\mu(z)<\infty$ (as seen in the proof of \textit{(1)}). This permits the following manipulations:
\begin{equation*}
\int_{\scrR_B}G(z)\,d\mu(z)=\int_{\scrD} G(z)\,d\mu(z),
\end{equation*}
hence
\begin{align*}
\int_0^B \int_0^1 G(z)\,d\mu(z)=\int_{\scrF\setminus \scrC_{k,B}} G(z)\,d\mu(z)+\sum_{\underset{\gamma\not\in\Gamma_{\infty}}{\gamma\in \Gamma_{\infty}\backslash\Gamma} } \int_{\gamma\cdot \scrF} G(z)\,d\mu(z).
\end{align*} 
Note that
\begin{equation*}
\int_0^B \int_0^1 G(z)\,d\mu(z)=\int_0^B \int_0^1F(z)\,dx\, y^{s-2}\,dy= \int_0^{B} a_0^{(k)}(y) y^{s-2}\,dy.
\end{equation*}
Writing $\scrF=\scrF_B\cup\bigcup_{i=1}^{\kappa}\scrC_{i,B}$, we have
\begin{equation*}
\int_{\scrF\setminus \scrC_{k,B}} G(z)\,d\mu(z)= \int_{\scrF_B} G(z)\,d\mu(z)+\sum_{\underset{i\neq k}{i=1}}^{\kappa} \int_{\scrC_{i,B}} G(z)\,d\mu(z),
\end{equation*}
and 
\begin{equation*}
\int_{\gamma\cdot\scrF} G(z)\,d\mu(z)= \int_{\gamma\cdot\scrF_B} G(z)\,d\mu(z)+\sum_{i=1}^{\kappa} \int_{\gamma\cdot\scrC_{i,B}} G(z)\,d\mu(z),
\end{equation*}
hence
\begin{align*}
\int_0^B \int_0^1 G(z)\,d\mu(z)=&\int_{\scrF\setminus \scrC_{k,B}} G(z)\,d\mu(z)+\sum_{\underset{\gamma\not\in\Gamma_{\infty}}{\gamma\in \Gamma_{\infty}\backslash\Gamma} } \int_{\gamma\cdot \scrF} G(z)\,d\mu(z)\\&= \int_{\scrF_B} G(z)\,d\mu(z)+\sum_{\underset{i\neq k}{i=1}}^{\kappa} \int_{\scrC_{i,B}} G(z)\,d\mu(z)\\&\quad + \sum_{\underset{\gamma\not\in\Gamma_{\infty}}{\gamma\in \Gamma_{\infty}\backslash\Gamma} } \left( \int_{\gamma\cdot\scrF_B} G(z)\,d\mu(z)+\sum_{i=1}^{\kappa} \int_{\gamma\cdot\scrC_{i,B}} G(z)\,d\mu(z)\right).
\end{align*}
Exchanging summation and integration (again permitted by the absolute convergence in \eqref{RSint}) and recalling \eqref{Eisdef}, we obtain
\begin{align}\label{aint}
\int_0^B a_0^{(k)}(y) y^{s-2}\,dy=\int_{\scrF_B} F(z)E_k(z,s)\,d\mu(z)+\sum_{i=1}^{\kappa}\int_{\scrC_{i,B}} F(z)\big(E_k(z,s)-\delta_{i,k}y_i^s\big)\,d\mu(z).
\end{align}
 Denoting by $\Xi_i^1(y)$ the term $``\Xi_i(y)"$  in \eqref{Xikdef} for the renormalizable function $F(z)E_k(z,s)$, for $i\neq k$, we have
\begin{align}\label{Ciint}
\int_{\scrC_{i,B}} F(z)E_k(z,s)\,d\mu(z)&=\int_{\scrC_{i,B}} \big(F(z)E_k(z,s)-\Xi_i^1(y_i)+\Xi_i^1(y_i)\big)\,d\mu(z)\\\notag=&\int_{\scrC_{i,B}} \big(F(z)E_k(z,s)-\Xi_i^1(y_i)\big)\,d\mu(z)-\widehat{\Xi_i^1}(B).
\end{align}
For $i=k$, after noting that $\Xi_k^1(y)=\Xi_k(y)(y^s +\varphi_{k,k}(s)y^{1-s})$, $\scrC_{k,B}=[0,1]+i[B,\infty)$, $y_k =y$ and $a_0^{(k)}(y)=\int_0^1 F(x+iy)\,dx$, we obtain the following:
\begin{align}\label{Ckint}
\int_{\scrC_{k,B}} &F(z)\big(E_k(z,s)-y^s\big)\,d\mu(z)\\&\notag=\int_{\scrC_{i,B}} \big(F(z)E_k(z,s)-F(y)y^s -\Xi_k^1(y)+\Xi_k(y)(y^s +\varphi_{k,k}(s)y^{1-s})\big)\,d\mu(z)\\\notag=&\int_{\scrC_{i,B}} \big(F(z)E_k(z,s)-\Xi_i^1(y)\big)\,d\mu(z)-\int_B^{\infty} \big(a_0^{(k)}(y)-\Xi_k(y)\big)y^{s-2}\,dy\\\notag&\qquad\qquad\qquad\qquad+\varphi_{k,k}(s)\int_B^{\infty}\Xi_k y^{1-s}y^{-2}\,dy-\widehat{\Xi}_k^1(B)+\widehat{\Xi}_k^1(B).
\end{align}
After recalling the definitions of renormalized integrals \eqref{RNintdef} and $R_k(F,s)$ \eqref{RSint}, substituting \eqref{Ciint} and \eqref{Ckint} into \eqref{aint} gives
\begin{align*}
R_k(F,s)=&R.N.\left( \int_{\GaH} F(z)E_k(z,s)\,d\mu(z) \right)\\&+\widehat{\Xi}_k^1(B)+\int_B^{\infty}\varphi_{k,k}(s)y^{1-s}\Xi_k(y) y^{-2}\,dy-\int_0^B \Xi_k(y) y^{s-2}\,dy.
\end{align*}
An explicit calculation using the definitions of $\Xi_k$ and $\widehat{\Xi}_k^1$ shows that the second line of this expression is in fact zero.
\end{proof}
\begin{cor}\label{RScont}
The Rankin-Selberg transform $R_k(F,s)$ has a meromorphic continuation to the entire of $\CC$. The only possible poles are at $s= \alpha_{ji},\,1-\alpha_{ki}$ and the poles of the Eisenstein series $E_k(z,s)$. For $s$ not one of these points, we have $R_k(F,s)=R.N.\left( \int_{\GaH} F(z)E_k(z,s)\,d\mu(z)\right)$.
\end{cor}
\begin{proof}
By Lemma \ref{RSprops} \textit{(2)}, it suffices to show that $s\mapsto R.N.\left( \int_{\GaH} F(z)E_k(z,s)\,d\mu(z)\right)$ is a meromorphic function that is holomorphic everywhere except possibly at the specified values. Using the definition of renormalized integrals, for $s\in \CC$ that is not a pole of an Eisenstein series or equal to any $\alpha_{ji}$, $1-\alpha_{ki}$, we have 
\begin{align}\label{RSRenormexp}
R.N.&\left( \int_{\GaH} F(z)E_k(z,s)\,d\mu(z)\right)=\int_{\scrF_B} F(z)E_k(z,s)\,d\mu(z)\\\notag&\qquad\qquad\qquad+\sum_{j=1}^{\kappa} \int_{\scrC_{j,B}} \left(F(z)E_k(z,s)-(\delta_{k,j} y_j^s+\varphi_{k,j}(s)y_j^{1-s})\Xi_j(y_j)\right)\,d\mu(z)\\\notag&\qquad\qquad\qquad\qquad-\sum_{j=1}^{\kappa}\widehat{\Xi}^1_j(B),
\end{align}
where for $j\neq k$,
\begin{equation*}
\widehat{\Xi}^1_j(B)=\varphi_{k,j}(s)\left( \sum_{i=1}^{I_j}c_{ji}\sum_{m=0}^{n_{ji}}\frac{(-1)^{n_{ji}-m}}{m!}\frac{B^{\alpha_{ji}-s}\log^mB}{(\alpha_{ji}-s)^{n_{ji}-m+1}}\right),
\end{equation*}
and
\begin{align*}
\widehat{\Xi}^1_k(B)=&\left( \sum_{i=1}^{I_k}c_{ki}\sum_{m=0}^{n_{ki}}\frac{(-1)^{n_{ki}-m}}{m!}\frac{B^{s+\alpha_{ki}-1}\log^mB}{(s+\alpha_{ki}-1)^{n_{ki}-m+1}}\right)\\&\quad+\varphi_{k,k}(s)\left( \sum_{i=1}^{I_k}c_{ki}\sum_{m=0}^{n_{ki}}\frac{(-1)^{n_{ki}-m}}{m!}\frac{B^{\alpha_{ki}-s}\log^mB}{(\alpha_{ki}-s)^{n_{ki}-m+1}}\right).
\end{align*}
From these expressions, we see that the functions $s\mapsto \widehat{\Xi}^1_j(B)$ are meromorphic for all $j\in\lbrace 1\ldots,\kappa\rbrace $, with poles contained in the set described in the statement of the corollary. It remains to prove that the integrals in lines one and two of the right-hand side of \eqref{RSRenormexp} are meromorphic functions of $s$. Once again, the uniformity of \eqref{Eisexp} in $s$ over compact subsets of $\CC\setminus\scrE_k$ is key; this bound, used together with $F(z)=\Xi_j(y_j)+O(y_j^{-A})$ for $z\in \scrC_{j,B}$, allows the use of Lebesgue's dominated convergence theorem to conclude that the functions
\begin{align*}
s&\mapsto \int_{\scrF_B} F(z)E_k(z,s)\,d\mu(z)\\s &\mapsto \int_{\scrC_{j,B}} \left(F(z)E_k(z,s)-(\delta_{k,j} y_j^s+\varphi_{k,j}(s)y_j^{1-s})\Xi_j(y_j)\right)\,d\mu(z) 
\end{align*}
are continuous on $D=\CC\setminus\scrE_k$, as well as that
\begin{align*}
&\oint_{\gamma}\int_{\scrF_B} F(z)E_k(z,s)\,d\mu(z)\,ds=\int_{\scrF_B} F(z)\oint_{\gamma}E_k(z,s)\,ds\,d\mu(z)=0\\ &\oint_{\gamma} \int_{\scrC_{j,B}} \left(F(z)E_k(z,s)-(\delta_{k,j} y_j^s+\varphi_{k,j}(s)y_j^{1-s})\Xi_j(y_j)\right)\,d\mu(z)\,ds\\&\qquad\qquad= \int_{\scrC_{j,B}}\oint_{\gamma} \left(F(z)E_k(z,s)-(\delta_{k,j} y_j^s+\varphi_{k,j}(s)y_j^{1-s})\Xi_j(y_j)\right)\,ds\,d\mu(z)=0
\end{align*} 
for any null-homotopic $C^1$ curve $\gamma$ in $D$. Morera's theorem then gives that these are holomorphic functions on $D$. Finally, we note that the points of $\scrE_k$ are not essential singularities of these functions: for $s_0\in\scrE_k$, let $m_0$ be the order of $E_k(z,s)$ at $s_0$. Then for all $j\geq m_0$ and $\gamma$ a closed $C^1$ curve around $s_0$,  arguing as above gives
\begin{align*}
&\oint_{\gamma}(s-s_0)^j\int_{\scrF_B} F(z)E_k(z,s)\,d\mu(z)\,ds=\int_{\scrF_B} F(z)\oint_{\gamma}(s-s_0)^jE_k(z,s)\,ds\,d\mu(z)=0\\ &\oint_{\gamma}(s-s_0)^j \int_{\scrC_{j,B}} \left(F(z)E_k(z,s)-(\delta_{k,j} y_j^s+\varphi_{k,j}(s)y_j^{1-s})\Xi_j(y_j)\right)\,d\mu(z)\,ds\\&\qquad\qquad= \int_{\scrC_{j,B}}\oint_{\gamma}(s-s_0)^j \left(F(z)E_k(z,s)-(\delta_{k,j} y_j^s+\varphi_{k,j}(s)y_j^{1-s})\Xi_j(y_j)\right)\,ds\,d\mu(z)=0,
\end{align*}
proving that $s_0$ is a pole of order $\leq m_0$.    
\end{proof}

We now recall that we have fixed an Eisenstein series $E(z)=E_{k_0}(z,s_0)$, and are interested in the Rankin-Selberg L-function $L_k(s)=L_k(|E|^2,s)$. 
\begin{prop}\label{GL=RS}
For $\Re(s)>2 \sigma_0+1$, 
\begin{equation}
G(s)L_k(s)=R_k(|E|^2,s),
\end{equation}
where
\begin{equation}\label{GammaFactors}
G(s)=\frac{\Gamma(\frac{s}{2}+i t_0)\Gamma(\frac{s}{2}-i t_0)\Gamma(\frac{s}{2}+\sigma_0-\frac{1}{2})\Gamma(\frac{s}{2}+\frac{1}{2}-\sigma_0)}{8\pi^s\Gamma(s)}.
\end{equation}
Consequently, $G(s)L_k(s)$ has a meromorphic continuation to all of $\CC$. The poles of $G(s)L_k(s)$ are contained in $\scrE_k$ and the points $ 2\sigma_0, 2\sigma_0-1, 2-2\sigma_0, 1-2i t_0, 1+2 i t_0,$ as well as $ 1-2 \sigma_0,\pm 2it_0$ if $k=k_0$.
\end{prop}
\begin{proof}
The function $|E(z)|^2$ is renormalizable, with each $\Xi_j(y)$, $j\in\lbrace 1,\ldots,\kappa\rbrace$, given by 
\begin{align}\label{EisXi}
\Xi_j(y)=&|\delta_{k_0,j} y^{s_0}+\varphi_{k_0,j}(s_0)y^{1-s_0}|^2\\\notag&=\delta_{k_0,j}\big( y^{2 \sigma_0}+ \varphi_{k_0,j}(s_0)y^{1-2i t_0}+ \overline{\varphi_{k_0,j}(s_0)}y^{1+2i t_0}\big)+ |\varphi_{k_0,j}(s_0)|^2 y^{2-2\sigma_0}.
\end{align}
The Fourier expansion \eqref{EjkFourier} and Parseval's formula give
\begin{align*}
\int_0^1 |E(h_k^{-1}\cdot(x_k+iy_k))|^2\,dx_k= \Xi_k(y_k) + \sum_{\underset{m\neq 0}{m\in\ZZ}} |\psi_{m,k}^{(k_0)}(s_0)|^2y_k |K_{s_0-\frac{1}{2}}(2 \pi |m| y_k)|^2,
\end{align*}
hence for $s\in \CC$ with $\Re(s)> 2\sigma_0+1$, Lemma \ref{RSprops} gives
\begin{align*}
R_k(|E|^2,s)=&\int_0^{\infty} \int_0^1 \left(|E(h_k^{-1}\cdot(x_k+iy_k))|^2- \Xi_k(y_k)\right)\, dx_k\, y_k^{s-2}\,dy_k\\&= \sum_{\underset{m\neq 0}{m\in\ZZ}}  |\psi_{m,k}^{(k_0)}(s_0)|^2 \int_0^{\infty} |K_{s_0-\frac{1}{2}}(2 \pi |m| y)|^2 y^{s-1}\,dy 
\\&= \sum_{\underset{m\neq 0}{m\in\ZZ}}  |\psi_{m,k}^{(k_0)}(s_0)|^2 \cdot \frac{\Gamma(\frac{s+1-2\sigma_0}{2})\Gamma(\frac{s+2\sigma_0-1}{2})\Gamma(\frac{s+2i t_0}{2})\Gamma(\frac{s-2i t_0}{2})}{8\pi^s |m|^s\Gamma(s)}=L_k(s)G(s), 
\end{align*}
where we used \cite[6.576 (4), p.\ 684]{GR} for the third equality. The remainder of the proposition now follows from Corollary \ref{RScont}.
\end{proof}
We now record the following bound for $L_k(s)$ with $s$ in the vertical strip $[\frac{1}{2}, \frac{3}{2}]+i\RR$:
\begin{lem}\label{apriori}
Assume $s_0\in\frac{1}{2}+i\RR$. There then exist $T_0,\,a>0$, depending only on $s_0$ and $\Gamma$, such that
\begin{equation*}
\int_{T}^{T+1}|L_k(\sigma+it)|^2\,dt \ll e^{aT}
\end{equation*}
for all $\sigma\in[\frac{1}{2},\frac{3}{2}]$ and $T\geq T_0$.
\end{lem}
\begin{proof}
From Proposition \ref{GL=RS}, $L_k(s)=\frac{R_k(|E|^2,s)}{G(s)}$, where $\frac{1}{G(s)}= \frac{8\pi^s\Gamma(s)}{\Gamma(\frac{s}{2}+i t_0)\Gamma(\frac{s}{2}-i t_0)\Gamma(\frac{s}{2})^2}$. Note that for $s\in [\frac{1}{2},\frac{3}{2}]+i\RR$, $s\mapsto\frac{1}{G(s)}$ is holomorphic and (by Stirling's formula) $|\frac{1}{G(s)}|\ll_{s_0,\epsilon} e^{(\frac{\pi}{2}+\epsilon)|\Im(s)|} $. By Corollary \ref{RScont}, 
\begin{align*}
R_k(|E|^2,s)=&\int_{\scrF_B}|E_{k_0}(z,s_0)|^2E_k(z,s)\,d\mu(z)\\&+\sum_{j=1}^{\kappa} \left(\int_{\scrC_{j,B}} \left(|E_{k_0}(z,s_0)|^2E_k(z,s)-\Xi_j^1(y_j) \right)\,d\mu(z)-\widehat{\Xi}_j^1(B) \right), 
\end{align*}
where
\begin{equation*}
\Xi_j^1(y)=\big((1+\delta_{k_0,j})y+ \delta_{k_0,j}(\varphi_{k_0,j}(s_0)y^{1-2i t_0}+ \overline{\varphi_{k_0,j}(s_0)}y^{1+2i t_0})\big)(\delta_{k,j}y^s+\varphi_{k,j}(s)y^{1-s}),
\end{equation*}
and $\widehat{\Xi}_j^1(B)$ is the corresponding function as in \eqref{Xihatdef} and the proof of Corollary \ref{RScont}. We now recall \cite[Theorem 12.9, p. 168]{Hejhal} and its generalization \cite[pp.\ 300-301]{Hejhal}: for $\frac{1}{2}\leq \sigma\leq \frac{3}{2}$ and $t\gg 1$, 
\begin{equation*}
E_k(z,\sigma + it)=\delta_{k,j}y_j^{\sigma +it}+\varphi_{k,j}(\sigma+it)y_j^{1-\sigma-it} + O\big(\sqrt{\omega(t)}e^{3|t|-\mu y_j}\big)\qquad\mathrm{for}\; y_j\gg 1,
\end{equation*}
where $\mu>0$, and $\omega$ is the \emph{spectral majorant function}; this satisfies $\omega(t)\geq 0$ for all $t$, and $\int_{-T}^T \omega(t)\,dt\ll T^2$ (cf.\ \cite[p.\ 315]{Hejhal}). Furthermore, each $\varphi_{k,j}(\sigma+it)$ is uniformly bounded over the set of $\sigma$ and $t$ being considered (again, see \cite[pp.\ 300-301]{Hejhal}). Using these facts together with \eqref{EjkFourier} and \eqref{Eisexp} gives 
\begin{equation*}
\left|\int_{\scrF_B}|E_{k_0}(z,s_0)|^2E_k(z,\sigma+it)\,d\mu(z)\right|\ll_{s_0} \sqrt{\omega(t)} e^{3|t|},
\end{equation*}
and
\begin{equation*}
\left|\int_{\scrC_{j,B}} \left(|E_{k_0}(z,s_0)|^2E_k(z,\sigma+it)-\Xi_j^1(y_j) \right)\,d\mu(z)\right|\ll_{s_0} \sqrt{\omega(t)} e^{3|t|}.
\end{equation*}
Note also that the uniform boundedness of the functions $\varphi_{k,j}(s)$ and the definition of $\widehat{\Xi}_j^1(B)$ imply that $|\widehat{\Xi}_j^1(B)|\ll_{s_0} 1$ for all $s$ with large enough imaginary part. Combining these bounds gives $|R_k(|E|^2,\sigma+it)| \ll 1+ \sqrt{\omega(t)}e^{3|t|}$, hence 
\begin{align*}
\int_T^{T+1} |L_k(\sigma+it)|^2\,dt=\int_T^{T+1} |G(\sigma+it)|^{-2}|R_k(|E|^2,\sigma+it)|^2\,dt\ll e^{a' T}\int_T^{T+1} \omega(t)\,dt \ll e^{a T}
\end{align*}
for some $a>a'>0$.
\end{proof}

\subsection{Proof of Theorem \ref{thm1} assuming Theorem \ref{thm2}}\label{implication}
Note that $\overline{L_k(s)}=L_k(\overline{s})$, so by Proposition \ref{GL=RS}:
\begin{align*}
\int_{0}^{T}& |L_k(\sfrac{1}{2}+it)|^2\,dt=\int_{0}^{T} |R_k(|E|^2,\sfrac{1}{2}-it)|^2 |G(\sfrac{1}{2}-it)|^{-2}\,dt.
\end{align*}
Stirling's formula gives $|G(\frac{1}{2}-it)| \asymp e^{-\pi|t|/2}(1+|t|)^{-1}$, hence by Corollary \ref{RScont} and Theorem \ref{thm2}:
\begin{align*}
\int_{0}^{T}& |L_k(\sfrac{1}{2}+it)|^2\,dt \ll T^2 \int_{0}^{T} |R_k(|E|^2,\sfrac{1}{2}-it)|^2 e^{\pi t}\,dt\\& \leq T^2 \int_0^T \left|R.N.\left( \int_{\GaH} E_{k_0}(z,s_0)E_{k_0}(z,\overline{s_0})\overline{E_k(z,\sfrac{1}{2}+it)}\,d\mu(z)\right) \right|^2 e^{\pi t}\,dt\ll_{s_0,\epsilon} T^{5+\epsilon}.
\end{align*}

\hspace{425.5pt}\qedsymbol

\section{Representation theory and Eisenstein series}
We now recall the connection between Eisenstein series and the representation theory of $G=\SL(2,\RR)$. Our main references for this section are \cite[Chapters 15-17]{Borel} and \cite[Section 2]{AndJensKron}.

\subsection{Principal series representations} Let $N$, $A$, $K$ denote the following subgroups of $G$:
\begin{align*}
N&=\lbrace \smatr 1 x 0 1 \,:\, x\in \RR\rbrace\\
A&=\lbrace \smatr {a}{0}{0}{a^{-1}}\,:\, a\in\RR_{>0}\in \RR\rbrace\\
K&=\SO(2).
\end{align*}
We parametrize the subgroups $N$, $A$, and $K$ as follows: let 
\begin{equation}
n_x:=\smatr{1}{x}{0}{1}\!\in N,\, a_y:=\smatr{\sqrt{y}}{0}{0}{\sqrt{y}^{-1}}\!\in A,\, k_{\theta}:=\smatr{\cos\theta}{-\sin\theta}{\sin\theta}{\cos\theta}\!\in K,\, \forall x\in\RR,\,y\in\RR_{>0},\,\theta\in \RR/2\pi\ZZ.
\end{equation}
The \emph{Iwasawa decomposition} of $G$ gives $G=NAK$, and to every $g\in G$, there exist \emph{unique} $x$, $y$, and $\theta$ as above such that $g=n_xa_yk_{\theta}$. The Iwasawa decomposition is particularly convenient when working with M\"obius transformations:
\begin{equation*}
n_xa_y\kthe\cdot i=x+ i y\qquad\forall x\in\RR,\,y\in\RR_{>0},\,\theta\in \RR/2\pi\ZZ,
\end{equation*}
and $K=\mathrm{Stab}_G(i)$ (this last fact gives rise to the identification $\HH\cong G/K$). 

We define functions $\wty:G\rightarrow \RR_{>0}$ and $\wtk:G\rightarrow K$ through
\begin{equation*}\wty(n_xa_y\kthe):= y,\qquad\qquad \wtk( n_xa_y\kthe):=\kthe.
\end{equation*}
These functions are well-defined by the uniqueness of the Iwasawa decomposition. We note that $\wty(g)=\Im(g\cdot i)$, $\wty(a_y g)=y\,\wty(g)$,  $\wty(n g)=\wty(g)$, $\wty(gk)=\wty(g)$, and $\wtk(gk)=\wtk(g)k$ for all $g\in G$, $n\in N$, $k\in K$, and $y>0$.

For each $j\in\lbrace 1,\ldots,\kappa\rbrace$, we define $k_j:=\wtk(h_j)$, $N_j=k_j^{-1} N k_j$, $A_j=k_j^{-1}A k_j$, and note that $N_j A_j= h_j^{-1} N A h_j$, as well as $\pm N_jA_j =\mathrm{Stab}_G(\eta_j)$. This gives rise to an Iwasawa decomposition of $G$ for each $j\in\lbrace 1 ,\ldots,\kappa\rbrace$: $G=N_j A_j K$. Let $\wtk_j: G\rightarrow K$ be defined through
\begin{equation*}
\wtk_j(nak)=k; \qquad \forall n\in N_j,\, a\in A_j,\,k\in K,
\end{equation*}
and $\wty_j:G\rightarrow \RR_{>0}$ by
\begin{equation*}
\wty_j(g):=\wty\big(k_j g\big).
\end{equation*}
Observe that
\begin{equation*}
\wty_j(n_xa_y \kthe)=\wty(h_j)^{-1}\Im\big(h_j\cdot(x+ i y)\big)\qquad \forall x+ iy\in\HH,\,\theta\in \RR/2\pi\ZZ,\,j\in\lbrace 1,\ldots,\kappa\rbrace,
\end{equation*}
and (by writing $g= k_j^{-1}nak_j\wtk_j(g)$, where $na\in NA$),
\begin{equation*}
\wtk(h_j g)=\wtk \big(k_j(k_j^{-1}nak_j\wtk_j(g))\big)=k_j\wtk_j(g) \qquad \forall g\in G,\,j\in\lbrace 1,\ldots,\kappa\rbrace.
\end{equation*}
Let $V\subset L^2(K)$ be the subspace of $\pi$-periodic functions in $L^2(K)$, i.e.
\begin{equation*}
V=\lbrace \vecv\in L^2(K)\,:\, \vecv(k_{\theta+\pi})=\vecv(\kthe),\,\forall \theta\in \RR/2\pi\ZZ\rbrace.
\end{equation*}
We now define a family of linear representations of $G$ on $V$, cf.\ \cite[Chapter 15]{Borel}, \cite[Chapter 7]{Knapp}: 
\begin{defn}\label{psr}
For each $s\in \CC$, the principal series representation $\scrP^s=(\pi^s,V)$ is defined through
\begin{equation}\label{psreq}
[\pi^s(g)\vecv](k):= \wty(kg)^s\vecv\big(\wtk(kg)\big)\qquad \forall g\in G,\, \vecv\in V,\,k\in K.
\end{equation}
Similarly, for each $s\in\CC$ and $j\in\lbrace 1,\ldots, \kappa\rbrace$, let $\scrP_j^s=(\pi_j^s,V)$ be defined by
\begin{equation}\label{psrjeq}
[\pi_j^s(g)\vecv](k):= \wty_j(kg)^s\vecv\big(\wtk_j(kg)\big)\qquad \forall g\in G,\, \vecv\in V,\,k\in K.
\end{equation}
\end{defn} 
\begin{remark}
The definition of principal series representations given here differs slightly from that given in \cite[Chapter 15]{Borel}. This is due to the fact we use the ``compact picture" of the principal series representations in order to keep the underlying vector space the same as we let $s$ vary, while \cite{Borel} uses the ``induced picture" (cf.\ \cite[Chapter 7, pp.\ 167-169]{Knapp}). The map $\vece_{2\upsilon}\mapsto$  $``\varphi_{P_j,2\upsilon,2s-1}"$ ($\vece_{2\upsilon}$ as in \eqref{stdbasis} below) provides an isomorphism between each $\scrP_j^s$ and $``H(0,2s-1)"$; $\varphi_{P_j,2\upsilon,2s-1}$ and $H(0,2s-1)$ being the notation used in \cite[Chapter 15]{Borel}.
\end{remark}
We have $\scrP^s\cong\scrP_j^s$ (isomorphism of $G$-representations) for all $s\in \CC$ and $j\in \lbrace 1,\ldots,\kappa\rbrace$. That these representations are indeed isomorphic is seen by noticing that $\pi_j^s(k_j^{-1} g k_j)=\pi^s(g)$: for fixed $k\in K$ and $g\in G$, let $kg=n_xa_y\kthe$ and $\vecv\in V$. Then
\begin{align*}
[\pi_j^s(k_j^{-1} g k_j)\vecv](k)=& \wty_j(k_j^{-1} kg k_j)^s \vecv\big(\wtk_j(kk_j^{-1} g k_j) \big)=\wty(kg)^s\vecv\big( \wtk_j(k_j^{-1}n_x a_yk_j\kthe)\big) \\&=\wty(kg)^s \vecv(\kthe)=\wty(kg)^s\vecv\big( \wtk(kg)\big)=[\pi^s(g)\vecv](k),
\end{align*} 
the third equality holding due to the fact that $k_j^{-1}n_xa_y k_j\in N_j A_j$. The operators $\pi_j^s(k_j)$ may thus be viewed as isomorphisms from $\scrP_j^s$ to $\scrP^s$: let $\scrK_j=\pi_j^s(k_j)$ (note that by \eqref{psrjeq}, $K$ acts by rotation on $V$; $\scrK_j$ is thus just a rotation by an element of $K$, and hence independent of $s\in\CC$). Then
\begin{equation}\label{Psjinterwtwine}
\pi_j^s(g)\scrK_j\vecv=\scrK_j\pi^s(g)\vecv\qquad \forall j\in\lbrace 1,\ldots,\kappa\rbrace,\,s\in \CC,\, \vecv\in V.
\end{equation}

Recall the following integration formula (cf.\ \cite[Lemma 15.4]{Borel}, \cite[(7.4), p. 170]{Knapp}):
\begin{equation}\label{Kint}
\int_K f(k)\,dk=\int_K f\big( \wtk_j(k g)\big)\wty_j (kg)\,dk\qquad \forall g\in G,\,j\in\lbrace 1,\ldots,\kappa\rbrace, f\in C(K),
\end{equation}
here $dk$ denotes the Haar probability measure on $K$. From this, we see that if $s\in \frac{1}{2}+i\RR$, then each $\scrP_j^s$ is a unitary representation with respect to the usual inner product on $L^2(K)$. 

We let $\lbrace \vece_{2\upsilon}\rbrace_{\upsilon\in \ZZ}$ denote the standard orthonormal basis of $V$:
\begin{equation}\label{stdbasis}
\vece_{2\upsilon}(\kthe)=e^{2\upsilon i \theta}\qquad \forall \upsilon\in\ZZ,\,\theta\in \RR/2\pi\ZZ,
\end{equation}  
and $V^{\infty}:=V\cap C^{\infty}(K)$.
\begin{prop}\label{intertwine}
For $s\in\CC\setminus\ZZ_{\leq 0}$, let $\scrI_s:V^{\infty}\rightarrow V^{\infty}$ be the linear operator defined by
\begin{equation*}
\scrI_s\vece_{2\upsilon}=\mathfrak{i}_{s,2\upsilon}\vece_{2\upsilon},
\end{equation*}
where 
\begin{equation*}
\mathfrak{i}_{s,2\upsilon}=\frac{(-1)^{\upsilon}\Gamma(s)^2}{\Gamma(s+\upsilon)\Gamma(s-\upsilon)}.
\end{equation*}
Then 
\begin{equation*}
\scrI_s\pi_j^s(g)\vecv=\pi_j^{1-s}(g)\scrI_s\vecv\;\;
\mathit{and}\; \;\scrI_{1-s}\scrI_s\vecv=\vecv\qquad\forall s\in\CC\setminus\ZZ,\,\vecv\in V^{\infty},\,g\in G.
\end{equation*}
Moreover, if $s\in\frac{1}{2}+i\RR$, then $\scrI_s$ is unitary.
\end{prop}
\begin{proof}
See e.g.\ \cite[Proposition 2.6.3]{Bump}.
\end{proof}

\subsection{Functions on $\GaG$ and $\GaH$}\label{GaGGaH} We now abuse notation slightly, and identify $\Gamma$ with its inverse image under the map from $G$ to $\PSL(2,\RR)=\SL(2,\RR)/\pm I$ given by $\smatr a b c d \mapsto \pm \smatr a b c d$; that is to say, we view $\Gamma$ as a lattice in $G$ that contains $-I$. Recall that $\HH\cong G/K$; there is thus an identification $\GaG/K\cong \GaH$. Let $\rho$ denote the linear $G$-action on $C(\GaG)$ given by right-translation: for $f\in C(\GaG)$, $h,\,g\in G$, $[\rho(g)f](\Gamma h)=f(\Gamma h g)$. There is a standard bijection between $C(\GaH)$ and $\rho(K)$-invariant elements in $C(\GaG)$ given by
\begin{equation*}
f(x+iy)\leftrightarrow f(\Gamma n_x a_y \kthe)\qquad \forall x\in \RR,\, y\in \RR_{>0},\,\theta\in\RR/2\pi\ZZ.
\end{equation*} 
Recall that $\big(\rho,L^2(\GaG)\big)$ is a unitary representation of $G$. We may decompose $L^2(\GaG)$ as the following (orthogonal) direct sum:
\begin{equation*}
L^2(\GaG)=\bigoplus_{\upsilon\in \ZZ} L^2(\GaG)_{2\upsilon},
\end{equation*} 
where
\begin{equation*}
L^2(\GaG)_{2\upsilon}=\lbrace f\in L^2(\GaG)\,:\, \rho(k)f=\vece_{2\upsilon}(k)f\quad\forall k\in K\rbrace.
\end{equation*}
Using the correspondence between $\rho(K)$-invariant functions on $\GaG$ and functions on $\GaH$ (as well as the decomposition $d\mu_G(n_x a_y \kthe)=d\mu(x+iy)\,d\kthe= \frac{ dx\,dy\,d\theta}{2\pi y^2}$ of the Haar measure $\mu_G$ on $G$) we may identify $L^2(\GaG)_0$ with $L^2(\GaH)$. We will now recall a similar bijection for $L^2(\GaG)_{2\upsilon}$ (for any $\upsilon\in\ZZ$). Let $L^2(\GaH,-2\upsilon)$ denote the following set of functions:
\begin{align*}
\left\lbrace f:\HH\rightarrow\CC\;:\;\int_{\scrF} |f|^2\,d\mu<\infty,\; f(\gamma\cdot z)=f(z)( \sfrac{c z + d}{|c z +d|})^{-2\upsilon}\quad\forall z\in \HH,\gamma=\smatr a b c d \in \Gamma\right\rbrace.
\end{align*}
Note that if $f_1,f_2\in L^2(\GaH,-2\upsilon)$, then $f_1\overline{f_2}$ is a $\Gamma$-invariant function in $L^1(\scrF)$. We thus have that $L^2(\GaH,-2\upsilon)$ is a Hilbert space with respect to the inner product
\begin{equation*}
\langle f_1,f_2\rangle_{L^2(\GaH,-2\upsilon)}:=\int_{\GaH} f_1\overline{f_2}\,d\mu=\int_{\scrF} f_1\overline{f_2}\,d\mu.
\end{equation*} 

Viewing an element $f\in L^2(\GaG)_{2\upsilon}$ as a left $\Gamma$-invariant function on $G$, we have
\begin{align*}
f(n_x a_y)=f(\gamma n_x a_y)=&f\big( n_{\Re(\gamma\cdot (x+iy))}a_{\Im(\gamma\cdot (x+iy))}\wtk(\gamma n_x a_y)\big)\\&=f\big( n_{\Re(\gamma\cdot (x+iy))}a_{\Im(\gamma\cdot (x+iy))}\big)\vece_{2\upsilon}\big(\wtk(\gamma n_x a_y)\big),
\end{align*}
i.e.
\begin{equation*}
f\big( n_{\Re(\gamma\cdot (x+iy))}a_{\Im(\gamma\cdot (x+iy))}\big)= f( n_x a_y)\vece_{-2\upsilon}\big(\wtk(\gamma n_x a_y)\big).
\end{equation*}
Note that if $\gamma= \smatr a b c d$ and $z=x+i y$, then $\vece_{2\upsilon}\big(\wtk(\gamma n_x a_y)\big)=( \frac{c z + d}{|c z +d|})^{2\upsilon}$. We thus define, for $f\in L^2(\GaG)_{2\upsilon}$, $\widetilde{f}\in L^2(\GaH;-2\upsilon)$ by $\widetilde{f}(x+iy):=f(n_xa_y)$.
\begin{lem}\label{GaH/GaGint}
For $f_1,f_2\in L^2(\GaG)_{2\upsilon}$, let $\widetilde{f_1},\widetilde{f_2}\in L^2(\GaH,-2\upsilon)$ be defined as above. Then
\begin{equation*}
\int_{\GaG}f\overline{f_2}\,d\mu_{\GaG} = \int_{\GaH} \widetilde{f_1}\overline{\widetilde{f_2}}\,d\mu,
\end{equation*}
$\mu_{\GaG}$ being the canonical projection of $\mu_G$ to $\GaG$.
\end{lem}
This shows that $f\mapsto \widetilde{f}$ is in fact an isomorphism (of Hilbert spaces) between $L^2(\GaG)_{2\upsilon}$ and $L^2(\GaH,-2\upsilon)$.

\subsection{Weighted Eisenstein series on $\GaH$ and the Eisenstein map}
We now recall the definition of \emph{weighted Eisenstein series} $E_j(z,s,2\upsilon)$; here $\upsilon\in\ZZ$, and $2\upsilon$ is referred to as a \emph{weight}. For $j\in\lbrace 1, \ldots,\kappa\rbrace$ and $\Re(s)>1$, we have (cf.\ \cite[pp.\ 63-69]{Kubota})
\begin{equation}
E_j(z,s,2\upsilon)=\sum_{\gamma\in\Gamma_{\eta_j}\backslash\Gamma} \Im(h_j\gamma\cdot z)^s\vece_{2\upsilon}\big(\wtk(h_j\gamma n_xa_y)\big)\qquad \forall z=x+iy\in\HH.
\end{equation}
Note that $E_j(z,s)=E_j(z,s,0)$. As is the case for weight zero Eisenstein series, each $E_j(z,s,2\upsilon)$ has a meromorphic continuation (in $s$), and the following Fourier decomposition at a cusp $\eta_k$ (cf.\ \cite[Lemma 2.6]{AndJensKron}; note that our $E_j(z,s,2\upsilon)$ corresponds to $``E_j(z,s,-2\upsilon)"$ in \cite{AndJensKron})
\begin{align}\label{Ej2uFour}
E_j(z,s,2\upsilon)=&\delta_{j,k}y_k^s+\mathfrak{i}_{s,2\upsilon} \varphi_{j,k}(s) y_k^{1-s}\\\notag&\quad+\sum_{\underset{m\neq 0}{m\in\ZZ}}\frac{(-1)^{\upsilon}\Gamma(s)}{2\Gamma(s-\upsilon\cdot\sgn(m))}\frac{\psi_{m,k}^{(j)}(s)}{\sqrt{|m|}}W_{-\upsilon\cdot\sgn(m),s-\frac{1}{2}}(4\pi |m|y_k)e(m x_k),
\end{align}
where $W_{\mu,\lambda}(r)$ is a Whittaker function (cf.\ \cite[9.22-2.23, pp.\ 1024-1028]{GR}), $\psi_{m,k}^{(j)}(s)$ is as in \eqref{EjkFourier}, and $\mathfrak{i}_{s,2\upsilon}$ is as in Proposition \ref{intertwine}. We have the following bound on the Eisenstein series of non-zero weight:
\begin{lem}\label{Ej2ubdd}
Let $\upsilon\in \ZZ\setminus\lbrace 0 \rbrace$. Then for $s\in\CC\setminus\scrE$ with $\Re(s)> \frac{1}{4}$ and $y_k\gg1$:
\begin{equation*}
E_j(z,s,2\upsilon)= \delta_{j,k}y_k^s+\mathfrak{i}_{s,2\upsilon} \varphi_{j,k}(s) y_k^{1-s}+O_s\big( \sqrt{1+\log|\upsilon|}\,|\upsilon|^{\max\lbrace \frac{1}{2}, \frac{3}{2}-2\Re(s) \rbrace} y_k^{-\min\lbrace\Re(s),\frac{1}{3}\rbrace}\big),
\end{equation*}
where the implied constant is uniform on compact subsets of $\lbrace s\in\CC\,:\, \Re(s)>\frac{1}{4}\rbrace\setminus\scrE$.
\end{lem}
The proof of Lemma \ref{Ej2ubdd} consists of combining \eqref{RanSelbd} with estimates for Whittaker functions, and is relegated to Appendix \ref{AppA} (see Appendix \ref{Ej2ubddproof}). We note also that for $\upsilon\neq 0$, $E_j(z,s,2\upsilon)$ is holomorphic at $s=1$, and satisfies the following bound:
\begin{lem}\label{Eisu1} For $\upsilon\in \ZZ\setminus\lbrace 0\rbrace$, $E_j(z,s,2\upsilon)$ is holomorphic at $s=1$, and for all small enough $\epsilon>0$, 
\begin{equation*}
E_j(z,s,2\upsilon)= \delta_{j,k}y_k^s+\mathfrak{i}_{s,2\upsilon} \varphi_{j,k}(s) y_k^{1-s}+O_{\epsilon}\big( |\upsilon| ^{1+2\epsilon}y_k^{2\epsilon-1}\big)\qquad \forall y_k\gg 1
\end{equation*}
for all $s\in\CC$ such that $0<|s-1|<\epsilon$, and
\begin{equation*}
E_j(z,1,2\upsilon)= \delta_{j,k}y_k-\frac{\mu(\GaH)^{-1}}{|\upsilon|}+O_{\epsilon}\big( |\upsilon| ^{\frac{1}{2}+2\epsilon}y_k^{2\epsilon-1}\big)\qquad \forall y_k\gg1.
\end{equation*}
\end{lem}
The proof of this is similar to that of Lemma \ref{Ej2ubdd}, and is also left to Appendix \ref{AppA} (see Appendix \ref{Eisu1proof}).

Turning to $\GaG$, for $\vecv\in V^{\infty}$, $s\in\CC$, $\Re(s)>1$, and $j\in\lbrace 1,\ldots,\kappa\rbrace$, let $\EE_j^s(\cdot,\vecv):\GaG\rightarrow \CC$ be defined thus
\begin{equation}\label{EGAGdef}
\EE_j^s(\Gamma g,\vecv):=\sum_{\gamma\in \Gamma_{\eta_j}\backslash\Gamma} [\pi^s_j(\gamma g)\vecv](e).
\end{equation} 
That $\EE_j^s(\Gamma g,\vecv)$ is well-defined follows from the fact that if $\gamma_0\in \Gamma_{\eta_j}$, then $h_j \gamma_0 h_j^{-1}= \smatr{1}{x_0}{0}{1}\in N$, hence for any $\gamma\in\Gamma$,
\begin{align*}
 [\pi^s_j(\gamma_0\gamma g)\vecv](e)=&\wty_j(\gamma_0\gamma g)^s\vecv\big( \wtk_j(\gamma_0\gamma g)\big)\\&=\wty\big(k_j h_j^{-1}\smatr{1}{x_0}{0}{1} h_j \gamma g \big)^s \vecv\big( \wtk_j(h_j^{-1}\smatr{1}{x_0}{0}{1} h_j\gamma g)\big)\\&=\wty(h_j)^{-s}\wty\big( h_j \gamma g \big)^s \vecv\big( k_j^{-1}\wtk( h_j\gamma g)\big)
 \\&=\wty\big( k_j \gamma g \big)^s \vecv\big( \wtk_j( \gamma g)\big)
 =\wty_j\big( \gamma g \big)^s \vecv\big( \wtk_j( \gamma g)\big)
 \\&=[\pi^s_j(\gamma g)\vecv](e),
\end{align*}
as well as that the sum is absolutely convergent: for $z=x+iy\in \HH$ and $\theta\in \RR/2\pi\ZZ$, letting $g=n_xa_y\kthe$, we have
\begin{align*}
& \sum_{\gamma\in \Gamma_{\eta_j}\backslash\Gamma} \left| [\pi^s_j(\gamma g)\vecv](e) \right|= \sum_{\gamma\in \Gamma_{\eta_j}\backslash\Gamma} \wty_j(\gamma n_x a_y)^{\Re(s)}\left|\vecv\big(\wtk_j(\gamma n_x a_y)\kthe \big)\right|\\&\leq  \|\vecv\|_{L^{\infty}(K)} \wty(h_j)^{-\Re(s)} E_j(z,\Re(s)).
\end{align*}
\begin{remark}
The functions $\EE_j^s(\cdot,\vece_{2\upsilon})$ are the same as the Eisenstein series $``E(P_j,2\upsilon,2s-1)"$ defined in \cite[Chapters 10 and 11]{Borel}. As in \cite{Borel}, it is fairly common to consider $\EE_j^s(\cdot,\vecv)$ only for $\vecv$ which are \emph{$K$-finite}, i.e. $\vecv$ which are finite linear combinations of the $\vece_{2\upsilon}$. However, we will require $\EE_j^s(\cdot,\vecv)$ for more general $\vecv\in V^{\infty}$.
\end{remark}

From \eqref{EGAGdef}, we see that $\EE_j^s$ may be viewed as a linear operator from $V^{\infty}$ to $C^{\infty}(\GaG)$, which we call an \emph{Eisenstein map}. It is also apparent from \eqref{EGAGdef} that this map intertwines the linear $G$-actions $\rho$ (on $C(\GaG)$) and $\pi^s_j$ (on $V^{\infty}$):
\begin{equation}\label{Intertwine}
\rho(g)\EE_j^s(\Gamma h,\vecv)=\EE_j^s(\Gamma hg,\vecv)=\EE_j^s(\Gamma h,\pi_j^s(g)\vecv)\qquad \forall g,h\in G,\,\vecv\in V^{\infty},\,\Re(s)>1.
\end{equation}
This property will be key for what follows. Given $\vecv\in V^{\infty}$, we have the following Fourier decomposition of $\vecv$:
\begin{equation*}
\vecv=\sum_{\upsilon\in\ZZ} a_{2\upsilon}\vece_{2\upsilon}
\end{equation*}
where $a_{2\upsilon}=\langle \vecv,\vece_{2\upsilon}\rangle_{L^2(K)}=\int_{K} \vecv(k)\overline{\vece_{2\upsilon}(k)}\,dk$. The bound $|a_{2\upsilon}|\ll_{\vecv,n}(1+|\upsilon|)^{-n}\;(\forall n\in\NN)$ allows us to expand $\vecv$ in the basis $\lbrace \vece_{2\upsilon}\rbrace$ and exchange the sum over $\upsilon\in\ZZ$ with the sum over $\gamma\in \Gamma_{\eta_j}\backslash \Gamma$ in \eqref{EGAGdef}, i.e. 
\begin{align}\label{EismapKFour}
\EE_j^s(\Gamma g,\vecv)&=\sum_{\gamma\in \Gamma_{\eta_j}\backslash\Gamma} [\pi^s_j(\gamma g)\vecv](e)=\sum_{\gamma\in \Gamma_{\eta_j}\backslash\Gamma}\left[ \pi_j^s(g)\left( \sum_{\upsilon\in\ZZ}a_{2\upsilon}\vece_{2\upsilon}\right)\right](e)\\\notag&=\sum_{\gamma\in \Gamma_{\eta_j}\backslash\Gamma} \sum_{\upsilon\in\ZZ}a_{2\upsilon}\left[ \pi_j^s(g)\vece_{2\upsilon}\right](e)= \sum_{\upsilon\in\ZZ}a_{2\upsilon}\sum_{\gamma\in \Gamma_{\eta_j}\backslash\Gamma}\left[ \pi_j^s(g)\vece_{2\upsilon}\right](e)\\\notag&=\sum_{\upsilon\in\ZZ}a_{2\upsilon} \EE_j^s(\Gamma g,\vece_{2\upsilon}).
\end{align} 
Turning our attention to the functions $\EE_j^s(\cdot,\vece_{2\upsilon})$, using \eqref{psrjeq} (as well as the relations for $\wty_j$ and $\wtk_j$) we obtain
\begin{align}\label{Eismapexpr}
\EE_j^s(\Gamma n_xa_y\kthe,\vece_{2\upsilon})&=\sum_{\gamma\in \Gamma_{\eta_j}\backslash\Gamma} [\pi^s_j(\gamma n_x a_y\kthe)\vece_{2\upsilon}](e)=\sum_{\gamma\in \Gamma_{\eta_j}\backslash\Gamma} \wty_j(\gamma n_x a_y\kthe)^s\vece_{2\upsilon}\big(\wtk_j(\gamma n_x a_y \kthe) \big)\\\notag&=\sum_{\gamma\in \Gamma_{\eta_j}\backslash\Gamma} \wty_j\big(\gamma n_xa_y\big)^s\vece_{2\upsilon}\big(\wtk_j(\gamma n_x a_y ) \big)\vece_{2\upsilon}(\kthe)\\\notag&=\vece_{2\upsilon}(\kthe)\wty(h_j)^{-s} \sum_{\gamma\in \Gamma_{\eta_j}\backslash\Gamma} \Im\big(h_j\gamma\cdot(x+iy)\big)^s\vece_{2\upsilon}\big( k_j^{-1}\wtk(h_j\gamma n_x a_y)\big)\\\notag&=\frac{1}{\wty(h_j)^s \vece_{2\upsilon}(k_j)}E_j(x+i y ,s, 2\upsilon)\vece_{2\upsilon}(\kthe)\qquad\forall x+iy\in\HH,\,\theta\in \RR/2\pi\ZZ.
\end{align}
The meromorphic continuation of $E_j(z,s,2\upsilon)$ thus allows us to meromorphically continue $\EE_j^s(\Gamma g,s,\vece_{2\upsilon})$ to all of $\CC$. Note that $\EE_j^s(\Gamma gk,s,\vece_{2\upsilon})=\vece_{2\upsilon}(k) \EE_j^s(\Gamma g,s,\vece_{2\upsilon})$ for all $g\in G$ and $k\in K$.
\begin{prop}\label{Eisucont}
For any fixed $\vecv\in V^{\infty}$, $j\in\lbrace 1,\ldots,\kappa\rbrace$, and $g\in G$, $\EE_j^s(\Gamma g,\vecv)$ has a meromorphic continuation to all of $\lbrace z\in\CC\,:\,\Re(z)> \frac{1}{4}\rbrace$, with the set of poles contained in $\scrE$. Furthermore, this extended function satisfies
\begin{equation}\label{Eismapintertwine}
\rho(g)\EE_j^s(\cdot , \vecv)=\EE^s_j(\cdot,\pi_j^s(g)\vecv)\qquad\forall g\in G
\end{equation}
and  $\EE_j^s(\cdot,\vecv)\in C^{\infty}(\GaG)$ for all $s\not\in \scrE$.
\end{prop}
\begin{proof}
Given $\vecv\in V^{\infty}$, we write $\vecv=\sum_{\upsilon\in \ZZ} a_{2\upsilon}\vece_{\upsilon}$ and 
\begin{equation*}
\pi^s_j(g)\vecv=\sum_{\upsilon\in\ZZ} a_{j,2\upsilon}^s(g)\vece_{2\upsilon}\qquad \forall\ g\in G,s\in \CC.
\end{equation*}
(we have $a_{2\upsilon}=a_{j,2\upsilon}^s(e)$ for all $j\in\lbrace 1,\ldots,\kappa\rbrace$ and $s\in \CC$). Note that
\begin{align*}
a_{j,2\upsilon}^s(g)=\langle \pi_j^s(g)\vecv,\vece_{2\upsilon}\rangle_{L^2(K)}=\int_K \wty_j(kg)^s\vecv\big( \wtk_j(k g)\big)\overline{\vece_{2\upsilon}(k)}\,dk.
\end{align*}
Since the maps $(g,k)\mapsto \wty_j(kg)$ and $(g,k)\mapsto \wtk_j(kg)$ are smooth (from $G\times K$ to $\RR_{>0}$ and $K$, respectively), the map $s\mapsto a_{j,2\upsilon}^s(g)$ is holomorphic (for every $g\in G$). Moreover, combining the smoothness of these maps with the fact that $\vecv\in V^{\infty}$, for every pair of compact subsets $U_1\subset G$, $U_2\subset \CC$ and $n\in\NN$, there exists a constant $C=C(U_1,U_2,n)$ such that
\begin{equation}\label{aFourbdd}
|a_{j,2\upsilon}^s(g)|\leq \frac{C}{1+|\upsilon|^{n}}\qquad \forall g\in U_1,\,s\in U_2,\,\upsilon\in \ZZ.
\end{equation}
For $\Re(s)>1$ and $g\in G$ we have, using \eqref{Intertwine}, \eqref{EismapKFour}, and \eqref{Eismapexpr},
\begin{align*}
\rho(g)\EE_j^s(\Gamma n_x a_y\kthe, \vecv)&=\EE_j^s(\Gamma n_x a_y\kthe ,\pi_j^s(g) \vecv)\\&=\wty(h_j)^{-s}\sum_{\upsilon\in \ZZ} \frac{a_{j,2\upsilon}^s(g)}{ \vece_{2\upsilon}(k_j)}E_j(x+i y ,s, 2\upsilon)\vece_{2\upsilon}(\kthe).
\end{align*}
Combining Lemma \ref{Ej2ubdd} with the bound \eqref{aFourbdd} allows us to conclude that this sum converges absolutely and uniformly for $(g,s,\Gamma n_x a_y \kthe)$ in compact subsets $U_1\times U_2\times U_3\subset G\times (\lbrace z\in\CC\,:\,\Re(z)> \frac{1}{4}\rbrace\setminus\scrE)\times \GaG$. This suffices to meromorphically continue $\EE^s_j(\Gamma h ,\vecv)$. (Letting $s_0\in \scrE$ have $\Re(s_0)>\frac{1}{4}$ and $m_0$ be the order of the pole for $E_j(z,s,0)$, we have that for all $\upsilon\in\ZZ$, $(s-s_0)^{m_0}E_j(z,s,2\upsilon)$ is holomorphic at $s_0$. Now using the maximum principle and Lemma \ref{Ej2ubdd},  we have that $(s-s_0)^{m_0}\EE_j^s(\Gamma n_x a_y \kthe,\vecv)=\wty(h_j)^{-s}\sum_{\upsilon\in \ZZ} \frac{a_{j,2\upsilon}^s(g)}{ \vece_{2\upsilon}(k_j)}(s-s_0)^{m_0}E_j(x+i y ,s, 2\upsilon)\vece_{2\upsilon}(\kthe)$ remains bounded as $s\rightarrow s_0$; the points of $\scrE$ can indeed only be poles of $\EE_j^s(\Gamma g ,\vecv)$, and not essential sigularities.) The relation \eqref{Eismapintertwine} then also holds by meromorphic continuation.
\end{proof}
Let $V_0^{\infty}=\lbrace \vecv\in V^{\infty}\,:\, \int_K \vecv\,dk=0\rbrace$. Using Lemma \ref{Eisu1}, we can also define $\EE^1(\cdot,\vecv)$ for $\vecv\in V_0^{\infty}$:
\begin{prop}
Let $\vecv\in V_0^{\infty}$ and $j\in\lbrace 1,\ldots,\kappa\rbrace$. Then $\EE_j^s(\Gamma g,\vecv)$ is holomorphic at $s=1$, and
\begin{equation}\label{Eismapintertwine1}
\rho(g)\EE_j^1(\cdot , \vecv)=\EE^1_j(\cdot,\pi_j^1(g)\vecv)+\frac{1}{\wty(h_j)\mu(\GaH)}\left.\frac{d}{ds}\right|_{s=1} \langle \pi_j^s(g)\vecv,\vece_0\rangle_{L^2(K)}\qquad\forall g\in G.
\end{equation}
\end{prop}
\begin{remark}
Note that by \eqref{Kint}, if $\vecv\in V_0^{\infty}$, then also $\pi_j^1(g)\vecv\in V_0^{\infty}$ for all $g\in G$.
\end{remark}
\begin{proof}
As in the proof of Proposition \ref{Eisucont}, we have $\vecv=\sum_{\underset{\upsilon\neq 0}{\upsilon\in\ZZ}}a_{2\upsilon}\vece_{2\upsilon}$, and
\begin{align*}
\EE^s_j(\Gamma n_x a_y \kthe,\vecv)= \wty(h_j)^{-s}\sum_{\underset{\upsilon\neq 0}{\upsilon\in\ZZ}}\frac{a_{2\upsilon}}{ \vece_{2\upsilon}(k_j)}E_j(x+i y ,s, 2\upsilon)\vece_{2\upsilon}(\kthe).
\end{align*}
The bound $|a_{2\upsilon}|\ll_{m,\vecv} (1+|\upsilon|)^{-m}$ (for all $m\in \NN$), combined with the bounds of Lemma \ref{Eisu1} shows (again, as in the proof of Proposition \ref{Eisucont}) that $\EE^s_j(\Gamma n_x a_y \kthe,\vecv)$ is indeed holomorphic at $s=1$. Now letting $a_{j,2\upsilon}^s(g)=\langle \pi_j^s(g)\vecv,\vece_{2\upsilon}\rangle_{L^2(K)}$, we note that $a_{j,2\upsilon}^1(g)=0$ for all $g\in G$. For $s\neq 1$ in a small neighbourhood of one as in Lemma \ref{Eisu1}, we have
\begin{align*}
\rho(g)\EE^s_j(\Gamma n_x a_y \kthe,\vecv)=\EE^s_j(\Gamma n_x a_y \kthe,\pi_j^s(g)\vecv)= \wty(h_j)^{-s}\sum_{\upsilon\in\ZZ}\frac{a_{j,2\upsilon}^s(g)}{ \vece_{2\upsilon}(k_j)}E_j(x+i y ,s, 2\upsilon)\vece_{2\upsilon}(\kthe),
\end{align*}
hence by Lemma \ref{Eis1pole},
\begin{align*}
\rho(g)\EE^s_j(\Gamma n_x a_y \kthe,\vecv)=& a_{j,0}^s(g)\wty(h_j)^{-s}\left(\frac{\mu(\GaH)^{-1}}{s-1}+\widetilde{E}_j(z,s) \right)\\&+ \wty(h_j)^{-s}\sum_{\underset{\upsilon\neq 0}{\upsilon\in\ZZ}}\frac{a_{j,2\upsilon}^s(g)}{ \vece_{2\upsilon}(k_j)}E_j(x+i y ,s, 2\upsilon)\vece_{2\upsilon}(\kthe).
\end{align*}
Letting $s\rightarrow1$ concludes the proof.
\end{proof}

For each $\beta\in\RR_{\geq 0}$ we define a \emph{Sobolev norm} $\scrS_{\beta}(\cdot)$ on $V^{\infty}$ by
\begin{equation}
\scrS_{\beta}(\vecv):=\left(\sum_{\upsilon\in\ZZ} (1+|\upsilon|^{\beta})^2|\langle\vecv,\vece_{2\upsilon}\rangle_{L^2(K)}|^2\right)^{1/2}.
\end{equation}
\begin{lem}\label{Eismapbd}
Let $z=x+iy\in \HH$, $\theta\in \RR/2\pi\ZZ$, $\upsilon\in\ZZ$, and $\vecv\in V^{\infty}$. Then for all $j,k\in\lbrace 1,\ldots,\kappa\rbrace$, $s\in\lbrace z\in\CC\,:\,\Re(z)\geq \frac{1}{3}\rbrace\setminus \scrE$, $ y_k\gg 1$, and any $\epsilon>0$, we have
\begin{align*}
\EE_j^s(\Gamma n_x a_y \kthe, \vecv)= \frac{1}{\wty(h_j)^s }&\left(\delta_{j,k} \vecv(\kthe k_j^{-1})y_k^s +\varphi_{j,k}(s)y_k^{1-s}[\scrI_s\vecv](\kthe k_j^{-1}) \right)\\&\qquad+O_{\epsilon,s}\big( \scrS_{\frac{1}{2}+{\max\lbrace \frac{1}{2}, \frac{3}{2}-2\Re(s) \rbrace}+\epsilon}(\vecv) y_k^{-\frac{1}{3}}\big),
\end{align*}
where the implied constant is uniform over any compact subset of $\lbrace z\in\CC\,:\,\Re(z)> \frac{1}{3}\rbrace\setminus\scrE$.
\end{lem}
\begin{proof}
Letting $\vecv=\sum_{\upsilon\in\ZZ} a_{2\upsilon} \vece_{2\upsilon}$, we have
\begin{equation*}
\EE^s_j (\Gamma n_x a_y \kthe,\vecv)=\sum_{\upsilon\in \ZZ}\frac{a_{2\upsilon}}{\wty(h_j)^s \vece_{2\upsilon}(k_j)}E_j(x+i y ,s, 2\upsilon)\vece_{2\upsilon}(\kthe).
\end{equation*}
Now defining $\upsilon^*:=(1+|\upsilon|^{\max\lbrace \frac{1}{2}, \frac{3}{2}-2\Re(s) \rbrace})\sqrt{1+\log(\max\lbrace1,|\upsilon|\rbrace)} $ and applying Lemma \ref{Ej2ubdd} yields 
\begin{align*}
\EE^s_j &(\Gamma n_x a_y \kthe,\vecv)=\sum_{\upsilon\in \ZZ}\frac{a_{2\upsilon}}{\wty(h_j)^s \vece_{2\upsilon}(k_j)}\Big( \delta_{j,k}y_k^s+\mathfrak{i}_{s,2\upsilon} \varphi_{j,k}(s) y_k^{1-s}+O_s\big( \upsilon^*y_k^{-\frac{1}{3}}\big)\Big)\vece_{2\upsilon}(\kthe)\\&=\frac{1}{\wty_j(h_j)^s}\bigg( \delta_{j,k}y_k^s\sum_{\upsilon\in \ZZ}a_{2\upsilon}\vece_{2\upsilon}(\kthe k_j^{-1})+\varphi_{j,k}(s) y_k^{1-s} \sum_{\upsilon\in \ZZ}\mathfrak{i}_{s,2\upsilon}a_{2\upsilon}\vece_{2\upsilon}(\kthe k_j^{-1})\\&\qquad\qquad\qquad\qquad\qquad\qquad\qquad\qquad\qquad\qquad\qquad\qquad+\sum_{\upsilon\in \ZZ}|a_{2\upsilon}|O_s\big(\upsilon^* y_k^{-\frac{1}{3}}\big)\bigg)\\&=\frac{1}{\wty_j(h_j)^s}\left( \delta_{j,k} y_k^{s}\vecv(\kthe  k_j^{-1})+\varphi_{j,k}(s)y_k^{1-s}[\scrI_s\vecv](\kthe k_j^{-1}) \right) +O_s\left(y_k^{-\frac{1}{3}}\left( \sum_{\upsilon\in\ZZ}  |a_{2\upsilon}|\upsilon^*\right) \right).
\end{align*}
The Cauchy-Schwarz inequality then gives
\begin{equation}
\sum_{\upsilon\in\ZZ} |a_{2\upsilon}|\upsilon^*\ll_{s,\epsilon} \scrS_{\frac{1}{2}+{\max\lbrace \frac{1}{2}, \frac{3}{2}-2\Re(s) \rbrace}+\epsilon}(\vecv).
\end{equation}
\end{proof}

\subsection{The direct integral decomposition of $\big(\rho,L^2(\GaG)\big)$ }
We now recall the \emph{direct integral decomposition} of the unitary representation $\big(\rho,L^2(\GaG)\big)$ into irreducible unitary representations, and how this relates to Eisenstein series. Given a unitary representation $(\pi,\scrH)$ of $G$, one can find a locally compact Hausdorff space $\mathsf{Z}$ equipped with a positive Radon measure $\lambda$ such that 
\begin{equation*}
(\pi,\mathcal{H})\cong \left(\int_{\mathsf{Z}}^{\oplus} \mathcal{\pi}_{\zeta}\,d\lambda(\zeta), \int_{\mathsf{Z}}^{\oplus} \mathcal{H}_{\zeta}\,d\lambda(\zeta) \right),
\end{equation*}
where each $(\pi_{\zeta},\scrH_{\zeta})$ is an irreducible unitary representation of $G$ (cf.\ \cite[Corollary 14.9.5]{Wallach}). For the particular case $(\pi,\scrH)=\big(\rho,L^2(\GaG)\big)$, we can say more (cf.\ \cite{Borel}, \cite{Lang}), namely 
\begin{equation}\label{L2decomp}
\big(\rho,L^2(\GaG)\big)\cong \bigoplus_{i=1}^{\infty}(\pi_i,\scrH_i)\oplus \bigoplus_{j=1}^{\kappa}\int_{\RR\geq 0}^{\oplus} \scrP_j^{1/2+it}\,dt,
\end{equation}
where all the direct sums are orthogonal and each $(\pi_i,\scrH_i)$ may be realised as $ (\pi_i,\scrH_i)\cong \big(\rho,\overline{\mathrm{span}\lbrace \rho(g)f_i\,:\,g\in G\rbrace}\big)$ for some $f_i\in L^2(\GaG)\cap C^{\infty}(\GaG)$. We may thus write $\scrH_i=\overline{\mathrm{span}\lbrace \rho(g)f_i\,:\,g\in G\rbrace}$, and let
\begin{equation*}
L^2(\GaG)_{\mathit{disc}}=\bigoplus_{i=1}^{\infty} \scrH_i.
\end{equation*}
The subrepresentation $\big(\rho, L^2(\GaG)_{\mathit{disc}}\big)= \bigoplus_{i=1}^{\infty}(\rho,\scrH_i)$ is called the \emph{discrete component} of the representation $(\rho,L^2(\GaG))$. Recall that one can further decompose $\big(\rho, L^2(\GaG)_{\mathit{disc}}\big)$ into the \emph{cuspidal} and \emph{residual} parts (we will not need these, however). Letting $L^2(\GaG)_{\mathit{cont}}=L^2(\GaG)_{\mathit{disc}}^{\perp}$, we have
\begin{equation*}
\big( \rho, L^2(\GaG)_{\mathit{cont}}\big)\cong \bigoplus_{j=1}^{\kappa}\int_{\RR\geq 0}^{\oplus} \scrP_j^{1/2+it}\,dt,
\end{equation*}
and call $\big( \rho, L^2(\GaG)_{\mathit{cont}}\big)$ the \emph{continuous component} of $\big(\rho,L^2(\GaG)\big)$. Recall that each direct integral of principal series $\int_{\RR\geq 0}^{\oplus} \scrP_j^{1/2+it}\,dt$ is a unitary representation $(\pi,\scrH)$ of $G$, where the underlying Hilbert space $\scrH$ is the vector space of functions $\vecv:\RR_{\geq 0}\rightarrow V$ such that $t\mapsto\langle \vecv(t),\vece_{2\upsilon}\rangle_{L^2(K)}$ is measurable for every $\upsilon\in\ZZ$ and $\int_0^{\infty}\|\vecv(t)\|^2\,dt<\infty$, and the action $\pi$ of $G$ on $\scrH$ is given by $[\pi(g)\vecv](t)=\pi_j^{1/2+it}(g)(\vecv(t))$ for a.e. $t\in \RR_{\geq 0}$ and all $g\in G$. As is standard, we write $\int_0^{\infty} \vecv(t)\,dt$ for the function $\vecv\in \scrH$

In the following theorem, we give an explicit isomorphism from $(\rho,L^2(\GaG))$ to the right-hand side of \eqref{L2decomp}: 
\begin{thm}\label{fdecomp}
The mapping from $C_c^{\infty}(\GaG)$ to $\bigoplus_{i=1}^{\infty}(\pi_i,\scrH_i)\oplus \bigoplus_{j=1}^{\kappa}\int_{\RR\geq 0}^{\oplus} \scrP_j^{1/2+it}\,dt$ given by
\begin{equation*}
f\mapsto\sum_{i=1}^{\infty} \Proj_i(f) + \sum_{j=1}^{\kappa} \int_0^{\infty}\vecv_{f,j}(t)\,dt,
\end{equation*}
where $\Proj_i(f) $ is the orthogonal projection of $f$ onto $\scrH_i\subset L^2(\GaG)$ and
\begin{equation}\label{vecvfjtdef}
\vecv_{f,j}(t)=\sfrac{\wty(h_j)^{1/2-it}}{\sqrt{2\pi}}\sum_{\upsilon\in \ZZ} \langle f, \EE_j^{1/2+it}(\cdot,\vece_{2\upsilon})\rangle_{L^2(\GaG)}\vece_{2\upsilon},
\end{equation}
may be extended to an isomorphism of unitary representations between $(\rho,L^2(\GaG))$ and $\bigoplus_{i=1}^{\infty}(\pi_i,\scrH_i)\oplus \bigoplus_{j=1}^{\kappa}\int_{\RR\geq 0}^{\oplus} \scrP_j^{1/2+it}\,dt$.
\end{thm}
\begin{remark}
Note that $\EE_j^{1/2+it}(\cdot,\vece_{2\upsilon})\not\in L^2(\GaG)$; hence $\langle f, \EE_j^{1/2+it}(\cdot,\vece_{2\upsilon})\rangle_{L^2(\GaG)}$ is undefined. However, since $f\in C_c^{\infty}(\GaG)$ in the above theorem,  $\int_{\GaG} |f\,\overline{\EE_j^{1/2+it}(\cdot,\vece_{2\upsilon})}|\,d\mu_{\GaG}<\infty$ for all $j,\,\upsilon$, and $t$. We thus define $ \langle f ,\EE_j^{1/2+it}(\cdot,\vece_{2\upsilon}) \rangle_{L^2(\GaG)}$ in \eqref{vecvfjtdef} to be $ \int_{\GaG}f\,\overline{\EE_j^{1/2+it}(\cdot,\vece_{2\upsilon})}\,d\mu_{\GaG}$. 
\end{remark}

\begin{remark}
The above result is certainly standard (cf.\ \cite[Chapters 16-17]{Borel}). However, we have not been able to find a formulation of the theorem as stated here, so we give the following fairly detailed proof:
\end{remark}

\begin{proof}
We need to prove that the map $f\mapsto\sum_{i=1}^{\infty} \Proj_i(f) + \sum_{j=1}^{\kappa} \int_0^{\infty}\vecv_{f,j}(t)\,dt$ is (the restriction to $C_c^{\infty}(\GaG)$ of) an isomorphism of unitary representations, i.e that $\Proj_i\big(\rho(g)f\big)=\rho(g)\Proj_i(f)$ and $\vecv_{\rho(g)f,j}(t)=\pi_j^{1/2+it}(g)(\vecv_{f,j}(t))$ for all $g\in G$, as well as
\begin{equation*}
\|f\|_{L^2(\GaG)}^2=\sum_{i=1}^{\infty}\|\Proj_i(f)\|^2+\sum_{j=1}^{\kappa}\int_0^{\infty} \|\vecv_{f,j}(t)\|_{L^2(K)}^2\,dt.
\end{equation*}
Starting with the intertwining property: each $\scrH_i$ is a closed, $\rho(G)$-invariant subspace of $L^2(\GaG)$, hence $\Proj_i\big(\rho(g)f\big)=\rho(g)\Proj_i(f)$ for all $f\in L^2(\GaG)$. We now consider $\vecv_{f,j}(t)$; let $c_{j,t}=\frac{\wty(h_j)^{1/2-it}}{\sqrt{2\pi}}$. Then
\begin{align*}
c_{j,t}^{-1}\vecv_{\rho(g)f,j}(t)=&\sum_{\upsilon\in\ZZ} \langle \rho(g)f, \EE^{1/2+it}_{j}(\cdot,\vece_{2\upsilon})\rangle_{L^2(\GaG)}\vece_{2\upsilon}=\sum_{\upsilon\in\ZZ} \langle f, \rho(g^{-1})\EE^{1/2+it}_{j}(\cdot,\vece_{2\upsilon})\rangle_{L^2(\GaG)}\vece_{2\upsilon}\\&=\sum_{\upsilon\in\ZZ} \langle f, \EE^{1/2+it}_{j}(\cdot,\pi_j^{1/2+it}(g^{-1})\vece_{2\upsilon})\rangle_{L^2(\GaG)}\vece_{2\upsilon}\\&=\sum_{\upsilon\in\ZZ} \left\langle f,\sum_{\sigma\in\ZZ} \langle \pi_j^{1/2+it}(g^{-1})\vece_{2\upsilon} ,\vece_{2\sigma}\rangle_{L^2(K)} \EE^{1/2+it}_{j}(\cdot,\vece_{2\sigma})\right\rangle_{L^2(\GaG)}\!\!\!\!\vece_{2\upsilon}
\end{align*}
The sum $\sum_{\sigma\in\ZZ} \langle \pi_j^s(g^{-1})\vece_{2\upsilon} ,\vece_{2\sigma}\rangle_{L^2(K)} \EE^{1/2+it}_{j}(\cdot,\vece_{2\sigma})$ converges absolutely and uniformly on compact subsets of $\GaG$, allowing us to write
\begin{align*}
c_{j,t}^{-1}\vecv_{\rho(g)f,j}(t)=&\sum_{\upsilon\in\ZZ}\sum_{\sigma\in\ZZ}  \langle \pi_j^{1/2+it}(g)\vece_{2\sigma} ,\vece_{2\upsilon}\rangle_{L^2(K)} \langle f,\EE^{1/2+it}_{j}(\cdot,\vece_{2\sigma})\rangle_{L^2(\GaG)}\vece_{2\upsilon}\\&=\sum_{\upsilon\in\ZZ} \left\langle \pi_j^{1/2+it}(g)\left(\sum_{\sigma\in \ZZ}  \langle f,\EE^{1/2+it}_{j}(\cdot,\vece_{2\sigma})\rangle_{L^2(\GaG)}\vece_{2\sigma}\right),\vece_{2\upsilon} \right\rangle_{L^2(K)}\vece_{2\upsilon}\\&=\sum_{\upsilon\in\ZZ}  \langle\pi_j^{1/2+it}(g) c_{j,t}^{-1}\vecv_{f,j}(t),\vece_{2\upsilon}\rangle\vece_{2\upsilon}=c_{j,t}^{-1}\pi_j^s(g)\vecv_{f,j}(t).
\end{align*}
It remains to confirm that the mapping is an isometry.
Recall that we have $L^2(\GaG)=\bigoplus_{\upsilon\in \ZZ} L^2(\GaG)_{2\upsilon}$, and that this is an orthogonal direct sum. For each $\upsilon\in \ZZ$ and $\psi\in L^2(\GaG)$, define
\begin{equation*}
\psi_{2\upsilon}(\Gamma g):=\int_K \psi(\Gamma g k)\overline{\vece_{2\upsilon}(k)}\,dk.
\end{equation*}
The mapping $\psi\mapsto \psi_{2\upsilon}$ is the orthogonal projection of $L^2(\GaG)$ onto $L^2(\GaG)_{2\upsilon}$; we have $\psi_{2\upsilon}\in L^2(\GaG)_{2\upsilon}$, and $\|\psi\|_{L^2(\GaG)}^2=\sum_{\upsilon\in\ZZ} \|\psi_{2\upsilon}\|_{L^2(\GaG)}^2$. Note that since $K$ is compact, for $f\in C_c^{\infty}(\GaG)$, we also have $f_{2\upsilon}\in C_c^{\infty}(\GaG)$; this allows us to now restrict our attention to $f\in L^2(\GaG)_{2\upsilon}\cap C_c^{\infty}(\GaG) $.

For each subrepresentation $(\rho,\scrH_i)$, either $L^2(\GaG)_{2\upsilon}\cap\scrH_i=\lbrace 0\rbrace$, or $L^2(\GaG)_{2\upsilon}\cap\scrH_i=\CC \phi_{i,2\upsilon}$ for some $\phi_{i,2\upsilon}\in\scrH_i$ with $\|\phi_{i,2\upsilon}\|_{L^2(\GaG)}=1$. By abusing notation slightly, we also let $\phi_{i,2\upsilon}=0$ if $L^2(\GaG)_{2\upsilon}\cap\scrH_i=\lbrace 0 \rbrace$. Then $\lbrace \phi_{i,2\upsilon}\rbrace_{\upsilon\in\ZZ}$ (minus the zeroes) is an orthonormal basis of $\scrH_i$, and $\rho(k)\phi_{i,2\upsilon}=\phi_{i,2\upsilon} \vece_{2\upsilon}(k)$ for all $k\in K$. Using the notation of Section \ref{GaGGaH}, let $\widetilde{f},\widetilde{\phi}_{i,2\upsilon}\in L^2(\GaH,-2\upsilon)$ (i.e.\ $\widetilde{f}(x+iy)=f(n_xa_y)$, and similarly for $\widetilde{\phi}_{i,2\upsilon}$). By \cite[p.\ 317, Proposition 5.3, p.\ 414, lines 12-16]{Hejhal}, we obtain
\begin{equation*}
\widetilde{f}(z)=\sum_{i=1}^{\infty} \langle \widetilde{f},\widetilde{\phi}_{i,2\upsilon}\rangle_{L^2(\GaH,-2\upsilon)} \widetilde{\phi}_{i,2\upsilon}(z)+\sum_{j=1}^{\kappa}\int_0^{\infty} g_j(t) E_j(z,\sfrac{1}{2}+it,2\upsilon)\,dt,
\end{equation*}
where $g_j(t)=\frac{1}{2\pi}\int_{\scrF}\widetilde{f}(z)\overline{E_j(z,\frac{1}{2}+it,2\upsilon)}\,d\mu(z)$ \cite[p. 243, Remark 2.4]{Hejhal} (recall that $E_j(z,s,2\upsilon)$ equals ``$E_j(z,s,-2\upsilon)$" in \cite{Hejhal,AndJensKron}; also note that $\widetilde{f}$ has compact support in $\scrF$). This decomposition also gives rise to the following Plancherel theorem:
\begin{equation*}
\|\widetilde{f}\|^2_{L^2(\GaH)} =\sum_{i=1}^{\infty} | \langle \widetilde{f},\widetilde{\phi}_{i,2\upsilon}\rangle_{L^2(\GaH,-2\upsilon)}|^2 +2\pi\sum_{j=1}^{\kappa} \int_0^{\infty} |g_j(t)|^2\,dt
\end{equation*}
(cf.\ \cite[p.\ 373 (item 12) and p.\ 374 (item 15)]{Hejhal}). Note that since $f\in L^2(\GaG)_{2\upsilon}$ and $\scrH_i\cap L^2(\GaG)_{2\upsilon}=\CC \phi_{i,2\upsilon}$, $f_i=\langle f, \phi_{i,2\upsilon}\rangle_{L^2(\GaG)}\phi_{i,2\upsilon}$. Furthermore, 
\begin{equation*}
\vecv_{f,j}(t)=c_{j,t}\langle f,\EE_j^{1/2+it}(\cdot,\vece_{2\upsilon})\rangle_{L^2(\GaG)}\vece_{2\upsilon}=\sqrt{\sfrac{1}{2\pi}}g_j(t)\vece_{2\upsilon}.
\end{equation*}
Using the fact that we have an isometry between $L^2(\GaG)_{2\upsilon}$ and $L^2(\GaH,-2\upsilon)$ (see Lemma \ref{GaH/GaGint}), we then obtain
\begin{align*}
\|f\|^2_{L^2(\GaG)}=&\sum_{i=1}^{\infty} | \langle f,\phi_{i,2\upsilon}\rangle_{L^2(\GaG)}|^2
+\sum_{j=1}^{\kappa} \frac{\wty(h_j)}{2\pi}  \int_0^{\infty} |\langle f,\EE_j^{1/2+it}(\cdot,\vece_{2\upsilon})\rangle_{L^2(\GaG)}|^2\,dt \\&=\sum_{i=1}^{\infty}\|f_i\|^2_{L^2(\GaG)}+\sum_{j=1}^{\kappa}\int_0^{\infty} \|\vecv_{f,j}(t)\|^2\,dt.
\end{align*}
The maps $f\mapsto f_i$ and $f\mapsto \int_0^{\infty }\vecv_{f,j}(t)\,dt$ in the statement of the theorem are thus densely defined and bounded with respect to $\|\cdot\|_{L^2(\GaG)}$; they therefore have a unique extension to all of $L^2(\GaG)$.
\end{proof}

\begin{cor}\label{fdecomp2}
The formulas for $f\mapsto f_i$ and $f\mapsto \int_0^{\infty }\vecv_{f,j}(t)\,dt$ given in Theorem \ref{fdecomp} also hold for $f\in L^2(\GaG)$ such that $f\in  C^{\infty}(\GaG)\cap L^p(\GaG)$ for some $p>2$. 
\end{cor}
\begin{proof}
Let $\lbrace \chi_n \rbrace_{n=1}^{\infty}\subset C_c^{\infty}(\GaG)$ be a sequence of smooth approximations of indicator functions that ``fill out" all of $\GaG$, i.e. $\chi_{n+1}(\Gamma g)\geq \chi_{n}(\Gamma g)$, $0\leq \chi_n(\Gamma g)\leq 1$, $\lim_{n\rightarrow \infty} \chi_n(\Gamma g)=1$ for all $g\in G$, and $\lim_{n\rightarrow\infty}\|\chi_n-1\|_{L^2(\GaG)}=0 $. The key properties that we need are that $\lim_{n\rightarrow\infty}\|f\chi_n-f\|_{L^2(\GaG)} =0$ and $|f(\Gamma g)\chi_n(\Gamma g)| \leq|f(\Gamma g)|$. Since the projections $\Proj_i$ are bounded operators on $L^2(\GaG)$, we have $f_i=\lim_{n\rightarrow\infty}\Proj_i(f\chi_n)=\Proj_i(f)$. We have \begin{equation*}
\lim_{n\rightarrow\infty} \int_0^{\infty} \| \vecv_{f\chi_n,j}(t)-\vecv_{f,j}(t)\|_{L^2(K)}^2\,dt=0,
\end{equation*}
and now show that $\vecv_{f,j}$ is in fact as in Theorem \ref{fdecomp}. Since $\vecv_{f\chi_n,j}\rightarrow \vecv_{f,j}$ in $L^2(\RR_{\geq 0}, V)$, we may assume (after possibly passing to a subsequence) that the sequence $\lbrace\chi_n\rbrace_n$ has been chosen so that \begin{equation*}
\lim_{n\rightarrow\infty} \| \vecv_{f\chi_{n},j}(t)-\vecv_{f,j}(t)\|=0\qquad\mathrm{for \;a.e.\;} t\in\RR_{\geq 0}.
\end{equation*}
Given $\upsilon\in \ZZ$, we have
\begin{align*}
\langle \vecv_{f,j}(t),\vece_{2\upsilon}\rangle_{L^2(K)}=&\lim_{n\rightarrow\infty}\langle  \vecv_{f\chi_{n},j}(t),\vece_{2\upsilon}\rangle_{L^2(K)}\\&=\lim_{n\rightarrow\infty}\sfrac{\wty(h_j)^{1/2-it}}{\sqrt{2\pi}} \langle f\chi_n, \EE_j^{1/2+it}(\cdot,\vece_{2\upsilon})\rangle_{L^2(\GaG)}\\&=\sfrac{\wty(h_j)^{1/2-it}}{\sqrt{2\pi}} \langle f, \EE_j^{1/2+it}(\cdot,\vece_{2\upsilon})\rangle_{L^2(\GaG)},
\end{align*}
the third equality holding due to our choice of the functions $\chi_n$ (i.e.\ they were chosen so as to permit the use of dominated convergence at this stage).
\end{proof}

\subsection{Tensor products of representations}
We will now embed a copy of the representation $(\pi^s_j,V^{\infty})\otimes (\pi^r_k,V^{\infty})$ in $\big(\rho,L^2(\GaG) \big)$ (here $V^{\infty}\otimes V^{\infty}$ denotes the usual abstract tensor product of vector spaces). Our strategy will be somewhat similar to that of Definition \ref{Psirsdef} and Proposition \ref{Phiprops}: we will multiply together Eisenstein maps of elements of $V^{\infty}$ and then subtract other Eisenstein maps until we obtain a function in $L^2(\GaG)$.
\begin{defn}\label{Phiuv}
For $j,k\in\lbrace 1,\ldots,\kappa\rbrace$, $r,s\in\lbrace z\in\CC\,:\,\Re(z)>\frac{1}{4}\rbrace \setminus \scrE$ such that $r+s,1-r+s,r+1-s,2-r-s\in \lbrace z\in\CC\,:\,\Re(z)>\frac{1}{3}\rbrace \setminus \scrE$, and $\vecu,\vecv\in V^{\infty}$, let $\Psi_{j,k}^{r,s}(\vecu\otimes\vecv)\in C^{\infty}(\GaG)$ be defined by
\begin{align*}
\Psi_{j,k}^{r,s}(\vecu\otimes\vecv):=& \EE_j^r(\cdot,\scrK_j\vecu)\EE_k^s(\cdot,\scrK_k\vecv)-\delta_{j,k}\EE_j^{r+s}\big(\cdot,\scrK_j(\vecu\vecv)\big)\\ &-\varphi_{k,j}(s)\wty(h_j)^{1- s}\wty(h_k)^{-s}\EE_j^{r+1-s}\big(\cdot,\scrK_j(\vecu\scrI_s\vecv)\big)\\&-\varphi_{j,k}(r)\wty(h_j)^{-r}\wty(h_k)^{1-r}\EE_k^{1-r+s}\big(\cdot,\scrK_k(\vecv\scrI_r\vecu)\big)\\&-\sum_{i=1}^{\kappa} \varphi_{j,i}(r)\varphi_{k,i}(s)\wty(h_i)^{2-s-r}\wty(h_j)^{- s}\wty(h_k)^{-r}\EE_i^{2-r-s}(\cdot, \scrK_i\big((\scrI_r\vecu)(\scrI_s\vecv)\big)\big).
\end{align*}
\end{defn}
Here we have used $\veca \vecb$ to denote the pointwise product of $\veca,\,\vecb\in V^{\infty}$; thus $\veca\vecb\in V^{\infty}$. By Proposition \ref{intertwine}, $\scrI_s\vecu,\,\scrI_r\vecv\in V^{\infty}$. Furthermore, $\scrK_j$, $\scrK_k$ are just rotations, so $\scrK_j\veca,\,\scrK_k\veca\in V^{\infty}$ for all $\veca\in V^{\infty}$, showing that $\vecu\otimes\vecv\mapsto\Psi_{j,k}^{r,s}(\vecu\otimes\vecv)$ is a well-defined map on $V^{\infty}\otimes V^{\infty}$. That $\Psi_{j,k}^{r,s}(\vecu\otimes\vecv)\in C^{\infty}(\GaG)$ follows from Proposition \ref{Eisucont}. As before (cf.\ Proposition \ref{Phiprops}), we shall see that $\Psi_{j,k}^{r,s}(\vecu\otimes\vecv)$ can be extended to a holomorphic function for $(r,s)\in\big(\frac{1}{2}+i\RR\big)\times \big(\frac{1}{2}+i\RR\big)$. We will start by considering $\Psi_{j,k}^{r,s}(\vece_{2\sigma}\otimes\vece_{2\upsilon})$ ($\sigma,\upsilon\in\ZZ$), and then use Fourier decomposition to pass to general functions $\vecu,\vecv\in V^{\infty}$. 
\begin{prop}\label{Phiupssig}
The map $(r,s)\mapsto \left[\Psi_{j,k}^{r,s}(\vece_{2\sigma}\otimes\vece_{2\upsilon})\right](\Gamma g)$ is holomorphic on $(\frac{1}{2}+i\RR)\times (\frac{1}{2}+i\RR)$ for all $\upsilon,\sigma\in\ZZ$, $j,k\in\lbrace 1,\ldots,\kappa\rbrace$, and $g\in G$. 
\end{prop}
\begin{proof}
As in the proof of Proposition \ref{Phiprops}, we start by fixing $T>0$ and finding $\frac{1}{12}>\delta(=\delta_T)>0$ such that $\scrE\cap\big([\frac{1}{2}-2\delta,\frac{1}{2}+2\delta]+i[-T-\delta,T+\delta]\big)=\emptyset$, and $\lbrace r+s,1-r+s,r+1-s,2-r-s \rbrace \cap \scrE\subset \lbrace 1\rbrace$ for all $r,s\in [\frac{1}{2}-\delta,\frac{1}{2}+\delta]+i[-T-\delta,T+\delta]$. Letting $U_T=(\frac{1}{2}-\delta,\frac{1}{2}+\delta)+i(-T,T)$, from Definition \ref{Phiuv} we have that $(r,s)\mapsto \left[\Psi_{j,k}^{r,s}(\vece_{2\sigma}\otimes\vece_{2\upsilon})\right](\Gamma g)$ is well-defined and holomorphic at all $(r,s)\in U_T\times U_T$, except possibly the complex lines $s=r$ and $s=1-r$. Fixing $r\in U_T$, we consider the function $s\mapsto \left[\Psi_{j,k}^{r,s}(\vece_{2\sigma}\otimes \vece_{2\upsilon})\right](\Gamma n_x a_y \kthe)$. For $s\in U_T\setminus\lbrace r,1-r\rbrace$, Definition \ref{Phiuv} and \eqref{Eismapexpr} give 
\begin{align*}
\left[\Psi_{j,k}^{r,s}(\vece_{2\upsilon}\otimes \vece_{2\sigma})\right](\Gamma n_x a_y \kthe)=&\frac{\vece_{2(\upsilon+\sigma)}(\kthe)}{\wty(h_j)^r \wty(h_k)^s}\bigg(E_j(x+i y ,r, 2\upsilon)E_k(x+i y ,s, 2\sigma)\\&\qquad-\delta_{j,k}E_j\big(x+i y ,r+s, 2(\upsilon+\sigma)\big)\\&\qquad-\varphi_{k,j}(s)\mathfrak{i}_{s,2\sigma}E_j\big(x+i y ,r+1-s, 2(\upsilon+\sigma)\big)\\&\qquad-\varphi_{j,k}(r)\mathfrak{i}_{r,2\upsilon}E_k\big(x+i y ,1-r+s, 2(\upsilon+\sigma)\big)\\&-\mathfrak{i}_{r,2\upsilon}\mathfrak{i}_{s,2\sigma}\sum_{i=1}^{\kappa} \varphi_{j,i}(r)\varphi_{k,i}(s)E_i\big(x+i y ,2-r-s, 2(\upsilon+\sigma)\big)\bigg).
\end{align*}
By Lemmas \ref{Ej2ubdd} and \ref{Eisu1}, if $\sigma\neq-\upsilon$, then all terms in this expression are holomorphic (in $s$) for all $s\in U_T$; in particular, also at $s=r$ or $s=1-r$. By symmetry, fixing $s\in U_T$, $r\mapsto  \left[\Psi_{j,k}^{r,s}(\vece_{2\sigma}\otimes \vece_{2\upsilon})\right](\Gamma n_x a_y \kthe)$ is also holomorphic, hence $(r,s)\mapsto  \left[\Psi_{j,k}^{r,s}(\vece_{2\sigma}\otimes \vece_{2\upsilon})\right](\Gamma n_x a_y \kthe)$ is jointly holomorphic on $U_T\times U_T$.

In the case $\sigma=-\upsilon$, we have
\begin{align*}
\left[\Psi_{j,k}^{r,s}(\vece_{2\upsilon}\otimes \vece_{-2\upsilon})\right](\Gamma n_x a_y \kthe)=&\frac{1}{\wty(h_j)^r \wty(h_k)^s}\bigg(E_j(x+i y ,r, 2\upsilon)E_k(x+i y ,s, -2\upsilon)\\&\qquad-\delta_{j,k}E_j(x+i y ,r+s)\\&\qquad-\varphi_{k,j}(s)\mathfrak{i}_{s,-2\upsilon}E_j(x+i y ,r+1-s)\\&\qquad-\varphi_{j,k}(r)\mathfrak{i}_{r,2\upsilon}E_k(x+i y ,1-r+s)\\&-\sum_{i=1}^{\kappa} \varphi_{j,i}(r)\varphi_{k,i}(s)\mathfrak{i}_{r,2\upsilon}\mathfrak{i}_{s,-2\upsilon}E_i(x+i y ,2-r-s)\bigg).
\end{align*}
As in Proposition \ref{Phiprops}, there are now three cases to consider: \textit{i)} $s=1-r$, $r\neq \frac{1}{2}$, \textit{ii)} $s= r$, $r\neq \frac{1}{2}$, \textit{iii)} $s=r=\frac{1}{2}$, and for brevity we focus only on $s=1-r$, $r\neq\frac{1}{2}$, the other cases being dealt with analogously. We let $s=1-r+w$, where $w\in\CC$ is contained in a small neighbourhood of zero. By Lemma \ref{Eis1pole}, we have
\begin{align*}
&\left[\Psi_{j,k}^{r,1-r+w}(\vece_{2\upsilon}\otimes \vece_{-2\upsilon})\right](\Gamma n_x a_y \kthe)\\&\quad\qquad\qquad\qquad=\frac{1}{\wty(h_j)^r \wty(h_k)^{1-r+w}}\Bigg(E_j(x+i y ,r, 2\upsilon)E_k(x+i y ,1-r+w, -2\upsilon)\\&\qquad\qquad\qquad\qquad-\delta_{j,k}\left( \frac{\mu(\GaH)^{-1}}{w}+\widetilde{E}_j(z,1+w)\right)\\&\qquad\qquad\qquad\qquad-\varphi_{k,j}(1-r+w)\mathfrak{i}_{1-r+w,-2\upsilon}E_j(x+i y ,2r-w)\\&\qquad\qquad\qquad\qquad-\varphi_{j,k}(r)\mathfrak{i}_{r,2\upsilon}E_k(x+i y ,2-2r+w)\\&\qquad\qquad\qquad\qquad-\sum_{i=1}^{\kappa} \varphi_{j,i}(r)\varphi_{k,i}(1-r+w)\mathfrak{i}_{r,2\upsilon}\mathfrak{i}_{1-r+w,-2\upsilon}\left( \frac{\mu(\GaH)^{-1}}{-w}+\widetilde{E}_i(z,1-w)\right)\Bigg).
\end{align*}
Now, $\mathfrak{i}_{r,2\upsilon}\mathfrak{i}_{1-r+w,-2\upsilon}=1+O_{r,\upsilon}(|w|)$, and $\sum_{i=1}^{\kappa} \varphi_{j,i}(r)\varphi_{k,i}(1-r+w)=\delta_{j,k}+O_{r}(|w|)$, showing that $\left|\left[\Psi_{j,k}^{r,1-r+w}(\vece_{2\upsilon}\otimes \vece_{-2\upsilon})\right](\Gamma n_x a_y \kthe) \right| $ remains bounded as $w\rightarrow0$. Since $s\mapsto \left[\Psi_{j,k}^{r,s}(\vece_{2\upsilon}\otimes \vece_{-2\upsilon})\right](\Gamma n_x a_y \kthe) $ is (by construction) meromorphic in $s$, this shows that the map is in fact holomorphic at $s=1-r$.
\end{proof}
\begin{prop}\label{PsiupssigFour}
Given $T>0$, let $0<\delta<\frac{1}{12}$ be such that $\big([\frac{1}{2}-2\delta,\frac{1}{2}+2\delta]+i[-T-\delta,T+\delta]\big)\cap \scrE=\emptyset$ and $\lbrace r+s,1-r+s,r+1-s,2-r-s\rbrace \cap \scrE\subset\lbrace 1\rbrace$ for all $r,s\in [\frac{1}{2}-2\delta,\frac{1}{2}+2\delta]+i[-T-2\delta,T+2\delta]$. Then for all $r,s\in [\frac{1}{2}-\delta,\frac{1}{2}+\delta]+i[-T,T]$, $j,k,l\in \lbrace 1,\ldots,\kappa\rbrace$, $\upsilon,\sigma\in\ZZ$, $x+iy\in\HH$ such that $y_l\gg 1$, and $\theta\in \RR/2\pi\ZZ$, 
\begin{align}\label{Psisigupsbdd}
\left|\left[\Psi_{j,k}^{r,s}(\vece_{2\upsilon}\otimes \vece_{2\sigma})\right](\Gamma n_x a_y \kthe)\right|\ll_{\delta,T}(1+|\upsilon|^{\frac{1}{2}+5\delta})(1+|\sigma|^{\frac{1}{2}+5\delta})y_l^{\frac{1}{6}+2\delta}.
\end{align}
\end{prop}
\begin{proof}
For $r\in [\frac{1}{2}-\delta,\frac{1}{2}+\delta]+i[-T,T]$ and $s\in [\frac{1}{2}-2\delta,\frac{1}{2}+2\delta]+i[-T-\delta,T+\delta]$ such that $|r-s|\geq\delta$ and $|r+s-1|\geq\delta$, we have
\begin{align*}
\left|\left[\Psi_{j,k}^{r,s}(\vece_{2\upsilon}\otimes \vece_{2\sigma})\right](\Gamma n_x a_y \kthe)\right|=&\frac{1}{\wty(h_j)^{\Re(r)} \wty(h_k)^{\Re(s)}}\Bigg|\bigg(E_j(x+i y ,r, 2\upsilon)E_k(x+i y ,s, 2\sigma)\\&\qquad-\delta_{j,k}E_j\big(x+i y ,r+s, 2(\upsilon+\sigma)\big)\\&\qquad-\varphi_{k,j}(s)\mathfrak{i}_{s,2\sigma}E_j\big(x+i y ,r+1-s, 2(\upsilon+\sigma)\big)\\&\qquad-\varphi_{j,k}(r)\mathfrak{i}_{r,2\upsilon}E_k\big(x+i y ,1-r+s, 2(\upsilon+\sigma)\big)\\&-\mathfrak{i}_{r,2\upsilon}\mathfrak{i}_{s,2\sigma}\sum_{i=1}^{\kappa} \varphi_{j,i}(r)\varphi_{k,i}(s)E_i\big(x+i y ,2-r-s, 2(\upsilon+\sigma)\big)\bigg)\Bigg|.
\end{align*}
Now using either Lemma \ref{Ej2ubdd}, or \eqref{EjkFourier} and \eqref{Eisexp}, we have (for all $\epsilon>0$)
\begin{align*}
&\Bigg|\bigg(E_j(x+i y ,r, 2\upsilon)E_k(x+i y ,s, 2\sigma)-\delta_{j,k}E_j\big(x+i y ,r+s, 2(\upsilon+\sigma)\big)\\&\quad-\varphi_{k,j}(s)\mathfrak{i}_{s,2\sigma}E_j\big(x+i y ,r+1-s, 2(\upsilon+\sigma)\big)-\varphi_{j,k}(r)\mathfrak{i}_{r,2\upsilon}E_k\big(x+i y ,1-r+s, 2(\upsilon+\sigma)\big)\\&\qquad\qquad\qquad\qquad\qquad\qquad\qquad-\mathfrak{i}_{r,2\upsilon}\mathfrak{i}_{s,2\sigma}\sum_{i=1}^{\kappa} \varphi_{j,i}(r)\varphi_{k,i}(s)E_i\big(x+i y ,2-r-s, 2(\upsilon+\sigma)\big)\bigg)\Bigg|
\end{align*}
\begin{align*}
&=\Bigg| \left( \delta_{j,l}y_l^r+\mathfrak{i}_{r,2\upsilon} \varphi_{j,l}(r) y_l^{1-r}+O\big( (1+|\upsilon|^{\frac{1}{2}+2\delta+\epsilon} )y_l^{-\frac{1}{3}}\big)\right)\\&\qquad\qquad\times\left( \delta_{j,l}y_l^s+\mathfrak{i}_{s,2\sigma} \varphi_{j,l}(s) y_l^{1-s}+O\big( (1+|\sigma|^{\frac{1}{2}+4\delta+\epsilon} )y_l^{-\frac{1}{3}}\big)\right)
\\&\qquad-\delta_{j,k} \left( \delta_{j,l}y_l^{s+r}+\mathfrak{i}_{{s+r},2(\upsilon+\sigma)} \varphi_{j,l}(s+r) y_l^{1-r-s}+O\big( (1+|\upsilon+\sigma|^{\frac{1}{2}+\epsilon} )y_l^{-\frac{1}{3}}\big)\right)
\\&\qquad-\varphi_{k,j}(s)\mathfrak{i}_{s,2\sigma} \Big( \delta_{j,l}y_l^{r+1-s}+\mathfrak{i}_{{r+1-s},2(\upsilon+\sigma)} \varphi_{j,l}(1+r-s) y_l^{s-r}\\&\qquad\qquad\qquad\qquad\qquad\qquad\qquad\qquad\qquad+O\big( (1+|\upsilon+\sigma|^{\frac{1}{2}+\epsilon} )y_l^{-\frac{1}{3}}\big)\Big)
\\&\qquad-\varphi_{j,k}(r)\mathfrak{i}_{r,2\upsilon} \Big( \delta_{k,l}y_l^{1-r+s}+\mathfrak{i}_{{1-r+s},2(\upsilon+\sigma)} \varphi_{k,l}(1-r+s) y_l^{r-s}\\&\qquad\qquad\qquad\qquad\qquad\qquad\qquad\qquad\qquad+O\big( (1+|\upsilon+\sigma|^{\frac{1}{2}+\epsilon} )y_l^{-\frac{1}{3}}\big)\Big)
\\&\qquad-\mathfrak{i}_{r,2\upsilon}\mathfrak{i}_{s,2\sigma}\sum_{i=1}^{\kappa} \varphi_{j,i}(r)\varphi_{k,i}(s) \Big( \delta_{i,l}y_l^{2-r-s}+\mathfrak{i}_{{2-r-s},2(\upsilon+\sigma)} \varphi_{i,l}(2-r-s) y_l^{r+s-1}\\&\qquad\qquad\qquad\qquad\qquad\qquad\qquad\qquad\qquad\qquad\quad+O\big( (1+|\upsilon+\sigma|^{\frac{1}{2}+\epsilon} )y_l^{-\frac{1}{3}}\big)\Big)\Bigg|
\end{align*}
\begin{align*}
&= \Bigg|\delta_{j,k}\mathfrak{i}_{{s+r},2(\upsilon+\sigma)} \mathfrak{i}_{r+1-s,2(\sigma+\upsilon)}\mathfrak{i}_{s,2\upsilon}  y_l^{s-r}+\varphi_{k,j}(s)\varphi_{j,l}(r+1-s)\mathfrak{i}_{r+1-s,2(\sigma+\upsilon)}\mathfrak{i}_{s,2\upsilon}  y_l^{s-r}\\&\qquad\qquad\qquad+\varphi_{j,k}(r)\varphi_{k,l}(1-r+s)\mathfrak{i}_{1-r+s,2(\sigma+\upsilon)}\mathfrak{i}_{r,2\sigma}  y_l^{r-s}\\&\qquad\qquad\qquad\qquad\qquad\qquad+\mathfrak{i}_{r,2\sigma}\mathfrak{i}_{s,2\upsilon}\mathfrak{i}_{2-r-s,2(\sigma+\upsilon)}y_l^{r+s-1} \sum_{i=1}^{\kappa} \varphi_{j,i}(r)\varphi_{k,i}(s)\varphi_{i,l}(2-r-s) \Bigg|\\&+O \Big(\!(1+|\sigma|^{\frac{1}{2}+4\delta+\epsilon})(y_l^{\Re(r)-\frac{1}{3}}+|\mathfrak{i}_{r,2\upsilon}| y_l^{\frac{2}{3}-\Re(r)})\!+\!(1+|\upsilon|^{\frac{1}{2}+2\delta+\epsilon})(y_l^{\Re(s)-\frac{1}{3}}+|\mathfrak{i}_{s,2\sigma}| y_l^{\frac{2}{3}-\Re(s)}) \\&\qquad\qquad\qquad+ (1+|\sigma|^{\frac{1}{2}+4\delta+\epsilon})(1+|\upsilon|^{\frac{1}{2}+2\delta+\epsilon})y_l^{-\frac{2}{3}}
\\&\qquad\qquad+(1+|\sigma+\upsilon|^{\frac{1}{2}+\epsilon})\big(1+|\mathfrak{i}_{s,2\sigma}|+|\mathfrak{i}_{r,2\upsilon}|+|\mathfrak{i}_{r,2\upsilon}\mathfrak{i}_{s,2\sigma}|  \big)y_k^{-\frac{1}{3}} \Big).
\end{align*}
Now using $|\mathfrak{i}_{z,2\upsilon}|\ll_z (1+|\upsilon|)^{1-2\Re(z)}$ (see e.g.\ \cite[Lemma 2.10]{Sarnak}), and that $\frac{1}{2}-\delta \leq \Re(r)\leq\frac{1}{2}+\delta$, $\frac{1}{2}-2\delta \leq \Re(s)\leq\frac{1}{2}+2\delta$, we obtain (for all $\epsilon>0$; in particular for $\epsilon=\delta$)
\begin{align}\label{maxdelt}
\left|\left[\Psi_{j,k}^{r,s}(\vece_{2\upsilon}\otimes \vece_{2\sigma})\right](\Gamma n_x a_y \kthe)\right|&\ll_{\delta,\epsilon,r,s} (1+|\sigma|^{\frac{1}{2}+4\delta+\epsilon})(1+|\upsilon|^{\frac{1}{2}+4\delta+\epsilon})y_l^{\frac{1}{6}+2\delta}
\\&\notag\ll_{\delta,r,s} (1+|\sigma|^{\frac{1}{2}+5\delta})(1+|\upsilon|^{\frac{1}{2}+5\delta})y_l^{\frac{1}{6}+2\delta}.
\end{align} 
Assuming now that $r,s\in[\frac{1}{2}-\delta,\frac{1}{2}+\delta]+i[-T,T]$ and either $|s-r|<\delta$ or $|1-r-s|<\delta$, since $s\mapsto \left[\Psi_{j,k}^{r,s}(\vece_{2\upsilon}\otimes \vece_{2\sigma})\right](\Gamma n_x a_y \kthe)$ is holomorphic for all $r,s\in[\frac{1}{2}-\delta,\frac{1}{2}+\delta]+i[-T,T]$, in particular at $s=r$ or $1-r$ (see the proof of Proposition \ref{Phiupssig}), by the maximum principle,
\begin{align}\label{maxprinciple}
&\left|\left[\Psi_{j,k}^{r,s}(\vece_{2\upsilon}\otimes \vece_{2\sigma})\right](\Gamma n_x a_y \kthe)\right|\leq \max_{\underset{|r-\zeta|=\delta\,\mathrm{or}\,|1-r-\zeta|=\delta}{\zeta\in\CC}}\left| \left[\Psi_{j,k}^{r,\zeta}(\vece_{2\upsilon}\otimes \vece_{2\sigma})\right](\Gamma n_x a_y \kthe) \right|\\\notag&\qquad\qquad\ll_{\delta,r} (1+|\sigma|^{\frac{1}{2}+5\delta})(1+|\upsilon|^{\frac{1}{2}+5\delta})y_l^{3\delta},
\end{align}
by \eqref{maxdelt}. Together, \eqref{maxdelt} and \eqref{maxprinciple} prove the proposition.
\end{proof}
\begin{cor}\label{Phiuvdef}
The map $(r,s)\mapsto\left[\Psi_{j,k}^{r,s}(\vecu\otimes\vecv)\right](\Gamma g)$ is holomorphic for all $(r,s)\in(\frac{1}{2}+i\RR)\times(\frac{1}{2}+i\RR)$, all $\vecu,\vecv\in V^{\infty}$, and all $g\in G$.
\end{cor}
\begin{proof}
For arbitrary $T>0$, let $\delta=\delta_T>0$ be as in Proposition \ref{PsiupssigFour}. Writing $\vecu=\sum_{\upsilon\in\ZZ} a_{2\upsilon}\vece_{2\upsilon}$ and $\vecv=\sum_{\upsilon\in\ZZ} b_{2\sigma}\vece_{2\sigma}$, we then have $\EE_j^r(\cdot,\scrK_j\vecu)=\sum_{\upsilon\in\ZZ}a_{2\upsilon}\EE_j^r(\cdot,\scrK_j\vece_{2\upsilon})$, and a similar decomposition of the other terms in Definition \ref{Phiuv}. Expanding all Eisenstein maps in such infinite sums and collecting terms gives that for $r,s\in (\frac{1}{2}-\delta,\frac{1}{2}+\delta)+i\RR$ with $s\neq r,\,1-r$, $\Psi_{j,k}^{r,s}(\vecu\otimes\vecv)$, we have
\begin{align*}
\left[\Psi_{j,k}^{r,s}(\vecu\otimes\vecv)\right](\Gamma g)=\sum_{\upsilon,\sigma\in\ZZ} a_{2\upsilon}b_{2\sigma}\left[\Psi_{j,k}^{r,s}(\vece_{2\upsilon}\otimes\vece_{2\sigma})\right](\Gamma g),
\end{align*}
with the sum converging absolutely and uniformly for $r,s$ in $[\frac{1}{2}-\delta,\frac{1}{2}+\delta]+i[-T,T]$; this is due to the bound \eqref{Psisigupsbdd} and $|a_{2\upsilon}|\ll_{\vecu,n} (1+|\upsilon|)^{-n}$, $|b_{2\sigma}|\ll_{\vecv,n} (1+|\sigma|)^{-n}$ (for all $n\in \NN$). Since each $(r,s)\mapsto\left[\Psi_{j,k}^{r,s}(\vece_{2\upsilon}\otimes\vece_{2\sigma})\right](\Gamma g)$ is holomorphic on $\big([\frac{1}{2}-\delta,\frac{1}{2}+\delta]+i[-T,T]\big)\times\big([\frac{1}{2}-\delta,\frac{1}{2}+\delta]+i[-T,T]\big)$, so is  $(r,s)\mapsto\left[\Psi_{j,k}^{r,s}(\vecu\otimes\vecv)\right](\Gamma g)$.
\end{proof}
\begin{cor}\label{Phiuvbd}
$\Psi_{j,k}^{r,s}(\vecu\otimes\vecv)\in L^3(\GaG)$ for all $\vecu,\vecv\in V^{\infty}$, $r,s\in \frac{1}{2}+i\RR$. Moreover,
\begin{equation*}
\|\Psi_{j,k}^{r,s}(\vecu\otimes\vecv)\|_{L^2(\GaG)}\ll_{r,s,\beta} \scrS_{\beta}(\vecu)\scrS_{\beta}(\vecv)
\end{equation*}
for all $\beta>1$.
\end{cor}
\begin{proof}
Choosing $T>\max\lbrace |\Im(r)|,|\Im(s)|\rbrace$ and $0<\delta<\frac{1}{12}$ as in Proposition \ref{PsiupssigFour} shows that $|\Psi_{j,k}^{r,s}(\vece_{2\upsilon}\otimes \vece_{2\sigma})|^p=O(y_l^{\eta})$ for some $\eta<1$ as $y_l\rightarrow\infty$. Since the measure $d\mu_{\GaG}$ is given by $\frac{dx_ldy_ld\theta}{y_l^2}$ in the cusp $\eta_l$, this shows that 
\begin{equation*}
\|\Psi_{j,k}^{r,s}(\vece_{2\upsilon}\otimes\vece_{2\sigma})\|_{L^2(\GaG)}\ll (1+|\sigma|^{\frac{1}{2}+5\delta})(1+|\upsilon|^{\frac{1}{2}+5\delta}),
\end{equation*}
hence writing $\vecu=\sum_{\upsilon\in\ZZ} a_{2\upsilon}\vece_{2\upsilon}$ and $\vecu=\sum_{\upsilon\in\ZZ} b_{2\sigma}\vece_{2\sigma}$, 
\begin{align*}
\|\Psi_{j,k}^{r,s}(\vecu\otimes\vecv)\|_{L^p(\GaG)}\leq &\sum_{\upsilon,\sigma\in\ZZ} |a_{2\upsilon}||b_{2\sigma}|\|\Psi_{j,k}^{r,s}(\vece_{2\upsilon}\otimes\vece_{2\sigma})\|_{L^p(\GaG)}\\& \ll \sum_{\upsilon,\sigma\in\ZZ} |a_{2\upsilon}||b_{2\sigma}| (1+|\sigma|^{\frac{1}{2}+5\delta})(1+|\upsilon|^{\frac{1}{2}+5\delta})<\infty.
\end{align*}
To prove the second part of the claim, given $\beta>1$, choose $\delta>0$ as in Proposition \ref{PsiupssigFour} so that $\alpha:=\beta-\frac{1}{2}-5\delta>\frac{1}{2}$. Then
\begin{align*}
\|\Psi_{j,k}^{r,s}(\vecu\otimes\vecv)\|_{L^2(\GaG)} \ll& \sum_{\upsilon,\sigma\in\ZZ} |a_{2\upsilon}||b_{2\sigma}| (1+|\sigma|^{\frac{1}{2}+5\delta})(1+|\upsilon|^{\frac{1}{2}+5\delta})\\&\ll \left(\sum_{\upsilon\in\ZZ} |a_{2\upsilon}|(1+|\upsilon|^{\beta})(1+|\upsilon|)^{-\alpha}\right)\left(\sum_{\sigma\in\ZZ}|b_{2\sigma}| (1+|\sigma|^{\beta})(1+|\sigma|)^{-\alpha}\right)\\&\quad\ll\scrS_{\beta}(\vecu)\scrS_{\beta}(\vecv).
\end{align*}
\end{proof}
\begin{remark}
In fact, $\Psi_{j,k}^{r,s}(\vecu\otimes\vecv)\in L^p(\GaG)$ for all $\vecu,\vecv\in V^{\infty}$, $r,s\in \frac{1}{2}+i\RR$ and all $p<\infty$. Higher order Sobolev norms are required to prove this, however.
\end{remark}

\begin{lem}\label{Phiuvintertwine} For all $r,s\in \frac{1}{2}+i\RR$,
\begin{equation*}
\rho(g)\Psi_{j,k}^{r,s}(\vecu\otimes\vecv)=\Psi_{j,k}^{r,s}(\pi^r(g)\vecu\otimes\pi^s(g)\vecv)\qquad \forall g\in G,\, \vecu,\vecv\in V^{\infty}.
\end{equation*}
Consequently, $\Psi_{j,k}^{r,s}(\vecu\otimes\vecv)\in C^{\infty}(\GaG)$ for all $\vecu,\vecv\in V^{\infty}$.
\end{lem}
\begin{proof}
The key components of the proof are \eqref{Psjinterwtwine}, Proposition \ref{intertwine}, and \eqref{Eismapintertwine}, as well as the identity $\pi_i^{c+d}(g)(\veca\vecb)=\big(\pi_i^{c}(g)\veca\big)\big(\pi_i^{d}(g)\vecb\big)$ for all $i\in\lbrace 1,\ldots,\kappa\rbrace$, $\veca,\vecb\in V^{\infty}$, $c,d\in \CC$, and $g\in G$. Combining these facts, we obtain, for $s\neq r,1-r$:
\begin{align*}
\rho(g)\Psi_{j,k}^{r,s}(\vecu\otimes\vecv)&:= \big(\rho(g)\EE_j^r(\cdot,\scrK_j\vecu)\big)\big(\rho(g)\EE_k^s(\cdot,\scrK_k\vecv)\big)-\delta_{j,k}\rho(g)\EE_j^{r+s}\big(\cdot,(\scrK_j\vecu)(\scrK_k\vecv)\big)\\ &-\varphi_{k,j}(s)\wty(h_j)^{1- s}\wty(h_k)^{-s}\rho(g)\EE_j^{r+1-s}\big(\cdot,(\scrK_j\vecu)(\scrI_s\scrK_j\vecv)\big)\\&-\varphi_{j,k}(r)\wty(h_j)^{-r}\wty(h_k)^{1-r}\rho(g)\EE_k^{1-r+s}\big(\cdot,(\scrI_r\scrK_k\vecu)(\scrK_k\vecv)\big)\\&-\sum_{i=1}^{\kappa} \varphi_{j,i}(r)\varphi_{k,i}(s)\wty(h_i)^{2-s-r}\wty(h_j)^{-s}\wty(h_k)^{-r}\rho(g)\EE_i^{2-r-s}(\cdot, (\scrI_r\scrK_i\vecu)(\scrI_s\scrK_i\vecv)\big)
\end{align*}
\begin{align*}=&\EE_j^r(\cdot,\pi_j^r(g)\scrK_j\vecu)\EE_k^s(\cdot,\pi_k^s(g)\scrK_k\vecv)-\delta_{j,k}\EE_j^{r+s}(\cdot,\pi^{r+s}_j(g)\big((\scrK_j\vecu)(\scrK_k\vecv)\big))\\ &-\varphi_{k,j}(s)\wty(h_j)^{1- s}\wty(h_k)^{-s}\EE_j^{r+1-s}(\cdot,\pi_j^{r+1-s}(g)\big((\scrK_j\vecu)(\scrI_s\scrK_j\vecv)\big))\\&-\varphi_{j,k}(r)\wty(h_j)^{-r}\wty(h_k)^{1-r}\EE_k^{1-r+s}(\cdot,\pi_j^{1-r+s}(g)\big((\scrI_r\scrK_k\vecu)(\scrK_k\vecv)\big))\\&\quad-\sum_{i=1}^{\kappa} \varphi_{j,i}(r)\varphi_{k,i}(s)\wty(h_i)^{2-s-r}\wty(h_j)^{-s}\wty(h_k)^{-r}\EE_i^{2-r-s}(\cdot, \pi_i^{2-r-s}(g)\big((\scrI_r\scrK_i\vecu)(\scrI_s\scrK_i\vecv)\big))
\end{align*}
\begin{align*}
=&\EE_j^r(\cdot,\pi_j^r(g)\scrK_j\vecu)\EE_k^s(\cdot,\pi_k^s(g)\scrK_k\vecv)-\delta_{j,k}\EE_j^{r+s}\big(\cdot,(\pi^{r}_j(g)\scrK_j\vecu)(\pi^{s}_k(g)\scrK_k\vecv)\big)\\ &-\varphi_{k,j}(s)\wty(h_j)^{1- s}\wty(h_k)^{-s}\EE_j^{r+1-s}\big(\cdot,(\pi^{r}_j(g)\scrK_j\vecu)(\pi^{1-s}_j(g)\scrI_s\scrK_j\vecv)\big)\\&-\varphi_{j,k}(r)\wty(h_j)^{-r}\wty(h_k)^{1-r}\EE_k^{1-r+s}\big(\cdot,(\pi_k^{1-r}(g)\scrI_r\scrK_k\vecu)(\pi_k^s(g)\scrK_k\vecv)\big)\\&\quad-\sum_{i=1}^{\kappa} \varphi_{j,i}(r)\varphi_{k,i}(s)\wty(h_i)^{2-s-r}\wty(h_j)^{-s}\wty(h_k)^{-r}\EE_i^{2-r-s}\big(\cdot, (\pi_i^{1-r}(g)\scrI_r\scrK_i\vecu)(\pi_i^{1-s}(g)\scrI_s\scrK_i\vecv)\big)
\end{align*}
\begin{align*}
=&\EE_j^r(\cdot,\scrK_j\pi^r(g)\vecu)\EE_k^s(\cdot,\scrK_k\pi^s(g)\vecv)-\delta_{j,k}\EE_j^{r+s}\big(\cdot,(\scrK_j\pi^{r}(g)\vecu)(\scrK_k\pi^{s}(g)\vecv)\big)\\ &-\varphi_{k,j}(s)\wty(h_j)^{1- s}\wty(h_k)^{-s}\EE_j^{r+1-s}\big(\cdot,(\scrK_j\pi^{r}(g)\vecu)(\scrI_s\scrK_j\pi^{s}(g)\vecv)\big)\\&-\varphi_{j,k}(r)\wty(h_j)^{-r}\wty(h_k)^{1-r}\EE_k^{1-r+s}\big(\cdot,(\scrI_r\scrK_k\pi^{r}(g)\vecu)(\scrK_k\pi^s(g)\vecv)\big)\\&\quad-\sum_{i=1}^{\kappa} \varphi_{j,i}(r)\varphi_{k,i}(s)\wty(h_i)^{2-s-r}\wty(h_j)^{-s}\wty(h_k)^{-r}\EE_i^{2-r-s}\big(\cdot, (\scrI_r\scrK_i\pi^{r}(g)\vecu)(\scrI_s\scrK_i\pi^{s}(g)\vecv)\big)\\&=\Psi_{j,k}^{r,s}(\pi^r(g)\vecu\otimes\pi^s(g)\vecv).
\end{align*}
Analytic continuation then shows that this also holds for $s=r$ or $s=1-r$.
\end{proof}

\section{Analytic continuation of representations}

\subsection{Analytic continuation of principal series representations} We shall recall one of the central ideas of \cite{Bern}, in which analytic continuation is used to ``extend" the  action of $G$ on $K$-finite vectors in $\scrP^s$ to a larger domain $U$, $G\subset U\subset \SL(2,\CC)$. By analytic continuation, certain relations that hold for $g\in G$ continue to hold for $g\in U$. However, the operators $\pi^s(g)$ cease to be unitary for $g\in U$; in particular, the precise asymptotic behaviour of $\|\pi^s(g)\vece_0\|$ as $g$ approaches the boundary of $U$ is key to the proof of Theorem \ref{thm2}.

We note also that the previous occurrences of analytic continuation in this article have been of $\CC$-valued meromorphic functions. In this section, however, we consider Hilbert space-valued analytic functions and their analytic continuations, cf.\ e.g.\ \cite[Chapter 3.5]{Rudin} 

The following result summarises the key results of \cite{Bern} that we will need:
\begin{prop}\label{Psancon}(cf.\ \cite[Proposition 2.2]{Bern}, as well as Appendix \ref{Analytic}).
Let $U\subset \SL(2,\CC)$ be  given by $U=\SL(2,\RR) I \SO(2,\CC)$, where $I=\lbrace \smatr{a}{0}{0}{a^{-1}}\,:\, a\in \CC,\,|\arg(a)|<\frac{\pi}{4}\rbrace $. Then
\begin{enumerate}[i)]
\item For each $s\in\CC$, the map from $G$ to $V$ given by $g\mapsto \pi^s(g)\vece_0$ may be analytically continued to 
all of $U$, and $\pi^s(g)\vece_0\in V^{\infty}$ for all $g\in U$.
\item For $0<\epsilon<\frac{1}{5}$, let $g_{\epsilon}=\smatr{e^{(\pi/4-\epsilon)i}}{}{}{e^{-(\pi/4-\epsilon)i}}\in U$. Then for all $t\in \RR$, we have the bound
\begin{equation*}
\|\pi^{1/2+it}(g_{\epsilon})\vece_0\|^2_{L^2(K)} \geq c  e^{(\pi-12\epsilon)|t|},
\end{equation*}
where the constant $c>0$ is independent of $t$ and $\epsilon$.
\end{enumerate}
\end{prop}
\begin{remark}
We parametrize the principal series representations by $s\in \CC$, while in \cite{Bern}, the representations are parametrized by $\lambda=1-2 s$.
\end{remark}

\subsection{Analytic continuation of $\big(\rho,L^2(\GaG)\big)$}
We will now use the map $\Psi_{j,k}^{r,s}:V^{\infty}\otimes V^{\infty}\rightarrow L^2(\GaG)$ to carry the analytic continuation of the representations $\scrP^r$, $\scrP^s$ into $L^2(\GaG)$. We start by making the following definition:
\begin{defn}\label{Psigdef}
For $j,k\in\lbrace 1,\ldots,\kappa\rbrace$, $r,s\in\frac{1}{2}+i\RR$, and $g\in U$, let $\Psi_g\in L^2(\GaG)$ be defined by
\begin{equation*}
\Psi_g:=\Psi_{j,k}^{r,s}(\pi^r(g)\vece_0\otimes\pi^s(g)\vece_0).
\end{equation*}
\end{defn}
Note that since $\pi^r(g)\vece_0,\pi^s(g)\vece_0\in V^{\infty}$ (by Proposition \ref{Psancon} \textit{i)}) $\Psi_g$ is well-defined by Corollary \ref{Phiuvdef} and lies in $L^2(\GaG)$ by Corollary \ref{Phiuvbd}. Moreover, by Lemma \ref{Phiuvintertwine}, $\Psi_g=\rho(g)\Psi_{e}$ for all $g\in G$ and $\Psi_g\in C^{\infty}(\GaG)$ for all $g\in U$. This allows us to view the map $g\mapsto \Psi_g$ as an ``extension" to $U$ of the action of $\rho(G)$ on $\Psi_e$. The following lemma shows that this extension is in fact analytic:
\begin{lem}\label{Phiga}
The map $g\mapsto \Psi_g$ is an $L^2(\GaG)$-valued analytic function on $U$.
\end{lem}
\begin{proof}
See Appendix \ref{AppB1}.
\end{proof}
We have the following decomposition of $\Psi_g$ in accordance with \eqref{L2decomp}:
\begin{equation}\label{Psigdecomp}
\Psi_g\mapsto\sum_{i=1}^{\infty} \Proj_i(\Psi_g)+\sum_{m=1}^{\kappa}\int_0^{\infty} \vecv_{\Psi_g,m}(t)\,dt.
\end{equation}

\begin{lem}\label{Phigb}
The maps $g\mapsto\Proj_i(\Psi_g)\in \scrH_i$, $g\mapsto \vecv_{\Psi_{g},m}(t)\in V$ are analytic on $U$ for all $i$, $m$, and almost every $t$.
\end{lem}
\begin{proof}
See Appendix \ref{AppB2}.
\end{proof}

\subsection{Proof of Theorem \ref{thm2}} Recall that in the statement of the theorem, we have fixed $j,k,l\in\lbrace 1,\ldots,\kappa\rbrace$, $s,r\in\frac{1}{2}+i\RR$, and consider $T> 1$. We start by considering $\Psi_e=\Psi_{j,k}^{r,s}(\vece_0\otimes\vece_0)$. By Corollary \ref{Phiuvbd} and Lemma \ref{Phiuvintertwine}, $\Psi_e\in L^2(\GaG)_0\cap C^{\infty}(\GaG)\cap L^p(\GaG)$ for all $p>1$. We may therefore apply Corollary \ref{fdecomp2} to $\Psi_e$, and obtain
\begin{align*}
\vecv_{\Psi_e,m}(t)&=\sfrac{\wty(h_m)^{1/2-it}}{\sqrt{2\pi}}\sum_{\upsilon\in\ZZ} \langle \Psi_{j,k}^{r,s}(\vece_0\otimes\vece_0), \EE_m^{1/2+it}(\cdot,\vece_{2\upsilon})\rangle_{L^2(\GaG)}\vece_{2\upsilon}\\&=\sfrac{\wty(h_m)^{1/2-it}}{\sqrt{2\pi}}\langle \Psi_{j,k}^{r,s}(\vece_0\otimes\vece_0), \EE_m^{1/2+it}(\cdot,\vece_{0})\rangle_{L^2(\GaG)}\vece_{0}.
\end{align*}
Since $\Psi_g =\rho (g)\Psi_e $ for all $g\in G$, and $\vecv_{\rho(g)f,m}(t)=\pi_m^{1/2+it}(g)\vecv_{f,m}(t)$ for all $f\in L^2(\GaG)$, we have
\begin{equation}\label{vgident}
\vecv_{\Psi_g,m}(t)=\sfrac{\wty(h_m)^{1/2-it}}{\sqrt{2\pi}}\langle \Psi_{j,k}^{r,s}(\vece_0\otimes\vece_0), \EE_m^{1/2+it}(\cdot,\vece_{0})\rangle_{L^2(\GaG)}\pi_m^{1/2+it}(g)\vece_{0} \qquad\forall g\in G. 
\end{equation}
By Proposition \ref{Psancon} and Lemma \ref{Phigb}, both sides of this relation are analytic functions of $g$, and by analytic continuation \eqref{vgident} holds for all $g\in U$. Entering this into \eqref{Psigdecomp}, we obtain
\begin{equation*}
\Psi_g=\sum_{i=1}^{\infty} \Proj_i(\Psi_g)+\sum_{m=1}^{\kappa} \int_0^{\infty}\sfrac{\wty(h_m)^{1/2-it}}{\sqrt{2\pi}} \langle \Psi_{j,k}^{r,s}(\vece_0\otimes\vece_0), \EE_m^{1/2+it}(\cdot,\vece_{0})\rangle_{L^2(\GaG)}\pi_m^{1/2+it}(g)\vece_{0} \,dt
\end{equation*}
for all $g\in U$, and hence
\begin{equation*}
\|\Psi_g\|_{L^2(\GaG)}^2 \geq \sfrac{\wty(h_l)}{2\pi}\int_0^{\infty} |\langle \Psi_{j,k}^{r,s}(\vece_0\otimes\vece_0), \EE_l^{1/2+it}(\cdot,\vece_{0})\rangle_{L^2(\GaG)}|^2\|\pi_l^{1/2+it}(g)\vece_{0}\|_{L^2(K)}^2 \,dt.
\end{equation*}
Restricting our attention to the case $g=g_{\epsilon}$, we let $\epsilon=\frac{1}{T} < \frac{1}{5}$, by Proposition \ref{Psancon} \textit{ii)} we have 
\begin{align*}
\|\Psi_{g_{T^{-1}}}\|_{L^2(\GaG)}^2& \geq \sfrac{\wty(h_l)}{2\pi}\int_0^{T} |\langle \Psi_{j,k}^{r,s}(\vece_0\otimes\vece_0), \EE_l^{1/2+it}(\cdot,\vece_{0})\rangle_{L^2(\GaG)}|^2  c e^{(\pi-12/T)t} \,dt\\&\gg  \int_0^{T} |\langle \Psi_{j,k}^{r,s}(\vece_0\otimes\vece_0), \EE_l^{1/2+it}(\cdot,\vece_{0})\rangle_{L^2(\GaG)}|^2  e^{\pi t} \,dt.
\end{align*}
Observe that both $ \Psi_{j,k}^{r,s}(\vece_0\otimes\vece_0)$ and $\EE_l^{1/2+it}(\cdot,\vece_{0})$ are $\rho(K)$-invariant; we may thus view them as functions on $\GaH$. In fact, we have (see Definition \ref{Psirsdef}, \ref{Psigdef}, and \eqref{Eismapexpr})
\begin{align*}
\big[\Psi_{j,k}^{r,s}(\vece_0\otimes\vece_0)\big](\Gamma n_x a_y)&=\frac{1}{\wty(h_j)^r\wty(h_k)^s}\Phi_{j,k}^{r,s}(x+i y)\\\EE_l^{1/2+it}(\Gamma n_x a_y,\vece_{0})&=\frac{1}{\wty(h_l)^{1/2+it}}E_l(x+i y, \sfrac{1}{2}+it)\qquad\forall x+i y \in \HH,
\end{align*}
and so by Lemma \ref{GaH/GaGint} and Proposition \ref{Phiprops},
\begin{align*}
&\langle \Psi_{j,k}^{r,s}(\vece_0\otimes\vece_0), \EE_l^{1/2+it}(\cdot,\vece_{0})\rangle_{L^2(\GaG)}=\wty(h_j)^{-r}\wty(h_k)^{-s} \wty(h_l)^{it-1/2}\langle \Phi_{j,k}^{r,s},E_l(\cdot, \sfrac{1}{2}+it)\rangle_{L^2(\GaH)}\\&= \wty(h_j)^{-r}\wty(h_k)^{-s} \wty(h_l)^{it-1/2} \,R.N.\left(\int_{\GaH} E_j(z,r)E_k(z,s)\overline{E_l(z, \sfrac{1}{2}+it)}\,d\mu(z) \right).
\end{align*}
This gives
\begin{align*}
\int_0^T \left|R.N.\left(\int_{\GaH} E_j(z,r)E_k(z,s)\overline{E_l(z, \sfrac{1}{2}+it)}\,d\mu(z) \right) \right|^2 e^{\pi t}\,dt \ll \|\Psi_{g_{T^{-1}}}\|_{L^2(\GaG)}^2
\end{align*}
By Corollary \ref{Phiuvbd} and Lemma \ref{gTSobbdd}, for any $\eta>0$,
\begin{equation*}
\|\Psi_{g_{\frac{1}{T}}}\|_{L^2(\GaG)}\ll \scrS_{1+\eta}(\pi^r(g_{T^{-1}})\vece_0)\scrS_{1+\eta}(\pi^s(g_{T^{-1}})\vece_0)\ll T^{2+2\eta}.
\end{equation*}

\hspace{425.5pt}\qedsymbol

\appendix

\section{Bounds for $E_j(z,s,2\upsilon)$}\label{AppA}

\subsection{Whittaker Functions}
We start by recalling the following integral representations of the Whittaker functions $W_{k,\mu}(r)$ (cf. \cite[p.\ 313, p.\ 431]{MOS}):
\begin{prop}\label{WhitInt}
Let $k,r>0$. 
\begin{enumerate}[i)]
\item For $s\in \CC$ with $\Re(s)>0$,
\begin{equation*}
W_{-k,s-\frac{1}{2}}(r)=\frac{r^s e^{-\frac{r}{2}}}{\Gamma(s+k)}\int_0^{\infty} e^{-r u} \left( \frac{u}{u+1}\right)^k u^{s-1} (u+1)^{s-1}\,du.
\end{equation*}
\item For $s\in \CC$ with $\Re(s)>\frac{1}{2}$,
\begin{equation*}
W_{k,s-\frac{1}{2}}(r)=4^{s-1}\pi^{-1}(-1)^{k}\Gamma(s+k) r^{1-s}\int_{-\infty}^{\infty} \left( \frac{1}{1+u^2}\right)^s\left( \frac{u-i}{u+i}\right)^{k} e^{-i\frac{r}{2}u}\,du.
\end{equation*}
\end{enumerate}
\end{prop}
We will now bound the integrals appearing in these formulas. It turns out that it is easier to obtain a satisfactory bound on the integral appearing in case \textit{i)} compared with that in case \textit{ii)}. We therefore start with the easier of the two cases: 
\begin{lem}\label{trivbdd}
For $k,r>0$ and $s\in \CC$ with $\Re(s)=\sigma>0$,
\begin{equation*}
\left|\int_0^{\infty} e^{-r u} \left( \frac{u}{u+1}\right)^k u^{s-1} (u+1)^{s-1}\,du\right|\leq F(r,\sigma),
\end{equation*}
where
\begin{equation*}
F(r,\sigma)= \begin{cases} r^{-\sigma}\Gamma\big( \sigma\big)\quad&\mathrm{if\;} \sigma\leq 1\\
r^{-\sigma}2^{\sigma-1} \big(\Gamma(\sigma)+r^{1-\sigma}\Gamma(2\sigma-1)\big)\quad&\mathrm{if\;} \sigma> 1.\end{cases}
\end{equation*}
\end{lem}
\begin{proof}
\begin{align*}
\left|\int_0^{\infty} e^{-r u} \left( \frac{u}{u+1}\right)^k u^{s-1} (u+1)^{s-1}\,du\right|&\leq \int_0^{\infty} e^{-r u} \left( \frac{u}{u+1}\right)^k u^{\sigma-1} (u+1)^{\sigma-1}\,du\\&\leq \int_0^{\infty} e^{-r u} u^{\sigma-1} (u+1)^{\sigma-1}\,du.
 \end{align*}
For $\sigma\leq 1$, we have
\begin{equation*}
\leq \int_0^{\infty} e^{-r u} u^{\sigma-1} \,du= r^{-\sigma}\Gamma(\sigma).
\end{equation*}
Otherwise (when $\sigma>1$),
\begin{align*}
\leq 2^{\sigma-1}\int_0^{\infty} e^{-r u } \left(u^{\sigma-1}+u^{2\sigma-2}\right)\,du = 2^{\sigma-1}\left(r^{-\sigma}\Gamma(\sigma)+r^{1-2\sigma}\Gamma(2\sigma-1) \right).
\end{align*}
\end{proof}
Turning now to case \textit{ii)} in Proposition \ref{WhitInt}, we define, for $r$, $s$ and $k$ as in the aforementioned proposition, the following function:
\begin{equation*}
I_k(r,s):=\int_{-\infty}^{\infty} \left( \frac{1}{1+u^2}\right)^s\left( \frac{u-i}{u+i}\right)^{k} e^{-i\frac{r}{2}u}\,du=(-1)^k\int_{-\infty}^{\infty} (u^2+1)^{-s} e^{-i(\frac{r}{2}u-2 k \arctan u)}\,du.
\end{equation*}
Note that we can use repeated integration by parts to analytically continue $I_k(r,s)$ in $s$ to all of $\CC$. Indeed, for $r>0$, $k\in\ZZ_{>0}$, and $s\in \CC$ with $\Re(s)>0$, we have
\begin{align}\label{Ikdef2}
&I_k(r,s)=\frac{4(-1)^k}{r}\int_{-\infty}^{\infty} (k+isu)(u^2+1)^{-(s+1)}e^{-i(\frac{r}{2}u-2 k \arctan u)}\,du\\&\notag= \frac{8(-1)^k}{r^2}\int_{-\infty}^{\infty} \big(-(2 s +1)s u^2+ 2 i k(2 s +1)u +(2k^2+s)\big)(u^2+1)^{-(s+2)}e^{-i(\frac{r}{2}u-2 k \arctan u)}\,du.
\end{align}
The triangle inequality then gives the following:
\begin{lem}\label{ibpbdd}
For $r>0$, $k\geq 1$, and $s\in\CC$ with $\Re(s)\geq \frac{1}{4}$, $|I_k(r,s)|\ll \frac{|s|+k}{r}$ and $|I_k(r,s)|\ll \frac{|s|^2+k^2}{r^2}$. The implied constants are absolute.
\end{lem}
We will now use the stationary phase method to obtain more precise bounds for $I_k(r,s)$:
\begin{lem}\label{statphase} For $r>0$, $k\geq 1$, and $s\in\CC$ with $\Re(s)\geq \frac{1}{4}$:
\begin{enumerate}
\item If $r\geq 8 k$, then $|I_k(r,s)|\ll\frac{|s|^2+1}{r^2}$.
\item If $\frac{1}{(\sqrt{3}-\sqrt{2})^4 k}\leq r\leq  k$, then $|I_k(r,s)|\ll \frac{1+|s|^2}{\Re(s)}k^{1/4-\Re(s)}r^{\Re(s)-3/4}+ \frac{|s|^2+1}{r^2}$.
\end{enumerate} 
The implied constants are absolute.
\end{lem}
\begin{proof} We start by assuming that $\Re(s)>\frac{1}{2}$, and let $p(u)=u^2+1$ and $q(u)=\frac{r}{2}u-2 k \arctan u$.

Starting with case \textit{(1)}, we have $q'(u) =\frac{r}{2}(1- \frac{4k/r}{1+u^2})$. By assumption $r\geq 8k\geq 8$, hence $q'(u)\geq q'(0)\geq \frac{r}{4}>0$ for all $u\in\RR$. We now use integration by parts:
\begin{align*}
I_k(r,s)=&(-1)^k\int_{-\infty}^{\infty} \frac{p(u)^{-s}}{-iq'(u)} (-iq'(u)e^{-iq(u)})\,du\\&=(-1)^k\left(\left[\frac{p(u)^{-s}}{-iq'(u)}e^{-iq(u)} \right]_{u=-\infty}^{\infty}- \int_{-\infty}^{\infty}\frac{d}{du}\left\lbrace\frac{p(u)^{-s}}{-iq'(u)} \right\rbrace e^{-iq(u)}\,du\right)\\&=-i(-1)^k \int_{-\infty}^{\infty}\frac{d}{du}\left\lbrace\frac{p(u)^{-s}}{q'(u)} \right\rbrace e^{-iq(u)}\,du
\\&=-i(-1)^k \int_{-\infty}^{\infty}\frac{(-s)p(u)^{-(s+1)}p'(u)q'(u)-p(u)^{-s}q''(u)}{(q'(u))^2}  e^{-iq(u)}\,du
\\&=-i(-1)^k \int_{-\infty}^{\infty}\frac{(-s)p(u)^{-(s+1)}p'(u)q'(u)-p(u)^{-s}q''(u)}{(q'(u))^2}\frac{1}{-iq'(u)}  (-iq'(u)e^{-iq(u)})\,du\\&=(-1)^k\left[\frac{(-s)p(u)^{-(s+1)}p'(u)q'(u)-p(u)^{-s}q''(u)}{(q'(u))^3} e^{-iq(u)}\right]_{u=-\infty}^{\infty}\\&\qquad\qquad-(-1)^k\int_{-\infty}^{\infty} \frac{d}{du}\left \lbrace\frac{(-s)p(u)^{-(s+1)}p'(u)q'(u)-p(u)^{-s}q''(u)}{(q'(u))^3} \right\rbrace e^{-i q(u)}\,du\\&=(-1)^k\int_{-\infty}^{\infty} \frac{d}{du}\left \lbrace\frac{sp(u)^{-(s+1)}p'(u)q'(u)+p(u)^{-s}q''(u)}{(q'(u))^3} \right\rbrace e^{-i q(u)}\,du\\&=(-1)^k\int_{-\infty}^{\infty} \frac{1}{q'(u)^2 p(u)^s}Q(u) e^{-i q(u)}\,du,
\end{align*}
where (suppressing the dependencies on $u$)
\begin{equation*}
Q= s \frac{p''}{p}+\frac{q'''}{q'}-\left( \frac{q''}{q'}\right)^2-s\left( \frac{p'}{p}\right)^2-\left(s\frac{p'}{p}+2\frac{q''}{q'} \right) \left(s\frac{p'}{p}+\frac{q''}{q'} \right).
\end{equation*}
Recalling that $p(u)=1+u^2$, we have
\begin{equation*}
\left| \frac{p'(u)}{p(u)}\right|\leq \frac{2}{\sqrt{1+u^2}}, \;\left| \frac{p''(u)}{p(u)}\right|= \frac{2}{1+u^2}.
\end{equation*}
Recall also that $q'(u)=\frac{r}{2}-\frac{2k}{1+u^2}$. By assumption, $r\geq 8k$, hence $\frac{r}{4k}-1\geq 1$, and thus
\begin{align*}
&\left| \frac{q''(u)}{q'(u)}\right|=\left| \frac{\frac{4ku}{(1+u^2)^2}}{\frac{r}{2}-\frac{2k}{1+u^2}}\right|=\left| \frac{\frac{2u}{1+u^2}}{(\frac{r}{4k}-1)+\frac{r}{4k}u^2}\right|\ll (1+u^2)^{-\frac{3}{2}},\\
&\left| \frac{q'''(u)}{q'(u)}\right|=\left| \frac{\frac{4k(1-3u^2)}{(1+u^2)^3}}{\frac{r}{2}-\frac{2k}{1+u^2}}\right|=\left| \frac{\frac{2(1-3u^2)}{(1+u^2)^2}}{(\frac{r}{4k}-1)+\frac{r}{4k}u^2}\right|\ll (1+u^2)^{-2},
\end{align*}
hence $|Q(u)|\ll\frac{|s|^2+1}{1+u^2} $. Using this bound,  see that 
\begin{equation*}
I_k(r,s)=(-1)^k\int_{-\infty}^{\infty} \frac{1}{q'(u)^2 p(u)^s}Q(u) e^{-i q(u)}\,du
\end{equation*}
for all $s\in\CC$ with $\Re(s)>-\frac{1}{2}$. In particular, for all $s$ with $\Re(s)>0$,
\begin{align*}
|I_k(r,s)|\leq \int_{-\infty}^{\infty} \frac{1}{|q'(u)|^2}|Q(u)| \,du\ll \frac{|s|^2+1}{r^2}.
\end{align*}
Turning to case \textit{(2)}, we once again start by assuming that $\Re(s)>\frac{1}{2}$, and locate the zeroes of $q'(u)$:
\begin{equation*}
q'(u)=0 \Leftrightarrow u^2=\frac{4k}{r}-1.
\end{equation*}
We thus let $u_0=\sqrt{\frac{4k}{r}-1}$, and note that $|q''(\pm u_0)|=\frac{4ku_{0}}{(1+u_0^2)^2}=\frac{r^2\sqrt{\frac{4k}{r}-1}}{4k}$. Using the assumption $r\leq k$, we have
\begin{equation*}
\frac{1}{2}\frac{1}{\sqrt{|q''(\pm u_0)|}}=\frac{\sqrt{k}}{r(\frac{4k}{r}-1)^{1/4}}\leq \frac{\sqrt{k}}{r(\frac{k}{r})^{1/4}}=\frac{k^{1/4}}{r^{3/4}}.
\end{equation*}
Letting $\epsilon=\frac{k^{1/4}}{r^{3/4}}$, observe that $\frac{1}{(\sqrt{3}-\sqrt{2})^4 k}\leq r \leq k$ implies that $u_0-\epsilon\geq \sqrt{\frac{3k}{r}}-\epsilon \geq \sqrt{\frac{2k}{r}}>1$. We now split the integral defining $I_k(r,s)$ into several parts:
\begin{align*}
(-1)^kI_k(r,s) =\left(\int_{-\infty}^{-u_0-\epsilon}+\int_{-u_0-\epsilon}^{-u_0+\epsilon}+\int_{-u_0+\epsilon}^{u_0-\epsilon}+\int_{u_0-\epsilon}^{u_0+\epsilon}+\int_{u_0+\epsilon}^{\infty}\right)&p(u)^{-s}e^{-iq(u)}\,du.
\end{align*}
Note that $s\mapsto \int_{\scrI}p(u)^{-s}e^{-iq(u)}\,du$, where $\scrI$ is any of the three finite-length intervals in the above partition of $\RR$, is a holomorphic function that is well-defined for all $s\in \CC$. We will use integration by parts (similarly to as in case \textit{(1)}) to extend the formulas for the maps $s\mapsto\int_{\scrI}p(u)^{-s}e^{-iq(u)}\,du$, where $\scrJ=(-\infty,-u_0-\epsilon]$ or $[u_0+\epsilon,\infty)$, to $\lbrace s\in\CC\,:\, \Re(s)>0\rbrace$. Note that $p(-u)=p(u)$ and $q(-u)=-q(u)$; these relations will show that the bounds we obtain on the integrals $\int_{u_0-\epsilon}^{u_0+\epsilon}p(u)^{-s}e^{-iq(u)}\,du$ and (the extension of) $\int_{u_0+\epsilon}^{\infty}p(u)^{-s}e^{-iq(u)}\,du$ will also hold for $\int_{-u_0-\epsilon}^{\epsilon-u_0}p(u)^{-s}e^{-iq(u)}\,du$ and (the extension of) $\int_{-\infty}^{-u_0-\epsilon}p(u)^{-s}e^{-iq(u)}\,du$, respectively. We thus need only consider the integrals over the intervals $[\epsilon-u_0,u_0-\epsilon]$, $[u_0-\epsilon,u_0+\epsilon]$, and $[u_0+\epsilon,\infty)$.

Starting with the integral $\int_{u_0-\epsilon}^{u_0+\epsilon}p(u)^{-s}e^{-iq(u)}\,du$, we note that $u\mapsto |p(u)^{-s}|=p(u)^{-\Re(s)}$ is non-increasing on $\RR_{\geq 0}$, hence
\begin{align}\label{ibp8}
\left|\int_{u_0-\epsilon}^{u_0+\epsilon}p(u)^{-s}e^{-iq(u)}\,du\right| &\leq 2\epsilon\cdot p(u_0-\epsilon)^{-\Re(s)} \leq 2\epsilon\cdot p\left(\sqrt{\sfrac{2k}{r}}\right)^{-\Re(s)}\\\notag&=2 k^{1/4} r^{-3/4}\left(1+\sfrac{2k}{r}\right)^{-\Re(s)} \ll 2^{-\Re(s)}k^{1/4-\Re(s)} r^{\Re(s)-3/4}.
\end{align}

We now consider $\int_{u_0+\epsilon}^{\infty}p(u)^{-s}e^{-iq(u)}\,du$. Integration by parts gives
\begin{align}\label{ibp1}
\int_{u_0+\epsilon}^{\infty}p(u)^{-s}e^{-iq(u_0+\epsilon)}\,du=\frac{p(u_0+\epsilon)^{-s}}{i q'(u_0+\epsilon)}e^{-iq(u)}+i \int_{u_0+\epsilon}^{\infty} \frac{p(u)^{-s}}{q'(u)}\left( \frac{2 s u}{1+u^2}+\frac{q''(u)}{q'(u)}\right)e^{-iq(u)}du.
\end{align}
Note that the expression in the right-hand side is well-defined and holomorphic for all $s\in \CC$ with $\Re(s)>0$. 
Observe now that $u_0>1$, and $q''(u)=\frac{4 k u}{(1+u^2)^2}>0$ is non-increasing on $\RR_{\geq 1}$; for $u\geq u_0$ we therefore have
\begin{align}\label{ftc}
|q'(u)|=\left| \int_{u_0}^{u}q''(t)\,dt\right|\geq (u-u_0) q''(u).
\end{align}
Since $u_0+\epsilon\leq 2u_0$, for $u\geq u_0+\epsilon$, we have
\begin{align}\label{q'bdd}
\left|\frac{1}{ q'(u)} \right|\leq \left|\frac{1}{ q'(u_0+\epsilon)} \right|\leq&\frac{1}{\epsilon q''(2u_0)}=\frac{1}{\frac{k^{1/4}}{r^{3/4}}\frac{8k\sqrt{\frac{4k}{r}-1}}{(1+4(\frac{4k}{r}-1))^2}}\\\notag&=k^{-5/4}r^{3/4}\frac{(\frac{16k}{r}-3)^2}{8\sqrt{\frac{4k}{r}-1}}\ll \frac{ k^{3/4} r^{-5/4}}{\sqrt{\frac{3k}{r}}}\ll k^{1/4} r^{-3/4}.
\end{align}
Note also that $|p(u)|^{-\Re(s)}$ is non-increasing on $\RR_{\geq 0}$, hence
\begin{equation}\label{u0bdd}
\left|\frac{p(u_0+\epsilon)^{-s}}{ q'(u_0+\epsilon)} \right|\ll k^{1/4}r^{-3/4} p(u_0)^{-\Re(s)}=4^{-\Re(s)}k^{1/4-\Re(s)}r^{\Re(s)-3/4} .
\end{equation}
We also use the bound $\frac{1}{|q'(u)|}\ll k^{1/4} r^{-3/4}$ when bounding the integral in \eqref{ibp1}. Firstly, 
\begin{align}\label{ibp2}
\left|\int_{u_0+\epsilon}^{\infty} \!\frac{p(u)^{-s}}{q'(u)} \frac{2 s u}{1+u^2}e^{-iq(u)}\,du\right|&\leq \frac{|s|}{|q'(u_0+\epsilon)|}\int_{u_0+\epsilon}^{\infty}  p(u)^{-(1+\Re(s))}2 u\,du\\\notag=\frac{|s|}{|q'(u_0+\epsilon)|}& \frac{p(u_0+\epsilon)^{-\Re(s)}}{\Re(s)}\ll \frac{ |s|k^{1/4-\Re(s)}r^{\Re(s)-3/4}}{4^{\Re(s)}\Re(s)}.
\end{align}
To bound the second term in the integral in \eqref{ibp1}, we observe that for $u\geq \sqrt{2\cdot\frac{4k}{r}-1}$,
\begin{equation*}
\frac{q''(u)}{q'(u)}=\frac{2u}{1+u^2} \frac{\frac{2 k}{1+u^2}}{\frac{r}{2}-\frac{2 k}{1+u^2}}=\frac{2u}{1+u^2} \frac{1}{\frac{r}{4k}(1+u^2)-1}\leq \frac{2 u}{1+u^2},
\end{equation*}
hence
\begin{align*}
&\left|\int_{u_0+\epsilon}^{\infty} p(u)^{-s}\frac{q''(u)}{q'(u)^2}e^{-iq(u)}\,du\right|\leq \int_{u_0+\epsilon}^{\sqrt{2\cdot\frac{4k}{r}}} p(u)^{-\Re(s)}\frac{q''(u)}{q'(u)^2}\,du+\int_{\sqrt{2\cdot\frac{4k}{r}}}^{\infty} p(u)^{-\Re(s)}\frac{q''(u)}{q'(u)^2}\,du
\\&\qquad\qquad\qquad\qquad\qquad\leq p(u_0+\epsilon)^{-\Re(s)} \left[\frac{-1}{q'(u)} \right]_{u=u_0+\epsilon}^{u=\sqrt{2\cdot\frac{4k}{r}}}+\int_{u_0+\epsilon}^{\infty} \frac{p(u)^{-\Re(s)}}{q'(u)} \frac{2  u}{1+u^2}\,du
\\&\qquad\qquad\qquad\qquad\qquad\qquad\qquad\leq \frac{p(u_0+\epsilon)^{-\Re(s)}}{q'(u_0+\epsilon)}+\frac{1}{q'(u_0+\epsilon)}\int_{u_0+\epsilon}^{\infty} p(u)^{-(\Re(s)+1)}2 u \,du.
\end{align*}
Thus, by \eqref{u0bdd} and \eqref{ibp2},
\begin{equation}\label{ibp3}
\left|\int_{u_0+\epsilon}^{\infty} p(u)^{-s}\frac{q''(u)}{q'(u)^2}e^{-iq(u)}\,du\right|\ll 4^{-\Re(s)} k^{1/4-\Re(s)} r^{\Re(s)-3/4}.
\end{equation}
Now combining \eqref{u0bdd}, \eqref{ibp2}, and \eqref{ibp3} permits us to bound the right-hand side of \eqref{ibp1} by 
\begin{align}\label{I2bdd}
&\left|\frac{p(u_0+\epsilon)^{-s}}{i q'(u_0+\epsilon)}e^{-iq(u_0+\epsilon)}\!+i\! \int_{u_0+\epsilon}^{\infty} \!\!\frac{p(u)^{-s}}{q'(u)}\!\left( \frac{2 s u}{1+u^2}\!+\!\frac{q''(u)}{q'(u)}\right)e^{-iq(u)}du\right|\!\\\notag&\qquad\qquad\qquad\qquad\qquad\qquad\qquad\ll\!\frac{ |s|  k^{1/4-\Re(s)} r^{\Re(s)-3/4}}{4^{\Re(s)}\Re(s)}.
\end{align}
It remains to bound the integral $\int_{\epsilon-u_0}^{u_0-\epsilon}p(u)^{-s}e^{-iq(u)}\,du$. Integration by parts yields
\begin{equation}\label{ibp7}
\int_{\epsilon-u_0}^{u_0-\epsilon}p(u)^{-s}e^{-iq(u)}\,du=\left[\frac{p(u)^{-s}e^{-iq(u)}}{-i q'(u)}\right]_{u=\epsilon-u_0}^{u=u_0-\epsilon}\!\!\!\!+i \int_{\epsilon-u_0}^{u_0-\epsilon} \frac{p(u)^{-s}}{q'(u)}\left( \frac{2 s u}{1+u^2}+\frac{q''(u)}{q'(u)}\right)e^{-iq(u)}du.
\end{equation}
Arguing as in \eqref{ftc} and \eqref{q'bdd}, $\frac{1}{|q'(\pm(u_0-\epsilon))|}\ll k^{1/4}r^{-3/4}$. Furthermore, $u_0-\epsilon\geq \sqrt{\frac{2 k}{r}}$, hence $|p(\pm(u_0-\epsilon))^{-s}|\leq p\big(\sqrt{\frac{2 k}{r}}\big)^{-\Re(s)}\leq 2^{-\Re(s)}k^{-\Re(s)} r^{\Re(s)}$. We thus obtain
\begin{equation}\label{ibp6}
\left| \left[\frac{p(u)^{-s}e^{-iq(u)}}{-i q'(u)}\right]_{u=\epsilon-u_0}^{u=u_0-\epsilon}\right|\ll 2^{-\Re(s)} k^{1/4-\Re(s)}r^{\Re(s)-3/4}.
\end{equation}
Noting now that if $|u|\leq \sqrt{\frac{2k}{r}-1}$ then $\frac{1}{|q'(u)|}\leq \frac{2}{r}$, we first split the integral obtained from integrating by parts as follows:
\begin{align*}
&\int_{\epsilon-u_0}^{u_0-\epsilon}\frac{p(u)^{-s}}{q'(u)}\left( \frac{2 s u}{1+u^2}+\frac{q''(u)}{q'(u)}\right)e^{-iq(u)}\,du\\&\qquad\quad\quad=\left(\int_{\epsilon-u_0}^{-\sqrt{\frac{2k}{r}-1}}+\int_{-\sqrt{\frac{2k}{r}-1}}^{\sqrt{\frac{2k}{r}-1}}+\int_{\sqrt{\frac{2k}{r}-1}}^{u_0-\epsilon}\right)\frac{p(u)^{-s}}{q'(u)}\left( \frac{2 s u}{1+u^2}+\frac{q''(u)}{q'(u)}\right)e^{-iq(u)}\,du,
\end{align*} 
and use integration by parts once more on the integral over $\left[-\sqrt{\frac{2k}{r}-1},\sqrt{\frac{2k}{r}-1}\right]$:
\begin{align*}
&\int_{-\sqrt{\frac{2k}{r}-1}}^{\sqrt{\frac{2k}{r}-1}}\frac{p(u)^{-s}}{q'(u)}\left( \frac{2 s u}{1+u^2}+\frac{q''(u)}{q'(u)}\right)e^{-iq(u)}\,du\\&\qquad\qquad\qquad\qquad=\left[\frac{p(u)^{-s}}{-iq'(u)^2}\left( \frac{2 s u}{1+u^2}+\frac{q''(u)}{q'(u)}\right)e^{-iq(u)}\right]_{u=-\sqrt{\frac{2k}{r}-1}}^{u=\sqrt{\frac{2k}{r}-1}}\\&\qquad\qquad\qquad\qquad\qquad\qquad-i\int_{-\sqrt{\frac{2k}{r}-1}}^{\sqrt{\frac{2k}{r}-1}} \frac{d}{du}\left\lbrace \frac{p(u)^{-s}}{q'(u)^2}\left( \frac{2 s u}{1+u^2}+\frac{q''(u)}{q'(u)}\right)\right\rbrace e^{-iq(u)}\,du.
\end{align*}  
As in case \textit{(1)},
\begin{equation*}
\int_{-\sqrt{\frac{2k}{r}-1}}^{\sqrt{\frac{2k}{r}-1}} \frac{d}{du}\left\lbrace \frac{p(u)^{-s}}{q'(u)^2}\left( \frac{2 s u}{1+u^2}+\frac{q''(u)}{q'(u)}\right)\right\rbrace e^{-iq(u)}\,du=\int_{-\sqrt{\frac{2k}{r}-1}}^{\sqrt{\frac{2k}{r}-1}} \frac{p(u)^{-s}}{q'(u)^2}Q(u) e^{-iq(u)}\,du,
\end{equation*}
where $Q= s \frac{p''}{p}+\frac{q'''}{q'}-\left( \frac{q''}{q'}\right)^2-s\left( \frac{p'}{p}\right)^2-\left(s\frac{p'}{p}+2\frac{q''}{q'} \right) \left(s\frac{p'}{p}+\frac{q''}{q'} \right)$. As previously, $\left| \frac{p'(u)}{p(u)}\right|\leq \frac{2}{\sqrt{1+u^2}}$ and $\left| \frac{p''(u)}{p(u)}\right|= \frac{2}{1+u^2}$. Moreover, $|u|\leq \sqrt{\frac{2k}{r}-1}$ implies that 
\begin{align*}
&\left| \frac{q''(u)}{q'(u)}\right|=\left| \frac{\frac{4ku}{(1+u^2)^2}}{\frac{r}{2}-\frac{2k}{1+u^2}}\right|= \frac{\frac{4k|u|}{(1+u^2)^2}}{\frac{2k}{1+u^2}-\frac{r}{2}}= \frac{\frac{2|u|}{1+u^2}}{1-\frac{r}{4k}(1+u^2)}\ll \frac{1}{\sqrt{1+u^2}},\\
&\left| \frac{q'''(u)}{q'(u)}\right|=\left| \frac{\frac{4k(1-3u^2)}{(1+u^2)^3}}{\frac{r}{2}-\frac{2k}{1+u^2}}\right|= \frac{\frac{2|1-3u^2|}{(1+u^2)^2}}{1-\frac{r}{4k}(1+u^2)}\ll \frac{1}{1+u^2},
\end{align*}
hence 
\begin{align*}\left| \frac{p(u)^{-s}}{q'(u)^2}Q(u) e^{-iq(u)} \right|\leq\frac{4}{r^2}(1+u^2)^{-\Re(s)}\bigg( 2|s|(1+u^2)^{-1}+12(1+u^2)^{-1}+16(1+u^2)^{-1}\\+4|s|(1+u^2)^{-1}+8|s|^2(1+u^2)^{-1}+128(1+u^2)^{-1}\bigg)\\\ll\frac{|s|^2+1}{r^2}\cdot\frac{1}{1+u^2}, 
\end{align*}
giving
\begin{equation}\label{ibp4}
\left|\int_{-\sqrt{\frac{2k}{r}-1}}^{\sqrt{\frac{2k}{r}-1}} \frac{d}{du}\left\lbrace \frac{p(u)^{-s}}{q'(u)^2}\left( \frac{2 s u}{1+u^2}+\frac{q''(u)}{q'(u)}\right)\right\rbrace e^{-iq(u)}\,du\right|\ll \frac{|s|^2+1}{r^2}.
\end{equation}
Combining \eqref{ibp4} with
\begin{align*}
&\left| \left[\frac{p(u)^{-s}}{-iq'(u)^2}\left( \frac{2 s u}{1+u^2}+\frac{q''(u)}{q'(u)}\right)e^{-iq(u)}\right]_{u=\sqrt{-\frac{2k}{r}-1}}^{u=\sqrt{\frac{2k}{r}-1}}\right|\\&\qquad\qquad\qquad\leq  \frac{2p(\sqrt{\frac{2k}{r}-1})^{-\Re(s)}}{q'(\sqrt{\frac{2k}{r}-1})^2}\left( \frac{2 |s| \sqrt{\frac{2k}{r}-1}}{1+(\sqrt{\frac{2k}{r}-1})^2}+\frac{|q''(\sqrt{\frac{2k}{r}-1})|}{|q'(\sqrt{\frac{2k}{r}-1})|}\right) \\&\quad\quad\qquad\qquad\qquad\ll 2^{-\Re(s)}|s| k^{-(\Re(s)+1/2)}r^{\Re(s)-3/2}\leq  2^{-\Re(s)}|s|r^{-2}
\end{align*}
yields
\begin{align}\label{ibp5}
&\left|\int_{-\sqrt{\frac{2k}{r}-1}}^{\sqrt{\frac{2k}{r}-1}}\frac{p(u)^{-s}}{q'(u)}\left( \frac{2 s u}{1+u^2}+\frac{q''(u)}{q'(u)}\right)e^{-iq(u)}\,du\right|\ll \frac{|s|^2+|s|2^{-\Re(s)}+1}{r^2}.
\end{align}
The two integrals that remain are $\int_{\scrA}\frac{p(u)^{-s}}{q'(u)}\left( \frac{2 s u}{1+u^2}+\frac{q''(u)}{q'(u)}\right)e^{-iq(u)}\,du$, where $\scrA$ is either of the two intervals $\left[\epsilon-u_0,-\sqrt{\frac{2k}{r}-1}\right]$ or $\left[\sqrt{\frac{2k}{r}-1},u_0-\epsilon\right]$. These integrals may be bounded in a symmetric way; for brevity, we consider only $\scrA=\left[\sqrt{\frac{2k}{r}-1},u_0-\epsilon\right]$:
\begin{align*}
&\left|\int_{\sqrt{\frac{2k}{r}-1}}^{u_0-\epsilon}\frac{p(u)^{-s}}{q'(u)}\left( \frac{2 s u}{1+u^2}+\frac{q''(u)}{q'(u)}\right)e^{-iq(u)}\,du\right|\\&\qquad\leq \int_{\sqrt{\frac{2k}{r}-1}}^{u_0-\epsilon} \left(\frac{|s|}{|q'(u_0-\epsilon)|} p(u)^{-(\Re(s)+1)}2 u+ p\big(\sqrt{\sfrac{2k}{r}-1}\big)^{-\Re(s)} \frac{q''(u)}{q'(u)^2} \right)\,du
\\&\qquad= \left[\frac{|s|}{|q'(u_0-\epsilon)|} \frac{-p(u)^{-\Re(s)}}{\Re(s)}+ p\big(\sqrt{\sfrac{2k}{r}-1}\big)^{-\Re(s)} \frac{-1}{q'(u)}  \right]_{u=\sqrt{\frac{2k}{r}-1}}^{u=u_0-\epsilon}\\&\qquad\leq \left(1+\frac{|s|}{\Re(s)} \right)\frac{p\big(\sqrt{\sfrac{2k}{r}-1}\big)^{-\Re(s)}}{|q'(u_0-\epsilon)|}\leq \left(1+\frac{|s|}{\Re(s)} \right)\frac{(\frac{2k}{r})^{-\Re(s)}}{\epsilon q''(u_0)}\\& \qquad= \left(1+\frac{|s|}{\Re(s)} \right)\frac{2^{-\Re(s)}k^{-\Re(s)}r^{\Re(s)}}{\frac{k^{1/4}}{r^{3/4}}\frac{4k\sqrt{\frac{4k}{r}-1}}{(\frac{4k}{r})^2}}\leq \left(1+\frac{|s|}{\Re(s)} \right) \frac{4}{\sqrt{3}}2^{-\Re(s)}k^{1/4-\Re(s)}r^{\Re(s)-3/4}.
\end{align*}
Entering this bound, together with \eqref{ibp6} and \eqref{ibp5}, into \eqref{ibp7} gives
\begin{equation*}
\left|\int_{\epsilon-u_0}^{u_0-\epsilon}p(u)^{-s}e^{-iq(u)}\,du \right|\ll \frac{|s|2^{-\Re(s)}}{\Re(s)} k^{1/4-\Re(s)}r^{\Re(s)-3/4} +\frac{ |s|^2+1}{r^2}.
\end{equation*}
This, combined with \eqref{ibp8} and \eqref{I2bdd}, completes the proof.
\end{proof}
When $k<r<8k$, further subdivisons are required to obtain adequate bounds on $I_k(r,s)$:

\begin{lem}\label{statphase1} For $r>0$, $k\geq 52$, and $s\in\CC$ with $\Re(s)\geq \frac{1}{4}$:
\begin{enumerate}
\item If $4k(1+k^{-2/3})\leq r\leq  8k$, then $|I_k(r,s)|\ll \frac{1+|s|}{r-4k}$.
\item If $4k\leq r\leq  4k(1+k^{-2/3})$, then $|I_k(r,s)|\ll\frac{1+|s|}{k^{1/3}}$.
\item If $\frac{4k}{1+k^{-2/3}}\leq r\leq  4k$, then $|I_k(r,s)|\ll\frac{1+|s|}{k^{1/3}}$.
\item If $k\leq r\leq \frac{4k}{1+k^{-2/3}}$, then $|I_k(r,s)|\ll\frac{1+|s|}{\left(k\sqrt{\frac{4k}{r}-1}\right)^{1/2}}$.
\end{enumerate} 
The implied constants are absolute.
\end{lem}
\begin{proof}
Starting with \textit{(1)} and \textit{(2)}, we write $r=4k(1+\eta)$, where $\eta\in[0,1]$. As in the proof of Lemma \ref{statphase}, we first assume that $\Re(s)>\frac{1}{2}$, hence
\begin{align*}
I_k(r,s)=(-1)^k\int_{-\infty}^{\infty} p(u)^{-s} e^{-iq(u)}\,du,
\end{align*}
where $p(u)=1+u^2$ and $q(u)=\frac{r}{2}u-2k\arctan u$. We now split this integral into three parts; let $A>0$. Then
\begin{equation*}
(-1)^k I_k(r,s)=\left( \int_{-\infty}^{-A} +\int_{-A}^{A} +\int_{A}^{\infty} \right)p(u) ^{-s}e^{-i q(u)}\,du.
\end{equation*} 
Using integration by parts yields
\begin{align*}
\int_{A}^{\infty} p(u) e^{-i q(u)}\,du=\frac{-i p(A)^{-s}e^{-iq(A)}}{q'(A)}-i\int_{A}^{\infty} \frac{e^{-iq(u)}}{q'(u)}\left(\frac{-s p'(u)}{p(u)^{s+1}}-\frac{1}{p(u)^{s}}\cdot\frac{q''(u)}{q'(u)} \right)\,du
\end{align*}
and
\begin{align*}
\int_{-\infty}^{-A} p(u) e^{-i q(u)}\,du=\frac{i p(-A)^{-s}e^{-iq(-A)}}{q'(-A)}-i\int_{-\infty}^{-A} \frac{e^{-iq(u)}}{q'(u)}\left(\frac{-s p'(u)}{p(u)^{s+1}}-\frac{1}{p(u)^{s}}\cdot\frac{q''(u)}{q'(u)} \right)\,du.
\end{align*}
Observe that for $s$ with $\Re(s)\geq \frac{1}{4}$, $\left|\frac{p'(u)}{p(u)^{s+1}}\right|\leq \frac{2}{(1+u^2)^{3/4}}$. Also,
\begin{equation}\label{q''/q'}
\left| \frac{q''(u)}{q'(u)}\right|=\left| \frac{\frac{4 k u}{(1+u^2)^2}}{\frac{r}{2}-\frac{2 k }{1 + u^2}}\right|=\left|\frac{\frac{2 u }{1+u^2}}{\eta+(1+\eta)u^2}\right|,
\end{equation}
hence the integrals from $\pm A$ to $\pm\infty$ are absolutely convergent. This allows us to express $I_k(r,s)$ as
\begin{align}\label{ibp}
(-1)^k& I_k(r,s)=\int_{-A}^{A}p(u) ^{-s}e^{-i q(u)}\,du+\frac{i p(-A)^{-s}e^{-iq(-A)}}{q'(-A)}-\frac{i p(A)^{-s}e^{-iq(A)}}{q'(A)}\\\notag&+is \left(\int_{-\infty}^{-A}+\int_{A}^{\infty}\right)\frac{p'(u)e^{-iq(u)}}{q'(u) p(u)^{s+1}}\,du+i \left(\int_{-\infty}^{-A}+\int_{A}^{\infty}\right)\frac{e^{-iq(u)}}{p(u)^{s}}\cdot\frac{q''(u)}{q'(u)^2}\,du
\end{align}
for all $s$ with $\Re(s)\geq \frac{1}{4}$. We now bound the various terms occurring in this expression: firstly,
\begin{equation*}
\left| \int_{-A}^{A}p(u) ^{-s}e^{-i q(u)}\,du\right| \leq 2 A
\end{equation*}
and
\begin{equation*}
\left|\frac{ p(\pm A)^{-s}e^{-iq(\pm A)}}{q'( \pm A)} \right| \leq \frac{1}{q'(A)}=\frac{1}{\frac{r}{2}-\frac{2k}{1+A^2}}=\frac{1}{2k}\cdot \frac{1}{1+\eta-\frac{1}{1+A^2}}\leq \frac{1}{k}\cdot \frac{1}{A^2+A^2\eta+\eta}.
\end{equation*}
We then bound the integrals from $\pm A$ to $\pm \infty$ (for the sake of simplicity, we only show the calculations for the integrals from $A$ to $\infty$; the same bounds hold for the integrals over $(-\infty,-A]$):
\begin{align*}
\left|\int_A^{\infty}\frac{p'(u)e^{-i q(u)}}{q'(u) p(u)^{s+1}}\,du  \right|\leq \frac{1}{q'(A)}\int_A^{\infty}  \frac{2}{(1+u^2)^{3/4}} \,du\ll \frac{1}{q'(A)}\leq \frac{1}{k}\cdot \frac{1}{A^2+A^2\eta+\eta}
\end{align*}
and by \eqref{q''/q'} (and assuming that $A\leq 1$),
\begin{align*}
\left|\int_A^{\infty}\frac{e^{-i q(u)}}{p(u)^{s}}\cdot\frac{q''(u)}{q'(u)^2}\,du \right|\leq &\int_A^1 \frac{q''(u)}{q'(u)^2}\,du+\int_1^{\infty} \frac{1}{q'(u)}\cdot \frac{1}{(1+u^2)^{1/4}}\cdot\frac{2 u}{1+u^2}\,du\\&\leq \left[ \frac{-1}{q'(u)} \right]_{u=A}^{u=1}+\frac{2}{q'(A)}\int_1^{\infty}\frac{1}{(1+u^2)^{3/4}}\,du \\&\ll\frac{1}{q'(A)}\leq  \frac{1}{k}\cdot \frac{1}{A^2+A^2\eta+\eta}.
\end{align*}
Combining these bounds gives
\begin{equation*}
|I_k(r,s)|\ll(1+|s|)\left( A +  \frac{1}{k}\cdot \frac{1}{A^2+A^2\eta+\eta}\right).
\end{equation*}
This bound is optimized by choosing $A=\frac{1}{k\eta}$ if $\eta \geq k^{-2/3}$ and $A=k^{-1/3}$ if $0\leq\eta\leq k^{-2/3}$, proving \textit{(1)} and \textit{(2)}.

Turning to \textit{(3)} and \textit{(4)}, we now write $r=\frac{4k}{1+\eta}$, where $\eta\in [0,3]$. Note that $q'(\pm \sqrt{\eta})=0$, $q'(u)$ is increasing on $[\sqrt{\eta},\infty)$, decreasing on $(-\infty,-\sqrt{\eta}]$, and $q'(\pm 2)\geq \frac{k}{10}$. Using these facts in \eqref{ibp} with $A=2$ gives
\begin{align*}
| I_k(r,s)|=&\bigg|\int_{-2}^{2}p(u) ^{-s}e^{-i q(u)}\,du+\frac{i p(-2)^{-s}}{q'(-2)}-\frac{i p(2)^{-s}}{q'(2)}\\\notag&+is \left(\int_{-\infty}^{-2}+\int_{2}^{\infty}\right)\frac{e^{-i q(u)} p'(u)}{q'(u) p(u)^{s+1}}\,du+i \left(\int_{-\infty}^{-2}+\int_{2}^{\infty}\right)\frac{e^{-i q(u)}}{p(u)^{s}}\cdot\frac{q''(u)}{q'(u)^2}\,du\bigg|\\&\ll\bigg|\int_{-2}^{2}p(u) ^{-s}e^{-i q(u)}\,du\bigg|+\frac{1+|s|}{k}.
\end{align*}
Considering first case \textit{(3)}, note that since $k\geq 52$, $\sqrt{\eta}+k^{-1/3}\leq 2$. Now,
\begin{align*}
\int_{-2}^{2}p(u) ^{-s}e^{-i q(u)}\,du=\left( \int_{-2}^{-(\sqrt{\eta}+k^{-1/3}}+\int_{-(\sqrt{\eta}+k^{-1/3})}^{\sqrt{\eta}+k^{-1/3}}+\int_{\sqrt{\eta}+k^{-1/3}}^{2}\right)p(u) ^{-s}e^{-i q(u)}\,du,
\end{align*}
and
\begin{equation*}
\left|\int_{-(\sqrt{\eta}+k^{-1/3})}^{\sqrt{\eta}+k^{-1/3}}p(u) ^{-s}e^{-i q(u)}\,du \right| \leq 2(\sqrt{\eta}+k^{-1/3})\leq 4 k^{-1/3}.
\end{equation*}
Using integration by parts as before yields
\begin{align*}
\left|\int_{\sqrt{\eta}+k^{-1/3}}^{2} p(u) ^{-s}e^{-i q(u)}\,du\right|\ll& \frac{1+|s|}{q'(\sqrt{\eta}+k^{-1/3})}=\frac{1+|s|}{\frac{r}{2}-\frac{2k}{1+(\sqrt{\eta}+k^{-1/3})^2}}\ll \frac{1}{k}\cdot\frac{1+|s|}{k^{-2/3}+2\sqrt{\eta}k^{-1/3}}\\&\leq \frac{1+|s|}{k^{1/3}},
\end{align*}
hence $|I_k(r,s)|\ll \frac{1+|s|}{k^{1/3}}$.

Finally, for case \textit{(4)}, we let $B\leq k^{-1/3}$, and have
\begin{align*}
\int_{0}^{2}p(u) ^{-s}e^{-i q(u)}\,du=\left( \int_{0}^{\sqrt{\eta}-B}+\int_{\sqrt{\eta}-B}^{\sqrt{\eta}+B}+\int_{\sqrt{\eta}+B}^{2}\right)p(u) ^{-s}e^{-i q(u)}\,du,
\end{align*}
(the integral from $-2$ to $0$ is again computed in a completely analogous way). As in previous cases, we obtain the bound
\begin{equation*}
|I_k(r,s)| \ll (1+|s|)\left( B+\frac{1}{|q'(\sqrt{\eta}-B)|}+\frac{1}{q'(\sqrt{\eta}+B)} \right).
\end{equation*}
Observing that since $B\leq k^{-1/3}\leq \sqrt{\eta}$, we have
\begin{equation*}
\frac{1}{|q'(\sqrt{\eta}\pm B)|}=\frac{1}{2k}\cdot\frac{1}{|\frac{1}{1+\eta}-\frac{1}{1+(\sqrt{\eta}\pm B)^2}|}\ll \frac{1}{k}\cdot\frac{1}{|B^2\pm 2 \sqrt{\eta}B|}\leq \frac{1}{kB\sqrt{\eta}}.
\end{equation*}
Then choosing $B=(k\sqrt{\eta})^{-1/2}$ completes the proof.
\end{proof}

\subsection{Proof of Lemma \ref{Ej2ubdd}}\label{Ej2ubddproof}
The Fourier decomposition of $E_j(z,s,2\upsilon)$ at a cusp $\eta_k$ reads 
\begin{align*}
E_j(z,s,2\upsilon)=&\delta_{j,k}y_k^s+\mathfrak{i}_{s,2\upsilon} \varphi_{j,k}(s) y_k^{1-s}\\\notag&\quad+\sum_{\underset{m\neq 0}{m\in\ZZ}}\frac{(-1)^{\upsilon}\Gamma(s)}{2\Gamma(s-\upsilon\cdot\sgn(m))}\frac{\psi_{m,k}^{(j)}(s)}{\sqrt{|m|}}W_{-\upsilon\cdot\sgn(m),s-\frac{1}{2}}(4\pi |m|y_k)e(m x_k).
\end{align*}
(see \eqref{Ej2uFour}). We thus need to bound
\begin{equation*}
\sum_{\underset{m\neq 0}{m\in\ZZ}}\frac{(-1)^{\upsilon}\Gamma(s)}{2\Gamma(s-\upsilon\cdot\sgn(m))}\frac{\psi_{m,k}^{(j)}(s)}{\sqrt{|m|}}W_{-\upsilon\cdot\sgn(m),s-\frac{1}{2}}(4\pi |m|y_k)e(m x_k).
\end{equation*}
From now on we assume that $\upsilon>0$ (the proof in the case $\upsilon<0$ is completely similar) and split the preceding sum into two parts:
\begin{align*}
\sum_{m=-\infty}^{-1}\frac{(-1)^{\upsilon}\Gamma(s)}{2\Gamma(s+\upsilon)}&\frac{\psi_{m,k}^{(j)}(s)}{\sqrt{|m|}}W_{\upsilon,s-\frac{1}{2}}(4\pi |m|y_k)e(m x_k)\\&+\sum_{m=1}^{\infty}\frac{(-1)^{\upsilon}\Gamma(s)}{2\Gamma(s-\upsilon)}\frac{\psi_{m,k}^{(j)}(s)}{\sqrt{m}}W_{-\upsilon,s-\frac{1}{2}}(4\pi m y_k)e(m x_k).
\end{align*}
We now use Proposition \ref{WhitInt} \textit{i)} to rewrite the second sum as follows:
\begin{align*}
\sum_{m=1}^{\infty}\frac{(-1)^{\upsilon}\Gamma(s)}{2\Gamma(s-\upsilon)}\frac{\psi_{m,k}^{(j)}(s)}{\sqrt{m}}e(m x_k)&\frac{(4\pi my_k)^s e^{-2\pi |m|y_k}}{\Gamma(s+\upsilon)}\\&\times\int_0^{\infty} e^{-4\pi |m|y_k u} \left( \frac{u}{u+1}\right)^{\upsilon} u^{s-1} (u+1)^{s-1}\,du,
\end{align*}
hence
\begin{align*}
&\left|\sum_{m=1}^{\infty}\frac{(-1)^{\upsilon}\Gamma(s)}{2\Gamma(s-\upsilon)}\frac{\psi_{m,k}^{(j)}(s)}{\sqrt{m}}W_{-\upsilon,s-\frac{1}{2}}(4\pi m y_k)e(m x_k)\right|\\&\qquad\leq \frac{(4\pi y_k)^{\Re(s)}|\mathfrak{i}_{s,2\upsilon}|}{|\Gamma(s)|}\sum_{m=1}^{\infty}|\psi_{m,k}^{(j)}(s)| m^{\Re(s)-1/2}  e^{-2\pi m y_k} F(4\pi m y_k,s),
\end{align*}
where $F(4\pi m y_k,s)$ is as in Lemma \ref{trivbdd}. For $\Re(s)\geq \frac{1}{4}$, and $y_k\gg 1$, Lemma \ref{trivbdd} gives $F(4\pi m y_k,s)\ll_s (4\pi m y_k)^{-\Re(s)}$, with the implied constant depending continuously on $s$. (For the remainder of the proof, when writing $\ll_s$, we mean that the implied constant is uniformly bounded for $s$ in compact subsets of $\lbrace z\in\CC\,:\,\Re(z)> \frac{1}{4}\rbrace\setminus\scrE$.)

Using \eqref{RanSelbd} and the reflection identity for the column-vector Eisenstein series and the scattering matrix
\begin{equation*}
\big(E_j(z,s) \big)_{j=1,\ldots,\kappa}=\big( \varphi_{i,l}(s)\big)_{1\leq i,l,\leq\kappa}\big(E_j(z,1-s) \big)_{j=1,\ldots,\kappa}
\end{equation*}
(valid for $s,1-s\not\in \scrE$), together with \eqref{RanSelbd} yields
\begin{equation}\label{3mod}
|\psi_{m,k}^{(j)}(s)|\ll_s (1+\log |m|)|m|^{s^*},\qquad  \sum_{\underset{0<|m|\leq M}{m\in\ZZ}} |\psi_{m,l}^{(j)}(s_0)|^2 \ll_{\Gamma,s} (1+\log M) M^{2 s^*},
\end{equation} 
where $s^*=\max\lbrace \Re(s),1-\Re(s)\rbrace$. Combining the bounds above gives
\begin{align*}
&\left|\sum_{m=1}^{\infty}\frac{(-1)^{\upsilon}\Gamma(s)}{2\Gamma(s-\upsilon)}\frac{\psi_{m,k}^{(j)}(s)}{\sqrt{|m|}}W_{-\upsilon,s-\frac{1}{2}}(4\pi |m|y_k)e(m x_k)\right|\\&\qquad\ll_s |\mathfrak{i}_{s,2\upsilon}| \sum_{m=1}^{\infty} (1+\log m)m^{s^*-1/2} e^{-2\pi m y_k}\\&\qquad\leq |\mathfrak{i}_{s,2\upsilon}| e^{-2\pi y_k} \sum_{m=0}^{\infty} (m+1)^{s^*} e^{-2\pi m y_k}.
\end{align*}
Recall that we have assumed $y_k\gg1$, i.e.\ $y_k\geq \delta >0$. We thus have
 \begin{equation*}
 \sum_{m=0}^{\infty} (m+1)^{s^*} e^{-2\pi m y_k}\leq \sum_{m=0}^{\infty} (m+1)^{s^*} e^{-2\pi m \delta}\ll_{s} 1.
 \end{equation*}
We have now bounded the sum over $m>0$ by $\ll_s |\mathfrak{i}_{s,2\upsilon} | e^{-2\pi y_k}$. To complete this case, we must bound $\mathfrak{i}_{s,2\upsilon}$ as in the statement of the lemma. The reflection formula for $\Gamma(z)$ gives
\begin{align*}
|\mathfrak{i}_{s,2\upsilon}|=\left| \frac{\Gamma(s)^2\sin(\pi s)}{\pi}\right|\left| \frac{\Gamma(\upsilon+1-s)}{\Gamma(\upsilon+s)}\right| 
\end{align*}
For $\upsilon$ such that $\upsilon>\Re(s)+1$, using Stirling's approximation, we obtain
\begin{align*}
&\left| \frac{\Gamma(\upsilon+1-s)}{\Gamma(\upsilon+s)}\right|=\left|(\upsilon-s)\frac{\Gamma(\upsilon-s)}{\Gamma(\upsilon+s)}\right|\\\quad&=\left|(\upsilon-s)\frac{(\upsilon-s)^{\upsilon-s-1/ 2} e^{\upsilon+s}}{(\upsilon+s)^{\upsilon+s-1/ 2} e^{\upsilon-s}}\left(\frac{1+O(|\upsilon-s|^{-1})}{1+O(|\upsilon+s|^{-1})}\right)\right|\\&\ll_s \left|\frac{(\upsilon-s)^{1/ 2 +\upsilon-s} }{(\upsilon+s)^{\upsilon+s-1/ 2} }\right|= \upsilon^{1 -2\Re(s)}\left|\frac{(1-\frac{s}{\upsilon})^{1/ 2 + \upsilon-s}}{(1+\frac{s}{\upsilon})^{\upsilon+s-1/ 2 }}\right|\ll_s \upsilon^{1/ 2},
\end{align*}
since $\Re(s)> \frac{1}{4}$. This bound was proved for $\upsilon>\Re(s)+1$, and since there are only finitely many $\upsilon\in\ZZ_{>0}$ with $\upsilon\leq \Re(s) + 1$ (for given $s$), the bound in fact holds for all $\upsilon\in\ZZ_{>0}$. In conclusion:
\begin{equation}\label{sumplusbdd}
\left|\sum_{m=1}^{\infty}\frac{(-1)^{\upsilon}\Gamma(s)}{2\Gamma(s-\upsilon)}\frac{\psi_{m,k}^{(j)}(s)}{\sqrt{|m|}}W_{-\upsilon,s-\frac{1}{2}}(4\pi |m|y_k)e(m x_k)\right|\ll_s \sqrt{\upsilon}e^{-2\pi y_k}.
\end{equation}

Turning now to $\sum_{m=-\infty}^{-1}\frac{(-1)^{\upsilon}\Gamma(s)}{2\Gamma(s+\upsilon)}\frac{\psi_{m,k}^{(j)}(s)}{\sqrt{|m|}}W_{\upsilon,s-\frac{1}{2}}(4\pi |m|y_k)e(m x_k)$, by Proposition \ref{WhitInt} \textit{ii)} and \eqref{Ikdef2}, we have 
\begin{equation*}
W_{\upsilon,s-\frac{1}{2}}(4\pi|m|y_k)=4^{s-1}\pi^{-1}(-1)^{\upsilon}\Gamma(s+\upsilon)(4\pi |m| y_k)^{1-s} I_{\upsilon}(4\pi|m|y_k,s).
\end{equation*}
The sum we are considering may therefore be rewritten as
\begin{equation*}
\frac{\Gamma(s)}{2\pi^s}y_k^{1-s}\sum_{m=1}^{\infty} m^{1/ 2 -s} \psi_{-m,k}^{(j)}(s) I_{\upsilon} (4 \pi m y_k,s)e(-m x_k).
\end{equation*}
Let $\upsilon_0\in \ZZ_{>0}$ be such that $ \upsilon_0\geq \max\lbrace\frac{1}{(\sqrt{3}-\sqrt{2})^44\pi \delta},52\rbrace$, hence Lemma \ref{statphase1} and $4\pi m  y_k\geq \frac{1}{(\sqrt{3}-\sqrt{2})^4 \upsilon}$ apply for all $\upsilon\geq \upsilon_0$, $z\in \scrF_{B}\cup \scrC_{k,B}$, and $m\in\ZZ_{\geq 1}$. 
For $\upsilon<\upsilon_0$ and $\Re(s)>\frac{1}{4}$, \eqref{3mod} and Lemma \ref{ibpbdd}  give
\begin{align}\label{sminusbdd1}
&\left| \frac{\Gamma(s)}{2\pi^s}y_k^{1-s}\sum_{m=1}^{\infty} m^{1/ 2 -s} \psi_{-m,k}^{(j)}(s) I_{\upsilon} (4 \pi m y_k,s)e(-m x_k)\right|\\\notag&\qquad\qquad\ll_s y_k ^{1-\Re(s)}\sum_{m=1}^{\infty} m^{\frac{1}{2}-\Re(s)} (1+\log m) m^{s^*} \frac{\upsilon^2}{(4\pi m y_k)^2}\\\notag&\qquad\qquad\qquad \ll_s \upsilon_0^2 \,y_k^{-(1+\Re(s))}\ll y_k^{-(1+\Re(s))} .
\end{align}
To deal with the case $\upsilon\geq \upsilon_0$, we split the sum into three parts: 
\begin{equation}\label{NNpart}
\sum_{m=1}^{\infty}= \sum_{1\leq m \leq \frac{\upsilon}{4\pi y_k}}+\sum_{\frac{\upsilon}{4\pi y_k}< m < \frac{2\upsilon}{\pi y_k} }+\sum_{m\geq\frac{2\upsilon}{\pi y_k}}. 
\end{equation}
We use \eqref{3mod} and Lemma \ref{statphase} \textit{(2)} to bound the first sum in the right-hand side of this expression:
\begin{align}\label{sminusbdd2}
&\left| \frac{\Gamma(s)}{2\pi^s}y_k^{1-s} \sum_{1\leq m \leq \frac{\upsilon}{4\pi y_k}} m^{1/2 -s} \psi_{-m,k}^{(j)}(s) I_{\upsilon} (4 \pi m y_k,s)e(-m x_k)\right|\\\notag&\ll_s y_k^{1-\Re(s)}\!\!\!\!\! \!\!\!\sum_{1\leq m \leq \frac{\upsilon}{4\pi y_k}}\!\! \! \!\left(\!\!  m^{1/ 2 -\Re(s)}\!|\psi_{-m,k}^{(j)}(s)|\upsilon^{1/4-\Re(s)}\!(m y_k)^{\Re(s)-3/4}\!\!+\!\!\frac{(1\!+\!\log m) m^{\max\lbrace \frac{1}{2}, \frac{3}{2}-2 \Re(s)\rbrace }}{( m y_k)^2}\!\!\right)\\&\notag\qquad \ll_s y_k^{-(1+\Re(s))} + \upsilon^{1/4-\Re(s)} y_k^{1/4}\sum_{1\leq m \leq \frac{\upsilon}{4\pi y_k}}|\psi_{-m,k}^{(j)}(s)| m^{-1/ 4 } \\&\qquad\qquad\leq\notag  y_k^{-(1+\Re(s))} + \upsilon^{1/4-\Re(s)} y_k^{1/4}\sqrt{\sum_{1\leq m \leq \frac{\upsilon}{4\pi y_k}} | \psi_{-m,k}^{(j)}(s) | ^2} \sqrt{\sum_{1\leq m \leq \frac{\upsilon}{4\pi y_k}} m^{-1/2}}\\\notag&\qquad\qquad\ll_s  y_k^{-(1+\Re(s))}  +\upsilon^{1/4-\Re(s)} y_k^{1/4} \sqrt{1+\log \upsilon} \left(\frac{\upsilon}{y_k} \right)^{s^*}\left(\frac{\upsilon}{y_k} \right)^{1/4}\\\notag&\qquad\qquad\qquad \ll_s \sqrt{1+\log\upsilon}\,\upsilon^{\max\lbrace \frac{1}{2}, \frac{3}{2}- 2 \Re(s) \rbrace} y_k^{-\Re(s)}.
\end{align}
(Note that we have used the fact that for this sum to be non-empty, we must have $\frac{\upsilon}{4\pi y_k}\geq 1$, hence $0\leq \log \left( \frac{\upsilon}{4\pi y_k}\right)\leq \log \left( \frac{\upsilon}{4\pi \delta}\right)\ll 1 + \log\upsilon$. ) We bound the third sum in the right-hand side of \eqref{NNpart} using \eqref{3mod} and Lemma \ref{statphase} \textit{(1)}:
\begin{align}\label{sminusbdd4}
&\left| \frac{\Gamma(s)}{2\pi^s}y_k^{1-s}\sum_{m\geq\frac{2\upsilon}{\pi y_k}} m^{1/2 -s} \psi_{-m,k}^{(j)}(s) I_{\upsilon} (4 \pi m y_k,s)e(-m x_k)\right|\\\notag&\qquad\qquad\ll_s y_k^{1-\Re(s)} \sum_{m\geq\frac{2\upsilon}{\pi y_k}} m^{1/ 2 -\Re(s)}  |\psi_{-m,k}^{(j)}(s) | \frac{1}{y_k^2m^2}\\\notag&\qquad\qquad\quad\ll_s  y_k^{-(1+\Re(s))} \sum_{m\geq\frac{2\upsilon}{\pi y_k}}\frac{\log m\, m^{\max\lbrace\frac{1}{2},\frac{3}{2}-2\Re(s)\rbrace}}{m^{^2}}\ll_s y_k^{-(1+\Re(s))}.
\end{align}
To bound the second sum in the right-hand side of \eqref{NNpart}, we proceed in a similar manner to \eqref{sminusbdd2} above:
\begin{align*}
&\left| \frac{\Gamma(s)}{2\pi^s}y_k^{1-s} \sum_{\frac{\upsilon}{4\pi y_k}< m < \frac{2\upsilon}{\pi y_k} } m^{1/2 -s} \psi_{-m,k}^{(j)}(s) I_{\upsilon} (4 \pi m y_k,s)e(-m x_k)\right|\\\quad&\ll_s y_k^{1-\Re(s)} \sqrt{\sum_{\frac{\upsilon}{4\pi y_k}< m < \frac{2\upsilon}{\pi y_k} } |\psi_{-m,k}^{(j)}(s) |^2 }\sqrt{\sum_{\frac{\upsilon}{4\pi y_k}< m < \frac{2\upsilon}{\pi y_k} }  m^{1-2\Re(s)}|I_{\upsilon} (4 \pi m y_k,s)|^2 }
\\\qquad \ll_s& y_k^{1-\Re(s)}\sqrt{1+\log \upsilon}\left(\frac{\upsilon}{y_k}\right)^{s^*}\left(\frac{\upsilon}{y_k}\right)^{\frac{1}{2}-\Re(s)}\sqrt{\sum_{\frac{\upsilon}{4\pi y_k}< m < \frac{2\upsilon}{\pi y_k} }  |I_{\upsilon} (4 \pi m y_k,s)|^2 }
\end{align*} 
The remaining sum in this expression is split into three parts in accordance with Lemma \ref{statphase1}:
\begin{equation*}
\sum_{\frac{\upsilon}{4\pi y_k}< m \leq \frac{2\upsilon}{\pi y_k} }=\sum_{\frac{\upsilon}{4\pi y_k}< m \leq \frac{\upsilon}{\pi y_k(1+\upsilon^{-2/3})} }+\sum_{\frac{\upsilon}{\pi y_k(1+\upsilon^{-2/3})}\leq m \leq \frac{\upsilon(1+\upsilon^{-2/3})}{\pi y_k}}+\sum_{\frac{\upsilon(1+\upsilon^{-2/3})}{\pi y_k}< m < \frac{2\upsilon}{\pi y_k} }
\end{equation*}
Firstly, Lemma \ref{statphase1} \textit{(4)} gives
\begin{align*}
&\sum_{\frac{\upsilon}{4\pi y_k}< m \leq \frac{\upsilon}{\pi y_k(1+\upsilon^{-2/3})} }|I_{\upsilon} (4 \pi m y_k,s)|^2\ll_s\sum_{\frac{\upsilon}{\pi y_k}< m \leq \frac{\upsilon}{\pi y_k(1+\upsilon^{-2/3})} } \frac{1}{\upsilon\sqrt{\frac{\upsilon}{\pi y_k m}-1}}\\&\ll \frac{1}{\upsilon^{1/2}}\sum_{\frac{\upsilon}{4\pi y_k}< m \leq \frac{\upsilon}{\pi y_k(1+\upsilon^{-2/3})} } \frac{1}{\sqrt{\upsilon-\pi y_k m}} \\&\leq \frac{1}{\upsilon^{1/2}}\left(\int_{\frac{\upsilon}{4\pi y_k}}^{\frac{\upsilon}{\pi y_k (1+ \upsilon^{-2/3})}}(\upsilon-\pi y_k x)^{-1/2}\,dx + \sqrt{\upsilon^{-1/3}+\upsilon^{-1}}\right)\\&\ll  \frac{1}{\upsilon^{1/2}}\left(\frac{\upsilon^{1/2}}{y_k} +\upsilon^{-1/6}\right)\ll y_k^{-2/3}
\end{align*}
(the last inequality holding due to the fact that for the sum to be non-empty, we must have $\upsilon\gg y_k$). Combining cases \textit{(2)} and \textit{(3)} of Lemma \ref{statphase1} yields
\begin{align*}
\sum_{\frac{\upsilon}{\pi y_k(1+\upsilon^{-2/3})}\leq m \leq \frac{\upsilon(1+\upsilon^{-2/3})}{\pi y_k}}\!\!\!\!\!\!\!|I_{\upsilon} (4 \pi m y_k,s)|^2\ll_s \upsilon^{-2/3}\left(\frac{\upsilon(1+\upsilon^{-2/3})}{\pi y_k}-\frac{\upsilon}{\pi y_k(1+\upsilon^{-2/3})} \right)\ll \frac{\upsilon^{-1/3}}{y_k}.
\end{align*}
The final sum is then bounded using case \textit{(1)} of Lemma \ref{statphase1}
\begin{align*}
\sum_{\frac{\upsilon(1+\upsilon^{-2/3})}{\pi y_k}< m < \frac{2\upsilon}{\pi y_k} }&|I_{\upsilon} (4 \pi  m y_k ,s)|^2 \ll_s \sum_{\frac{\upsilon(1+\upsilon^{-2/3})}{\pi y_k}< m < \frac{2\upsilon}{\pi y_k} } \frac{1}{(4\pi m y_k -4\upsilon)^2}\\&\leq \frac{1}{\upsilon^{2/3}}+\int_{\frac{\upsilon(1+\upsilon^{-2/3})}{\pi y_k}}^{\frac{2\upsilon}{\pi y_k}}\frac{1}{(4\pi x y_k -4\upsilon)^2}\,dx\ll \frac{1}{\upsilon^{2/3}}+\frac{1}{y_k \upsilon^{1/3}}\ll \frac{1}{y_k^{2/3}},
\end{align*}
hence
\begin{equation*}
\sqrt{\sum_{\frac{\upsilon}{4\pi y_k}< m < \frac{2\upsilon}{\pi y_k} }  |I_{\upsilon} (4 \pi m y_k,s)|^2 }\ll_s \frac{1}{y_k^{1/3}}.
\end{equation*}
This gives
\begin{align}\label{sminusbdd3}
&\left| \frac{\Gamma(s)}{2\pi^s}y_k^{1-s}\sum_{\frac{\upsilon}{4\pi y_k}< m < \frac{2\upsilon}{\pi y_k} } m^{1/2 -s} \psi_{-m,k}^{(j)}(s) I_{\upsilon} (4 \pi m y_k,s)e(-m x_k)\right|\\\notag&\qquad\qquad\ll_s \sqrt{1+\log \upsilon} \upsilon^{\max\lbrace \frac{1}{2}, \frac{3}{2}-2\Re(s)\rbrace} y_k^{-\frac{1}{3}}.
\end{align}

Combining \eqref{sumplusbdd}, \eqref{sminusbdd1}, \eqref{sminusbdd2}, \eqref{sminusbdd4}, and \eqref{sminusbdd3} completes the proof.

\hspace{425.5pt}\qedsymbol

\subsection{Proof of Lemma \ref{Eisu1}}\label{Eisu1proof}
We start by proving a version of \eqref{RanSelbd} for $s$ in a neighbourhood of $1$. By Lemma \ref{Eis1pole}, $\Res_{s=1} E_j(z,s)=\Res_{s=1}\varphi_{j,k}(s)= \frac{1}{\mu(\GaH)}$ for all $j,k$; the maps $s\mapsto \psi_{m,k}^{(j)}(s)$ are thus holomorphic at $s=1$.

Let $\epsilon>0$ be such that $D_{2\epsilon}=\lbrace s\in\CC\,:\, 0<|s-1|\leq 2\epsilon\rbrace \subset \CC\setminus\scrE$. By Cauchy's integral formula, for all $s\in D_{\epsilon}=\lbrace s\in\CC\,:\, |s-1|\leq \epsilon\rbrace$, we have  $\psi_{m,k}^{(j)}(s)=\frac{1}{2\pi i} \oint_{\partial D_{2\epsilon}} \frac{\psi_{m,k}^{(j)}(\zeta)}{\zeta-s}\,d\zeta$, hence
\begin{align*}
\sum_{1\leq |m|\leq M} |\psi_{m,k}^{(j)}(s)|^2= &\sum_{1\leq |m|\leq M} \left|\sfrac{1}{2\pi i} \oint_{\partial D_{\epsilon}} \sfrac{\psi_{m,k}^{(j)}(\zeta)}{\zeta-s}\,d\zeta \right|^2 \leq \frac{1}{2\pi}\int_0^{2\pi} \sum_{1\leq |m|\leq M}|\psi_{m,k}^{(j)}(1+2\epsilon e^{i\theta})|^2 \,d\theta\\&\ll_{\epsilon} M^{2+4\epsilon},
\end{align*}
i.e.
\begin{equation}\label{psis=1}
\sum_{1\leq |m|\leq M} |\psi_{m,k}^{(j)}(s)|^2\ll_{\epsilon} M^{2+4\epsilon}\qquad \forall s\in D_{\epsilon}.
\end{equation}
For $s\in D_{\epsilon}$, the Fourier decomposition of $E_j(z,s,2\upsilon)$ at a cusp $\eta_k$ reads
\begin{align*}
E_j(z,s,2\upsilon)=&\delta_{j,k}y_k^s+\mathfrak{i}_{s,2\upsilon} \varphi_{j,k}(s) y_k^{1-s}\\\notag&\quad+\sum_{\underset{m\neq 0}{m\in\ZZ}}\frac{(-1)^{\upsilon}\Gamma(s)}{2\Gamma(s-\upsilon\cdot\sgn(m))}\frac{\psi_{m,k}^{(j)}(s)}{\sqrt{|m|}}W_{-\upsilon\cdot\sgn(m),s-\frac{1}{2}}(4\pi |m|y_k)e(m x_k).
\end{align*}
Again by Lemma \ref{Eis1pole}, for $s\in D_{\epsilon}$, $\varphi_{j,k}(s)=\frac{\mu(\GaH)^{-1}}{s-1}+\widetilde{\varphi}_{j,k}(s)$, and $\widetilde{\varphi}_{j,k}$ is holomorphic on $D_{\epsilon}\cup\lbrace 1 \rbrace$. From the definition of $\mathfrak{i}_{s,2\upsilon}$, we then have (for $s\in D_{\epsilon}$)
\begin{align*}
\mathfrak{i}_{s,2\upsilon}\varphi_{j,k}(s)y_k^{1-s} =\frac{(-1)^{\upsilon}\Gamma(s)^2}{\Gamma(s+\upsilon)\Gamma(s-\upsilon)}\left( \frac{\mu(\GaH)^{-1}}{s-1}+\widetilde{\varphi}_{j,k}(s)\right)y_k^{1-s}.
\end{align*}
We have
\begin{equation*}
\mathfrak{i}_{s,2\upsilon}=\frac{(-1)^{|\upsilon|}\Gamma(s)^2}{\Gamma(s+|\upsilon|)\Gamma(s-|\upsilon|)}=\frac{(1-s)(2-s)\ldots(|\upsilon|-s)\Gamma(s)}{\Gamma(s+|\upsilon|)},
\end{equation*}
hence
\begin{equation*}
\lim_{s\rightarrow 1} \mathfrak{i}_{s,2\upsilon}\varphi_{j,k}(s)y_k^{1-s}= -\frac{\mu(\GaH)^{-1}}{|\upsilon|},
\end{equation*}
and so
\begin{align*}
E_j(z,1,2\upsilon)=&\delta_{j,k}y_k-\frac{\mu(\GaH)^{-1}}{|\upsilon|}\\\notag&\quad+\sum_{\underset{m\neq 0}{m\in\ZZ}}\frac{(-1)^{\upsilon}}{2\Gamma(1-\upsilon\cdot\sgn(m))}\frac{\psi_{m,k}^{(j)}(1)}{\sqrt{|m|}}W_{-\upsilon\cdot\sgn(m),\frac{1}{2}}(4\pi |m|y_k)e(m x_k)\\&=\delta_{j,k}y_k-\frac{\mu(\GaH)^{-1}}{|\upsilon|}\\\notag&\quad+\sum_{\underset{\sgn(m)=-\sgn(\upsilon)}{m\in\ZZ}}\frac{(-1)^{\upsilon}}{2\Gamma(1+|\upsilon|)}\frac{\psi_{m,k}^{(j)}(1)}{\sqrt{|m|}}W_{|\upsilon|,\frac{1}{2}}(4\pi |m|y_k)e(m x_k).
\end{align*}
The rest of the proof is the same as for Lemma \ref{Ej2ubdd}, though with \eqref{psis=1} being used in place of \eqref{RanSelbd} in the calculations leading to, \eqref{sminusbdd1}, \eqref{sminusbdd2}, \eqref{sminusbdd4}, and \eqref{sminusbdd3}.

\hspace{425.5pt}\qedsymbol

\section{Analytic continuation of principal series representations}\label{Analytic}

Here we recall some of the key facts regarding the analytic continuation of principal series representations. 
\subsection{Vector-valued analytic functions}
Recall that for $g=\smatr a b c d\in \SL(2,\RR)$, $\wty\smatr a b c d = (c^2+d^2)^{-1}$, hence
\begin{align*}
\quad\left[ \pi^s(g)\vece_0\right](\kthe)=\wty(\kthe g)^s =\big((a^2+b^2)\sin^2\theta+(ac+bd)&\sin2\theta+(c^2+d^2)\cos^2\theta\big)^{-s}\\&\qquad \forall g=\smatr a b c d \in G,\,\theta\in \RR/2\pi\ZZ.
\end{align*}
Note that
\begin{equation*}
0<\min_{\theta\in \RR/2\pi\ZZ} \Re\Big((a^2+b^2)\sin^2\theta+(ac+bd)\sin2\theta+(c^2+d^2)\cos^2\theta \Big)\qquad \forall \smatr a b c d \in U.
\end{equation*}
For each $g=\smatr a b c d\in U$, we define a function $\vecv_g$ on $K$ by
\begin{equation}\label{vgdef}
\vecv_g(\kthe):=(a^2+b^2)\sin^2\theta+(ac+bd)\sin2\theta+(c^2+d^2)\cos^2\theta\qquad \forall \theta\in\RR/2\pi\ZZ.
\end{equation}
\begin{prop}\label{vgprops}$ $
\begin{enumerate}[i)]
\item $\vecv_g^{-s}\in V^{\infty}$ for all $s\in \CC$, $g\in U$.
\item Let $D$ and $S$ be compact subsets of $U$ and $\CC$, respectively. Then  $|\langle \vecv_g^{-s},\vece_{2\upsilon}\rangle_{L^2(K)}|\ll_{D,S,n} (1+|\upsilon|)^{-n}$ for all $g\in D$, $s\in S$, $\upsilon\in\ZZ$, $n\in \NN$.
\item For each fixed $s\in\CC$, $g\mapsto \vecv_g^{-s}$ is an analytic function from $U$ to $V$. Consequently, $g\mapsto\pi^s(g)\vece_0$ has an analytic continuation to all of $U$.
\end{enumerate}
\end{prop}

\begin{proof}
Starting with \textit{i)}, it suffices to note that $\theta\mapsto v_g(\kthe)$ is smooth, and
\begin{equation*}
\min_{\theta\in \RR/2\ZZ} \Re(\vecv_g(\kthe))\geq C_g >0.
\end{equation*}
Likewise, \textit{ii)} follows from a standard application of integration by parts; note that $\vecv_g(\theta)$ depends polynomially on the coefficients of $g$, there therefore exists, for each $n\in\NN$, a polynomial $P_n$ in the coefficients of $g$ and $|s|$ such that
\begin{equation*}
\left| \sfrac{d^n}{d\theta^n}\vecv_g(\kthe)\right|\ll  \max\left\lbrace 1, \max_{s\in S,g\in D,\theta\in\RR/2\pi\ZZ} \left|\vecv_g(\kthe)^{-(s+n)}\right|\right\rbrace P_n( |a|,|b|,|c|,|d|,|s|).
\end{equation*}

Finally, in order to prove \textit{iii)}, we use \cite[Theorem 3.31]{Rudin}; in order to prove that $g\mapsto\vecv_g^{-s}$ is analytic, it suffices to prove that $g\mapsto\langle \vecv_g^{-s},\vecu\rangle_{L^2(K)}$ is analytic for all $\vecu\in V$. By \textit{ii)}, we have
\begin{align*}
\langle \vecv_g^{-s},\vecu\rangle_{L^2(K)}=\sum_{\upsilon\in\ZZ}\langle \vecv_g^{-s},\vece_{2\upsilon}\rangle_{L^2(K)}\langle \vece_{2\upsilon},\vecu\rangle_{L^2(K)},
\end{align*}
with the sum converging absolutely and uniformly over $g$ in compact subsets of $U$. Each $g\mapsto\langle \vecv_g^{-s},\vece_{2\upsilon}\rangle_{L^2(K)}$ is an analytic function, hence so is $g\mapsto\langle \vecv_g^{-s},\vecu\rangle_{L^2(K)}$.
\end{proof}
Recall now that for $0<\epsilon<\frac{1}{5}$, $g_{\epsilon}=\smatr{e^{(\pi/4-\epsilon)i}}{}{}{e^{-(\pi/4-\epsilon)i}}\in U$.
\begin{lem}\label{gTSobbdd} For $1\leq\beta\leq 2$ and $0<\epsilon<\frac{1}{5}$, 
\begin{equation*}
\scrS_{\beta}\big(\pi^{1/2+it}(g_{\epsilon})\vece_0\big)\ll_{t,\beta} \epsilon^{-\beta}.
\end{equation*}
\end{lem}
\begin{proof}
For $\vecv\in V^{\infty}$ and $n=1,2$, we have 
$\scrS_n(\vecv)\ll \|\vecv\|_{L^2(K)} +\|\frac{d^n}{d\theta^n}\vecv\|_{L^2(K)}$ (when viewing $\vecv$ as a function of $\theta$ under the map $\theta\mapsto \vecv(\kthe)$).
We have
\begin{align*}
\left[ \pi^{1/2+it}(g_{\epsilon})\vece_0\right](\kthe)&=\big(i e^{-2\epsilon i}\sin^2\theta-i e^{2\epsilon i}\cos^2\theta\big)^{-1/2-it}\\&=(\cos2\epsilon)^{-1/2-it}(\tan2\epsilon-i\cos 2\theta)^{-1/2-it},
\end{align*}
hence
\begin{align*}
&\sfrac{d}{d\theta}\left(\left[ \pi^{1/2+it}(g_{\epsilon})\vece_0\right] \right)(\kthe)=-(1+2 i t)(\cos2\epsilon)^{-1/2-it}\sin 2\theta(\tan2\epsilon-i\cos 2\theta)^{-3/2-it}\\
&\sfrac{d^2}{d\theta^2}\left(\left[ \pi^{1/2+it}(g_{\epsilon})\vece_0\right] \right)(\kthe)\\&\quad=-(\cos2\epsilon)^{-1/2-it}(2+4 i t)\big(1+(\sfrac 1 2 + it )\sin^2 2\theta + i \tan 2\epsilon \cos 2 \theta\big)(\tan2\epsilon-i\cos 2\theta)^{-5/2-it}.
\end{align*}
This gives
\begin{equation*}
\|\pi^{1/2+it}(g_{\epsilon})\vece_0\|_{L^2(K)}\ll_t|\log \epsilon|,\qquad \|\sfrac{d^n}{d\theta^n}\pi^{1/2+it}(g_{\epsilon})\vece_0\|_{L^2(K)}\ll_t \epsilon^{-n}.
\end{equation*}
 hence $\scrS_n\big(\pi^{1/2+it}(g_{\epsilon})\vece_0\big)\ll_t \epsilon^n$ (for $n=1,2$). Now assuming $1<\beta<2$, we use H\"older's inequality to interpolate between the cases $\beta=1$ or $\beta=2$: let $p=\frac{1}{2-\beta}$ and $q=\frac{1}{\beta-1}$. Writing $a_{2\upsilon}=\langle \pi^{1/2+it}(g_{\epsilon})\vece_0,\vece_{2\upsilon}\rangle_{L^2(K)}$, we have
 \begin{align*}
\scrS_{\beta}\big(\pi^{1/2+it}(g_{\epsilon})\vece_0\big)^2&=\sum_{\upsilon\in\ZZ} (1+|\upsilon|^{\beta})^2|a_{2\upsilon}|^2\\&=\sum_{\upsilon\in\ZZ} \left((1+|\upsilon|^{\beta})^{\frac{2}{\beta p}}|a_{2\upsilon}|^{\frac{2}{p}}\right)\left((1+|\upsilon|^{\beta})^{\frac{4}{\beta q}}|a_{2\upsilon}|^{\frac{2}{q}}\right)\\&\leq  \left(\sum_{\upsilon\in\ZZ}(1+|\upsilon|^{\beta})^{\frac{2}{\beta }}|a_{2\upsilon}|^{2}\right)^{\frac{1}{p}}\left(\sum_{\upsilon\in\ZZ}(1+|\upsilon|^{\beta})^{\frac{4}{\beta }}|a_{2\upsilon}|^{2}\right)^{\frac{1}{q}}\\&\ll_{\beta} \scrS_{1}\big(\pi^{1/2+it}(g_{\epsilon})\big)^{\frac{2}{p}} \scrS_{2}\big(\pi^{1/2+it}(g_{\epsilon})\big)^{\frac{2}{q}}
\\&= \scrS_{1}\big(\pi^{1/2+it}(g_{\epsilon})\big)^{4-2\beta} \scrS_{2}\big(\pi^{1/2+it}(g_{\epsilon})\big)^{2\beta-2}\\&\ll_t \epsilon^{-2\beta}. 
\end{align*}
\end{proof}
\subsection{Proof of Lemma \ref{Phiga}}\label{AppB1}

Proposition \ref{vgprops} \textit{ii)} and \eqref{Psisigupsbdd} allow us to write
\begin{equation}\label{Psigdecomp1}
\Psi_g=\sum_{\upsilon,\sigma\in \ZZ}  \langle \vecv_g^{-r},\vece_{2\upsilon}\rangle_{L^2(K)}\langle \vecv_g^{-s},\vece_{2\sigma}\rangle_{L^2(K)}\Psi_{j,k}^{r,s}(\vece_{2\upsilon}\otimes \vece_{2\sigma}).
\end{equation}
We then have, for all $f\in L^2(\GaG)$,
\begin{align*}
&|\langle \Psi_g, f \rangle_{L^2(\GaG)}|\\&= \left|\sum_{\upsilon,\sigma\in \ZZ}\langle \vecv_g^{-r},\vece_{2\upsilon}\rangle_{L^2(K)}\langle \vecv_g^{-s},\vece_{2\sigma}\rangle_{L^2(K)}\langle\Psi_{j,k}^{r,s}(\vece_{2\upsilon}\otimes \vece_{2\sigma}),f \rangle_{L^2(\GaG)} \right|\\&\quad\leq  \sum_{\upsilon,\sigma\in \ZZ} |\langle \vecv_g^{-r},\vece_{2\upsilon}\rangle_{L^2(K)}||\langle \vecv_g^{-s},\vece_{2\sigma}\rangle_{L^2(K)}| \|\Psi_{j,k}^{r,s}(\vece_{2\upsilon}\otimes \vece_{2\sigma})\|_{L^2(\GaG)} \|f\|_{L^2(\GaG)},
\end{align*}
with the sum converging absolutely and uniformly over $g$ in compact subsets of $U$ (again by Proposition \ref{vgprops} \textit{ii)} and \eqref{Psisigupsbdd}). Since each
\begin{equation*}
g\mapsto \langle \vecv_g^{-r},\vece_{2\upsilon}\rangle_{L^2(K)}\langle \vecv_g^{-s},\vece_{2\sigma}\rangle_{L^2(K)}\langle\Psi_{j,k}^{r,s}(\vece_{2\upsilon}\otimes \vece_{2\sigma}),f \rangle_{L^2(\GaG)} 
\end{equation*}
is an analytic function on $U$, so is $g\mapsto \langle \Psi_g, f \rangle_{L^2(\GaG)}$. We have thus shown that $g\mapsto  \Psi_g$ is weakly analytic. Again, by \cite[Theorem 3.31]{Rudin}, $g\mapsto\Psi_g$ is analytic.

\hspace{425.5pt}\qedsymbol

\subsection{Proof of Lemma \ref{Phigb}}\label{AppB2}

We start by proving that $g\mapsto \Proj_i(\Psi_g)$ is analytic. By Lemma \ref{Phiga}, $g\mapsto\Psi_g$ is analytic (from $U$ to $L^2(\GaG)$) since $\Proj_i:L^2(\GaG)\rightarrow \scrH_i$ is a bounded linear operator, $g\mapsto 
\Proj_i(\Psi_g)$ is analytic (from $U$ to $\scrH_i$).

Turning to $g\mapsto \vecv_{\Psi_{g},m}(t)$, Since $\Psi_g\in C^{\infty}(\GaG)\cap L^p(\GaG)$ for all $p<\infty$, Corollary \ref{fdecomp2} gives
\begin{equation*}
 \vecv_{\Psi_{g},m}(t)=\sfrac{\wty(h_m)^{1/2-it}}{\sqrt{2\pi}}\sum_{\upsilon\in \ZZ} \left(\int_{\GaG} \Psi_g\, \overline{\EE_m^{1/2+it}(\cdot,\vece_{2\upsilon})}\,d\mu_{\GaG}\right)\vece_{2\upsilon}.
\end{equation*}
Once again using \eqref{Psigdecomp1}, Proposition \ref{vgprops} \textit{ii)} and \eqref{Psisigupsbdd}, we have
\begin{align*}
&\langle \Psi_g, \EE_m^{1/2+it}(\cdot,\vece_{2\upsilon})\rangle_{L^2(\GaG)}\\&\qquad=\sum_{\tau,\sigma\in \ZZ}  \langle \vecv_g^{-r},\vece_{2\tau}\rangle_{L^2(K)}\langle \vecv_g^{-s},\vece_{2\sigma}\rangle_{L^2(K)} \left(\int_{\GaG} \Psi_{j,k}^{r,s}(\vece_{2\tau}\otimes \vece_{2\sigma})\, \overline{\EE_m^{1/2+it}(\cdot,\vece_{2\upsilon})}\,d\mu_{\GaG}\right)\!.
\end{align*}
Now noting that $\rho(k) \Psi_{j,k}^{r,s}(\vece_{2\tau}\otimes \vece_{2\sigma})=\vece_{2(\tau+\sigma)}(k)\Psi_{j,k}^{r,s}(\vece_{2\tau}\otimes \vece_{2\sigma})$ and $\rho(k)\EE_m^{1/2+it}(\cdot,\vece_{2\upsilon})=\vece_{2\upsilon}(k)\EE_m^{1/2+it}(\cdot,\vece_{2\upsilon})$, by Lemma \ref{Ej2ubdd}, or \eqref{EjkFourier} and \eqref{Eisexp}, and \eqref{Psisigupsbdd}, we have
\begin{equation*}
\left|\int_{\GaG} \Psi_{j,k}^{r,s}(\vece_{2\tau}\otimes \vece_{2\sigma})\, \overline{\EE_m^{1/2+it}(\cdot,\vece_{2\upsilon})}\,d\mu_{\GaG}\right|=\begin{cases} 0\qquad&\mathrm{if\;}\upsilon\neq \sigma+\tau\\ O_{s,r,t}\big((1+|\sigma|)^3(1+|\tau|)^3\big)&\mathrm{otherwise}.\end{cases}
\end{equation*}
Thus, as in the proof of Lemma \ref{Phiga}, Proposition \ref{vgprops} \textit{ii)}, and \eqref{Psisigupsbdd} give that for any $\vecu\in V$, 
\begin{align*}
\langle \vecv_{\Psi_{g},m}(t), \vecu\rangle_{L^2(K)}&=\sfrac{\wty(h_m)^{1/2-it}}{\sqrt{2\pi}}\sum_{\tau,\sigma\in\ZZ} \langle \vecv_g^{-r},\vece_{2\tau}\rangle_{L^2(K)}\langle \vecv_g^{-s},\vece_{2\sigma}\rangle_{L^2(K)}\langle \vece_{2(\tau+\sigma)},\vecu\rangle_{L^2(K)}\\&\qquad\qquad\qquad\qquad\quad\times\left(\int_{\GaG} \Psi_{j,k}^{r,s}(\vece_{2\tau}\otimes \vece_{2\sigma})\, \overline{\EE_m^{1/2+it}(\cdot,\vece_{2(\tau+\sigma)})}\,d\mu_{\GaG}\right)\!,
\end{align*}
with the sum converging absolutely and uniformly over $g$ in compact subsets of $U$, proving that $g\mapsto\langle \vecv_{\Psi_{g},m}(t), \vecu\rangle_{L^2(K)}$ is analytic, hence (once again by \cite[Theorem 3.31]{Rudin}) $g\mapsto\vecv_{\Psi_{g},m}(t)$ is analytic.

\hspace{425.5pt}\qedsymbol

\section{Contour integrals of $L_k(|E|^2,s)$}\label{AppC}
The goal of this section is demonstrate how one deduces Corollary \ref{Fccor} from Theorem \ref{thm1}. 
\subsection{The Phragm\'en-Lindel\"of principle} We start by proving the following variation of the Phragm\'en-Lindel\"of principle:
\begin{lem}\label{PhragLind}
Assume that $f:\CC\rightarrow\CC$ is holomorphic on $[\sigma_1,\sigma_2]+i\RR$, $f(s)\in\RR$ for $s\in[\sigma_1,\sigma_2]$, and that $\int_T^{T+1} |f(\sigma_j+it)|^2\,dt\leq 1$ for all $T\in \RR$ and $j\in\lbrace 1,2\rbrace$. Assume furthermore that there exists $B>0$ such that $\int_T^{T+1} |f(\sigma+it)|^2\,dt\leq e^{B(1+|T|)}$ for all $\sigma\in[\sigma_1,\sigma_2]$ and $T\in\RR$. Then $\int_T^{T+1}|f(\sigma+it)|^2\,dt\leq 1$ for all $\sigma\in[\sigma_1,\sigma_2]$ and $T\in\RR$. 
\end{lem}
\begin{proof}
For $U\in\RR$ and $s\in [\sigma_1,\sigma_2]+i\RR$, define $F_U$ by
\begin{equation*}
F_U(s):=\int_0^1 f(s+i\tau)f(s-i2U-i\tau)\,d\tau.
\end{equation*}
Since $f$ is holomorphic on $[\sigma_1,\sigma_2]+i\RR$, so is $F_U$ (for fixed $U$). For $s\in [\sigma_1,\sigma_2]+iU$, we have $s-i2U-i\tau=\overline{s+i\tau}$, giving $F_U(s)=\int_0^1|f(s+i\tau)|^2\,d\tau$. Note also that for all $s\in[\sigma_1,\sigma_2]+i\RR$,
\begin{align*}
|F_U(s)|\leq \sqrt{\int_0^1 |f(s+i\tau)|^2\,d\tau}\sqrt{\int_0^1 |f(s-i2U-i\tau)|^2\,d\tau},
\end{align*}
hence
\begin{equation*}
|F_U(\sigma_j+it)|\leq 1 \qquad \forall t\in \RR,\,j\in\lbrace 1, 2\rbrace,
\end{equation*}
as well as
\begin{equation*}
|F_U(\sigma+it)|\leq e^{\frac{1}{2}B(1+|t|+1+|t-2U-1|)}\qquad\forall \sigma\in[\sigma_1,\sigma_2],\,t\in\RR.
\end{equation*}
We may now apply the Phragm\'en-Lindel\"of principle to $F_U$: by \cite[5.65]{Titch}, $|F_U(s)|\leq 1$ for all $s\in[\sigma_1,\sigma_2]+i\RR$. In particular, for $s\in [\sigma_1,\sigma_2]+i\RR$ with $\Im(s)=U$, $1\geq |F_U(s)|=\int_0^1 |f(s+i\tau)|^2\,d\tau$, as desired.
\end{proof}
\subsection{Proof of Corollary \ref{Fccor}}
We start by following \cite[pp.\ 120-121]{Petridis95}. For $M>U>2$, let $\psi_U\in C^{\infty}(\RR)$ be such that $\psi_U(x)\geq 0$ for all $x\in\RR$, and
\begin{equation*}
\psi_U(x)=\begin{cases}0\qquad&\mathrm{if\;} x\geq 1+U^{-1}\\1\qquad&\mathrm{if\;}x\leq 1-U^{-1}. \end{cases}
\end{equation*}
We choose $\psi_U$ so that $|\frac{d^n}{dx^n}\psi_U(x)|\ll_n U^n$. Let $\Psi_U(s)$ denote the Mellin transform of $\psi_U$, that is $\Psi_U(s)=\int_0^{\infty} \psi_U(x) x^{s-1}\,dx$ (where $\Re(s)>0$). The bound on $\frac{d^n}{dx^n}\psi_U(x)$ implies (cf.\ \cite[(28)-(30)]{Petridis95}) that
\begin{equation}\label{PsiUbdd}
\left|\Psi_U(s) \right|\ll_c \frac{1}{|s|}\left(\frac{U}{1+|s|}\right)^{c-1}\qquad \forall c\in \RR_{\geq 0},
\end{equation}
where the implied constant is uniformly bounded for $s\in [\sigma_1,\sigma_2]+i\RR$, $0<\sigma_1<\sigma_2<\infty$ and $c$ in compact subsets of $\RR_{\geq 0}$. Now using the bound on $\Psi_U$ above and the Mellin inversion formula, 
\begin{align}\label{Mellin}
\frac{1}{2\pi i}\int_{\frac{3}{2}+i\RR} L_k(s)\Psi_U(s) M^s\,ds&=\sum_{\underset{m\neq 0}{m\in\ZZ}} |\psi_{m,k}^{(k_0)}(\sfrac 1 2 + i t_0)|^2\frac{1}{2\pi i}\int_{\frac{3}{2}+i\RR}\Psi_U(s) \left(\sfrac{M}{|m|}\right)^s\,ds\\\notag=&\sum_{\underset{m\neq 0}{m\in\ZZ}} |\psi_{m,k}^{(k_0)}(\sfrac 1 2 + i t_0)|^2\psi_U\left(\sfrac{|m|}{M} \right)
\end{align}
(here $L_k(s)=L_k(|E|^2,s)$).  We will now move the contour of integration from $\frac{3}{2}+i\RR$ to $\frac{1}{2}+i\RR$. By Proposition \ref{GL=RS}, the poles of $L_k(s)$ contained in $[\frac{1}{2},\frac{3}{2}]+i\RR$ are $\lbrace 1, 1-2it_0,1+2it_0\rbrace\cup (\scrE_k\cap (\frac{1}{2},1])$. We denote by $\scrE'$ the multiset of poles of $L_k(s)$ in $[\frac{1}{2},\frac{3}{2}]+i\RR$, and let
\begin{equation*}
f(s)=L_k(s)\Psi_U(s)\prod_{\zeta\in\scrE'}(s-\zeta)
\end{equation*}
(note that $\#\scrE' <\infty$). Our goal will now be to apply Lemma \ref{PhragLind} to $f(s)$. By construction, $f$ is holomorphic on $[\frac{1}{2},\frac{3}{2}]+i\RR$. Furthermore, using \eqref{Lkdef}, $|L_k(\frac{3}{2}+it)|\ll 1$, hence by this and \eqref{PsiUbdd}, $|f(\frac{3}{2}+it)|\ll 1$ for all $t\in \RR$. This gives $\int_T^{T+1} |f(\frac{3}{2}+it)|^2\,dt\ll1$ for all $T\in\RR$. Since $f$ is holomorphic, $|f(\frac{1}{2}+it)|\ll 1$ (the implied constants depend on $U$; this is of no importance for moving the contour of integration) for $|t|\leq T_0$ (for all $T_0>0$). For $|T|\geq T_0$,  Theorem \ref{thm1} gives
\begin{align*}
\int_T^{T+1} |f(\sfrac{1}{2}+it)|^2\,dt \leq &\left(\max_{t\in[T,T+1] }|\Psi_U(\sfrac{1}{2}+it)|^2 \prod_{\zeta\in\scrE'}|\sfrac{1}{2}+it-\zeta|^2\right)\int_T^{T+1} |L_k(\sfrac{1}{2}+it)|^2\,dt\\&\ll \left(\max_{t\in[T,T+1] }|\Psi_U(\sfrac{1}{2}+it)|^2 \prod_{\zeta\in\scrE'}|\sfrac{1}{2}+it-\zeta|^2\right)T^6
\end{align*}
Choosing $c$ large enough in \eqref{PsiUbdd} then gives $\int_T^{T+1} |f(\frac{1}{2}+it)|^2\,dt\ll1$ for all $T\in\RR$. Arguing in a similar manner, though using Lemma \ref{apriori} in place of Theorem \ref{thm1} gives $\int_T^{T+1} |f(\sigma+it)|^2\,dt\ll e^{B(1+|T|)}$ for all $\sigma\in[\frac{1}{2},\frac{3}{2}]$ and some $B>0$. Now applying Lemma \ref{PhragLind} gives
\begin{equation*}
\int_T^{T+1} |f(\sigma+it)|^2\,dt\leq C\qquad\forall \sigma\in[\sfrac{1}{2},\sfrac{3}{2}],\; T\in\RR
\end{equation*}
for some $C>0$ (which depends on $U$). This then gives
\begin{equation*}
\int_{1/2}^{3/2}\int_T^{T+1} |f(\sigma+it)|^2\,dt \,d\sigma= \int_T^{T+1}\int_{1/2}^{3/2}|f(\sigma+it)|^2\,d\sigma\,dt\leq C \qquad \forall T\in\RR.
\end{equation*} 
Thus, for every $T\in \RR$ there exists $T\leq T'\leq T+1$ such that $\int_{1/2}^{3/2}|f(\sigma+iT')|^2\,d\sigma\leq C$. We may therefore find a sequence $\lbrace T_j\rbrace_{j=0}^{\infty}$ such that $0<T_j\leq T_{j+1}$, $T_j\rightarrow\infty$ as $j\rightarrow\infty$, and
\begin{equation}\label{Tjintbdd}
\int_{1/2}^{3/2}|f(\sigma\pm iT_j)|^2\,d\sigma\leq C\qquad \forall j\in\NN
\end{equation}
(observe that $\overline{f(s)}=f(\overline{s})$). Once again using $|L_k(\frac{3}{2}+it)|\ll1$ and \eqref{PsiUbdd}, we note that the integral $\int_{\frac{3}{2}+i\RR} L_k(s)\Psi_U(s)M^s\,ds $ converges absolutely, hence
\begin{align*}
\int_{\frac{3}{2}+i\RR} L_k(s)\Psi_U(s) M^s\,ds=\lim_{j\rightarrow\infty} \int_{-T_j}^{T_j} L_k(\sfrac 3 2 + it )\Psi_U(\sfrac 3 2 + it) M^{\sfrac 3 2 + it}\,dt.
\end{align*}
By Cauchy's integral theorem, for all large enough $T_j$, 
\begin{align*}
 \int_{-T_j}^{T_j} L_k(\sfrac 3 2 + it )\Psi_U(\sfrac 3 2 + it) M^{\sfrac 3 2 + it}\,dt= &\int_{-T_j}^{T_j} L_k(\sfrac 1 2 + it )\Psi_U(\sfrac 1 2 + it) M^{\sfrac 1 2 + it}\,dt\\&+\quad2\pi i\sum_{\zeta\in\scrE'}\mathrm{Res}_{s=\zeta}\big( L_k(s)\Psi_U(s)M^s\big)\\&\quad+\int_{\sfrac 1 2}^{\sfrac 3 2} L_k(\sigma+iT_j)\Psi_U(\sigma+i T_j)M^{\sigma+iT_j}\,d\sigma\\&\quad-\int_{\sfrac 1 2}^{\sfrac 3 2} L_k(\sigma-iT_j)\Psi_U(\sigma-i T_j)M^{\sigma-iT_j}\,d\sigma.
\end{align*}
Now, by \eqref{Tjintbdd},
\begin{align*}
&\lim_{j\rightarrow\infty} \left|\int_{\sfrac 1 2}^{\sfrac 3 2} L_k(\sigma\pm iT_j)\Psi_U(\sigma\pm i T_j)M^{\sigma\pm iT_j}\,d\sigma \right|\\&\qquad=\lim_{j\rightarrow\infty} \left|\int_{\sfrac 1 2}^{\sfrac 3 2} f(\sigma+iT_J)M^{\sigma\pm iT_j}\left(\prod_{\zeta\in\scrE'}(s-\zeta)\right)^{-1}\,d\sigma \right|\\&\qquad\leq\lim_{j\rightarrow\infty} \sqrt{C}M^{3/2}\sqrt{\int_{\sfrac 1 2 }^{\sfrac 3 2} \prod_{\zeta\in\scrE'}| \sigma \pm i T_j -\zeta|^{-2}\,d\sigma}=0,
\end{align*}
since $\scrE'$ is non-empty. Theorem \ref{thm1} and \eqref{PsiUbdd} give
\begin{align*}
&\int_{-\infty}^{\infty} |L_k(\sfrac{1}{2}+it) \Psi_U(\sfrac{1}{2}+it)M^{\sfrac 1 2 + it}|\,dt\\&\quad\leq 2\sqrt{M}\sum_{n=0}^{\infty}\sqrt{\int_n^{n+1}|L_k(\sfrac{1}{2}+it)|^2\,dt}\sqrt{\int_n^{n+1}|\Psi_U(\sfrac{1}{2}+it)|^2}\,dt<\infty,
\end{align*}
hence 
\begin{equation*}
\lim_{j\rightarrow\infty} \int_{-T_j}^{T_j} L_k(\sfrac 1 2 + it )\Psi_U(\sfrac 1 2 + it) M^{\sfrac 1 2 + it}\,dt=\int_{-\infty}^{\infty} L_k(\sfrac 1 2 + it )\Psi_U(\sfrac 1 2 + it) M^{\sfrac 1 2 + it}\,dt.
\end{equation*}
It follows that
\begin{equation*}
\int_{\frac{3}{2}+i\RR} L_k(s)\Psi_U(s) M^s\,ds=\int_{\frac{1}{2}+i\RR} L_k(s)\Psi_U(s) M^s\,ds+2\pi i\sum_{\zeta\in\scrE'}\mathrm{Res}_{s=\zeta}\big( L_k(s)\Psi_U(s)M^s\big).
\end{equation*}
Entering this into \eqref{Mellin} yields
\begin{equation}\label{shifted}
\sum_{\underset{m\neq 0}{m\in\ZZ}} |\psi_{m,k}^{(k_0)}(\sfrac 1 2 + i t_0)|^2\psi_U\left(\sfrac{|m|}{M} \right)=\frac{1}{2\pi i}\int_{\frac{1}{2}+i\RR} L_k(s)\Psi_U(s) M^s\,ds+\sum_{\zeta\in\scrE'}\mathrm{Res}_{s=\zeta}\big( L_k(s)\Psi_U(s)M^s\big).
\end{equation}
We now estimate the terms in the right-hand side of \eqref{shifted}; we start with the integral:
\begin{align*}
\left|\frac{1}{2\pi i}\int_{\frac{1}{2}+i\RR} L_k(s)\Psi_U(s) M^s\,ds\right|\ll \sqrt{M}+\sqrt{M}\int_{1}^{\infty}|L_k(\sfrac{1}{2}+it)\Psi_U(\sfrac{1}{2}+it)|\,dt
\end{align*}
(both here and henceforth, all implied constants are now independent of $U$). Defining $F(T):=\int_1^T|L_k(\frac{1}{2}+it)|^2\,dt$, by Theorem \ref{thm1}, $F(T)\ll_{\epsilon} T^{6+\epsilon}$. For $\eta_0>0$ and $c>\frac{7+\epsilon}{2}+\eta_0$, we use this bound together with \eqref{PsiUbdd}, the Cauchy-Schwarz inequality, and integration by parts in order to bound $\int_{1}^{\infty}|L_k(\sfrac{1}{2}+it)\Psi_U(\sfrac{1}{2}+it)|\,dt$ as follows:
\begin{align*}
&\int_{1}^{\infty}|L_k(\sfrac{1}{2}+it)\Psi_U(\sfrac{1}{2}+it)|\,dt\ll\lim_{T\rightarrow\infty}U^{c-1}\int_1^T |L_k(\sfrac{1}{2}+it)|t^{-c}\,dt\\&\leq \lim_{T\rightarrow\infty} U^{c-1}\sqrt{\int_1^{T}|L_k(\sfrac{1}{2}+it)|^2 t^{1+2\eta_0-2 c}\,dt}\sqrt{\int_1^{T}t^{-(1+2\eta_0)}\,dt}\\&= \lim_{T\rightarrow\infty}  U^{c-1}\sqrt{ \frac{1-T^{-2\eta_0}}{2\eta_0}} \sqrt{ \big[F(t)t^{1+2\eta_0-2c} \big]_{t=1}^{t=T}+(2c-1-2\eta_0)\int_1^T F(t)t^{2\eta_0-2c}\,dt}
\\&\quad\ll \lim_{T\rightarrow\infty}  U^{c-1}\sqrt{ \frac{1-T^{-2\eta_0}}{2\eta_0}}\sqrt{ T^{6+\epsilon+2\eta_0-2c} +(2c-1-2\eta_0)\int_1^T t^{6+\epsilon+2\eta_0-2c}\,dt}\\&\quad= \lim_{T\rightarrow\infty} U^{c-1}\sqrt{ \frac{1-T^{-2\eta_0}}{2\eta_0}}\sqrt{ T^{6+\epsilon+2\eta_0-2c} +\left(\frac{2c-1-2\eta_0}{2c-6-\epsilon-2\eta_0}\right)(1-T^{6+\epsilon+2\eta_0-2c})}\\&\quad= U^{c-1}\sqrt{ \frac{1}{2\eta_0}}\sqrt{\frac{2c-1-2\eta_0}{2c-6-\epsilon-2\eta_0}}\ll_{\eta_0} U^{c-1}.
\end{align*}
Thus, writing $c=\frac{7}{2}+\eta$ for some $\eta>0$, we obtain
\begin{equation}\label{Lintbdd}
\left|\int_{\frac{1}{2}+i\RR} L_k(s)\Psi_U(s) M^s\,ds\right|\ll_{\eta}\sqrt{M}U^{\frac{5}{2}+\eta}.
\end{equation}
Turning our attention to the residues in the left-hand side of \eqref{shifted}, we start by observing that $\Psi_U(s)$ and $M^s$ are holomorphic in $s$. We recall now the following facts regarding $L_k(s)$ and $\scrE'$: by Proposition \ref{GL=RS}, $L_k(s)=\frac{R_k(|E|^2,s)}{G(s)}$, where $\frac{1}{G(s)}$ is holomorphic for $s\in[\frac{1}{2},\frac{3}{2}]+i\RR$, and (using the assumption $t_0\neq 0$) $R_k(|E|^2,s)$ has a pole of order two at $s=1$, and simple poles at $\lbrace 1\pm 2it_0\rbrace\cup(\scrE_k\cap\big(\sfrac 1 2 ,1)\big))$. This gives
\begin{equation*}
\Res_{s=\zeta}\big( L_k(s)\Psi_U(s)M^s \big)=\frac{M^{\zeta}\Psi_U(\zeta)}{G(\zeta)}\Res_{s=\zeta}\big( R_k(|E|^2,s) \big),\qquad \zeta\in (\sfrac 1 2 , 1)\cap \scrE'.
\end{equation*}
By \cite[(28)]{Petridis95} $\Psi_U(s)=\frac{1}{s}+O_s(\frac{1}{U})$ (with the implied constant being uniform over $s$ in compact subsets of $\lbrace z\in\CC\,:\, \Re(z)>0\rbrace$), hence
\begin{equation*}
\sum_{\zeta\in (\frac{1}{2},1)\cap\scrE_k} \Res_{s=\zeta}\big( L_k(s)\Psi_U(s)M^s \big)=\sum_{\zeta\in (\frac{1}{2},1)\cap\scrE_k} c_{\zeta}M^{\zeta}+ O\big( \sfrac{M}{U}\big),
\end{equation*}
where
\begin{equation}\label{czeta}
c_{\zeta} =\frac{\Res_{s=\zeta}\big( R_k(|E|^2,s) \big)}{\zeta G(\zeta)} ,
\end{equation} 
and $|M^{\zeta}|\leq M$ was used to obtain $O\big( \sfrac{M}{U}\big)$. Turning to those $\zeta\in\scrE'$ with $\Re(\zeta)=1$, and again using $\Psi_U(s)=\frac{1}{s}+O(\frac{1}{U})$,
\begin{align*}
\mathrm{Res}_{s=1\pm 2i t_0}\big( L_k(s)\Psi_U(s)M^s\big)=\frac{M^{1\pm 2i t_0}\Res_{s=1\pm 2 i t_0}\big( R_k(|E|^2,s) \big)}{(1\pm 2i t_0)G(1\pm 2 i t_0)}+ O\big( \sfrac{M}{U}\big).
\end{align*}
By \eqref{RSRenormexp} and \eqref{EisXi},
\begin{align*}
\Res_{s=1\pm 2 i t_0}\big( R_k(|E|^2,s) \big)=&\lim_{s\rightarrow 1\pm 2 i t_0} \big(s-(1\pm 2i t_0)\big) R_k(|E|^2,s)\\&=\varphi_{k,k_0}(1\pm 2 i t_0)\varphi_{k_0,k_0}(\sfrac 1 2 \mp i t_0).
\end{align*}
Now, $\overline{\varphi_{j,k}(s)}=\varphi_{j,k}(\overline{s})$ and $\overline{G(s)}=G(\overline{s})$, giving
\begin{align*}
&\mathrm{Res}_{s=1+ 2i t_0}\big( L_k(s)\Psi_U(s)M^s\big)+\mathrm{Res}_{s=1- 2i t_0}\big( L_k(s)\Psi_U(s)M^s\big)
\\&=2\Re\left( \frac{M^{1+ 2i t_0}\varphi_{k,k_0}(1+ 2 i t_0)\varphi_{k_0,k_0}(\sfrac 1 2 - i t_0)}{(1+ 2i t_0)G(1+ 2 i t_0)}\right)+ O\big( \sfrac{M}{U}\big)
\\&=\Re\left( M^{1+2i t_0}\frac{8(2\pi)^{2i t_0}\varphi_{k,k_0}(1+ 2 i t_0)\varphi_{k_0,k_0}(\sfrac 1 2 - i t_0)\Gamma(1+ i t_0)}{\Gamma(\frac{1}{2}+2 i t_0)\Gamma(\frac{3 }{2}+i t_0)}\right)+ O\big( \sfrac{M}{U}\big).
\end{align*}
Define
\begin{equation}\label{c12it0}
c_{1+2i t_0}=\frac{8(2\pi)^{2i t_0}\varphi_{k,k_0}(1+ 2 i t_0)\varphi_{k_0,k_0}(\sfrac 1 2 - i t_0)\Gamma(1+ i t_0)}{\Gamma(\frac{1}{2}+2 i t_0)\Gamma(\frac{3 }{2}+i t_0)}.
\end{equation}

Finally, the pole at $s=1$ gives the following contribution:
\begin{align*}
\mathrm{Res}_{s=1}\big( L_k(s)\Psi_U(s)M^s\big)=\Res_{s=1} \left(  \frac{M^s \Psi_U(s)R_k(|E|^2,s)}{G(s)}\right).
\end{align*}
By Corollary \ref{RScont} (and \eqref{RSRenormexp} and \eqref{EisXi}), 
\begin{align*}
R_k(|E|^2,s)=&R.N.\left( \int_{\GaH} |E_{k_0}(\cdot,\sfrac{1}{2}+it_0)|^2E_k(\cdot,s)\,d\mu\right)\\&=\int_{\scrF_B} |E_{k_0}(z,\sfrac{1}{2}+it_0)|^2 E_k(z,s)\,d\mu(z)\\\notag&+\sum_{j=1}^{\kappa} \int_{\scrC_{j,B}} \left(|E_{k_0}(z,\sfrac{1}{2}+it_0)|^2 E_k(z,s)-(\delta_{k,j} y_j^s+\varphi_{k,j}(s)y_j^{1-s})\Xi_j(y_j)\right)\,d\mu(z)\\\notag&\qquad\qquad-\sum_{j=1}^{\kappa}\widehat{\Xi}^1_j(B),
\end{align*}
where
\begin{align*}
\Xi_j(y)=( |\varphi_{k_0,j}(\sfrac 1 2 + i t_0)|^2+\delta_{k_0,j})y+\delta_{k_0,j}\varphi_{k_0,j}(\sfrac 1 2 - i t_0)y^{1+2i t_0}+\delta_{k_0,j}\varphi_{k_0,j}(\sfrac 1 2 +it_0)y^{1-2i t_0}
\end{align*}
and for $j\neq k$:
\begin{align*}
\widehat{\Xi}^1_j(B)=&\varphi_{k,j}(s)\Bigg(\frac{( |\varphi_{k_0,j}(\sfrac 1 2 + i t_0)|^2+\delta_{k_0,j})B^{1-s}}{1-s}+\frac{\delta_{k_0,j}\varphi_{k_0,j}(\frac{1}{2} - i t_0)B^{1+2i t_0-s}}{1+2i t_0-s}\\&\qquad\qquad\qquad\qquad\qquad\qquad+\frac{\delta_{k_0,j}\varphi_{k_0,j}(\frac{1}{2} + i t_0)B^{1-2i t_0-s}}{1-2i t_0-s}\Bigg),
\end{align*}
and
\begin{align*}
\widehat{\Xi}^1_k(B)=&\frac{( |\varphi_{k_0,k}(\sfrac 1 2 + i t_0)|^2+\delta_{k_0,j})B^{s}}{s}+\frac{\delta_{k_0,j}\varphi_{k_0,j}(\frac{1}{2} - i t_0)B^{s+2i t_0}}{s+2i t_0}+\frac{\delta_{k_0,j}\varphi_{k_0,j}(\frac{1}{2} + i t_0)B^{s-2i t_0}}{s-2i t_0}\\&\quad+\varphi_{k,k}(s)\Bigg(\frac{( |\varphi_{k_0,k}(\sfrac 1 2 + i t_0)|^2+\delta_{k_0,k})B^{1-s}}{1-s}+\frac{\delta_{k_0,k}\varphi_{k_0,k}(\frac{1}{2} - i t_0)B^{1+2i t_0-s}}{1+2i t_0-s}\\&\qquad\qquad\qquad\qquad\qquad\qquad\qquad\qquad\qquad+\frac{\delta_{k_0,k}\varphi_{k_0,k}(\frac{1}{2} + i t_0)B^{1-2i t_0-s}}{1-2i t_0-s}\Bigg).
\end{align*}
By Lemma \ref{Eis1pole} and using $E_k(z,s)=\delta_{k,j} y_j^s+\left(\frac{\mu(\GaH)^{-1}}{s-1}+\widetilde{\varphi}_{k,j}(s)\right)y_j^{1-s}+O(y_j^{-A})$ (with the implied constant uniformly bounded for $s$ in a neighbourhood of $1$) and $|E_k(z,\frac{1}{2}+i t_0)|^2=\Xi_j(y_j)+O(y_j^{-A})$ for all $A>0$, $y_j\gg 1$, we obtain, as $s\rightarrow 1$,
\begin{align*}
R_k(|E|^2,s)=&\frac{\mu(\GaH)^{-1}}{s-1}\int_{\scrF_B} |E_{k_0}(z,\sfrac{1}{2}+i t_0)|^2\,d\mu(z)\\&+\sum_{j=1}^{\kappa} \varphi_{k,j}(s)\int_{\scrC_{j,B}} y_j^{1-s}\left(|E_{k_0}(z,\sfrac{1}{2}+it_0)|^2-\Xi_j(y_j)\right)\,d\mu(z)\\-\sum_{j=1}^{\kappa} &\varphi_{k,j}(s)\Bigg(\frac{( |\varphi_{k_0,j}(\sfrac 1 2 + i t_0)|^2+\delta_{k_0,j})B^{1-s}}{1-s}+\frac{\delta_{k_0,j}\varphi_{k_0,j}(\frac{1}{2} - i t_0)B^{1+2i t_0-s}}{1+2i t_0-s}\\&\qquad\qquad\qquad\qquad\qquad\qquad+\frac{\delta_{k_0,j}\varphi_{k_0,j}(\frac{1}{2} + i t_0)B^{1-2i t_0-s}}{1-2i t_0-s}\Bigg)+O(1)
\end{align*}
\begin{align*}
=&\frac{\mu(\GaH)^{-1}}{s-1}\int_{\scrF_B} |E_{k_0}(z,\sfrac{1}{2}+i t_0)|^2\,d\mu(z)\\&+\sum_{j=1}^{\kappa} \frac{\mu(\GaH)^{-1}}{s-1}\int_{\scrC_{j,B}}\left(|E_{k_0}(z,\sfrac{1}{2}+it_0)|^2-\Xi_j(y_j)\right)\,d\mu(z)\\-\sum_{j=1}^{\kappa} &\left(\frac{\mu(\GaH)^{-1}}{s-1}+\widetilde{\varphi}_{k,j}(s)\right)\Bigg(( |\varphi_{k_0,j}(\sfrac 1 2 + i t_0)|^2+\delta_{k_0,j})\left(\frac{1}{1-s}+\log B+O(|1-s|) \right)\\&+\frac{\delta_{k_0,j}\varphi_{k_0,j}(\frac{1}{2} - i t_0)B^{2i t_0}}{2i t_0}+\frac{\delta_{k_0,j}\varphi_{k_0,j}(\frac{1}{2} + i t_0)B^{-2i t_0}}{-2i t_0}+O(|1-s|)\Bigg)+O(1)
\end{align*}
\begin{align*}
=&\frac{2\mu(\GaH)^{-1}}{(s-1)^2}+\sum_{j=1}^{\kappa}\frac{( |\varphi_{k_0,j}(\sfrac 1 2 + i t_0)|^2+\delta_{k_0,j})\widetilde{\varphi}_{k,j}(1)}{s-1}\\&+\frac{\mu(\GaH)^{-1}}{s-1}\Bigg(\int_{\scrF_B} |E_{k_0}(z,\sfrac{1}{2}+i t_0)|^2\,d\mu(z)\\&\quad+\sum_{j=1}^{\kappa} \int_{\scrC_{j,B}}\left(|E_{k_0}(z,\sfrac{1}{2}+it_0)|^2-\Xi_j(y_j)\right)\,d\mu(z)\\&\qquad-\sum_{j=1}^{\kappa} \bigg((|\varphi_{k_0,j}(\sfrac 1 2 + i t_0)|^2+\delta_{k_0,j})\log B +\frac{\delta_{k_0,j}\varphi_{k_0,j}(\frac{1}{2} - i t_0)B^{2i t_0}}{2i t_0}\\&\qquad\quad+\frac{\delta_{k_0,j}\varphi_{k_0,j}(\frac{1}{2} + i t_0)B^{-2i t_0}}{-2i t_0}\Bigg)+O(1)
\end{align*}
\begin{align*}
=&\frac{2\mu(\GaH)^{-1}}{(s-1)^2}+\sum_{j=1}^{\kappa}\frac{( |\varphi_{k_0,j}(\sfrac 1 2 + i t_0)|^2+\delta_{k_0,j})\widetilde{\varphi}_{k,j}(1)}{s-1}\\\quad&+\frac{\mu(\GaH)^{-1}}{s-1}R.N.\left(\int_{\GaH}|E_{k_0}(\cdot,\sfrac{1}{2}+it_0)|^2\,d\mu\right)+O(1).
\end{align*}
Using the Maass-Selberg relations \eqref{MaasSelexp} (cf.\ the proof of Lemma \ref{maassel}), one can show that
\begin{equation*} 
R.N.\left(\int_{\GaH}|E_{k_0}(\cdot,\sfrac{1}{2}+it_0)|^2\,d\mu\right)=-\sum_{j=1}^{\kappa}\varphi_{k_0,j}'(\sfrac{1}{2}+i t_0)\varphi_{k_0,j}(\sfrac{1}{2}-i t_0),
\end{equation*}
hence
\begin{equation*}
R_k(|E|^2,s)=\frac{2\mu(\GaH)^{-1}}{(s-1)^2} + \frac{c_1}{s-1} +O(1),
\end{equation*}
where
\begin{align}\label{c1def}
c_1=& -\sum_{j=1}^{\kappa}\frac{\varphi_{k_0,j}'(\sfrac{1}{2}+i t_0)\varphi_{k_0,j}(\sfrac{1}{2}-i t_0)}{\mu(\GaH)}+\sum_{j=1}^{\kappa}( |\varphi_{k_0,j}(\sfrac 1 2 + i t_0)|^2+\delta_{k_0,j})\widetilde{\varphi}_{k,j}(1),
\end{align}
as $s\rightarrow 1$.
This gives
\begin{align*}
&\Res_{s=1}\big( L_k(s)\Psi_U(s)M^s\big)=\frac{2}{\mu(\GaH)}\left( \frac{\Psi_U(1)M\log M}{G(1)}+\frac{M\Psi_U'(1)}{G(1)}-\frac{M\Psi_U(1)G'(1)}{G(1)^2}\right)\\&\qquad\qquad\qquad\qquad\qquad\qquad\qquad\qquad\qquad\qquad+\frac{c_1 M \Psi_U(1)}{G(1)}\\&=\frac{2}{\mu(\GaH)G(1)}M\log M +\frac{1}{\mu(\GaH)G(1)}\left(-2+c_1\mu(\GaH)-2\frac{G'(1)}{G(1)} \right)M+O\big(\sfrac{M\log M}{U}\big)
\end{align*}
(here we used $\Psi_U'(1)=\int_0^{1-\frac{1}{U}} \log x \,dx+\int_{1-\frac{1}{U}}^{1+\frac{1}{U}} \psi_U(x)\log x\,dx =-1+O(\frac{1}{U}) $). Now,
\begin{align*}
&G(1)=\frac{\Gamma(\frac{1}{2}+i t_0)\Gamma(\frac{1}{2}-i t_0)\Gamma(\frac{1}{2})^2}{8\pi}=\frac{\pi}{8\cosh(\pi t_0)}\\
&\frac{G'(1)}{G(1)}=\frac{1}{2}\cdot\frac{\Gamma'(\frac{1}{2}+i t_0)}{\Gamma(\frac{1}{2}+i t_0)}+\frac{1}{2}\cdot\frac{\Gamma'(\frac{1}{2}-i t_0)}{\Gamma(\frac{1}{2}-i t_0)}+2\cdot\frac{\frac{1}{2}\cdot\Gamma'(\frac{1}{2})}{\Gamma(\frac{1}{2})}-\frac{8\pi\big(\log\pi+\Gamma'(1)\big)}{8\pi \Gamma(1)}\\&\qquad\qquad=\Re\left(\frac{\Gamma'(\frac{1}{2}+i t_0)}{\Gamma(\frac{1}{2}+i t_0)} \right)-\log (4\pi),
\end{align*}
hence
\begin{align*}
\Res_{s=1}\big( L_k(s)\Psi_U(s)M^s\big)=&\frac{16\cosh(\pi t_0)}{\mu(\GaH)\pi}M\log M \\&+\frac{8\cosh(\pi t_0)}{\mu(\GaH)\pi}\left(c_1\mu(\GaH)+2\log (4\pi)-2-2\Re\left(\sfrac{\Gamma'(\frac{1}{2}+i t_0)}{\Gamma(\frac{1}{2}+i t_0)} \right) \right)M\\&\quad+O\left(\frac{M\log M}{U}\right).
\end{align*}
Entering the formulas for the residues and \eqref{Lintbdd} into \eqref{shifted} gives
\begin{align*}
\sum_{\underset{m\neq 0}{m\in\ZZ}} |\psi_{m,k}^{(k_0)}(\sfrac 1 2 + i t_0)|^2\psi_U&\left(\sfrac{|m|}{M} \right)=\frac{16\cosh(\pi t_0)}{\mu(\GaH)\pi}M\log M \\&+\frac{8\cosh(\pi t_0)}{\mu(\GaH)\pi}\left(c_1\mu(\GaH)+2\log (4\pi)-2-2\Re\left(\sfrac{\Gamma'(\frac{1}{2}+i t_0)}{\Gamma(\frac{1}{2}+i t_0)} \right) \right)M\\&+\Re(M^{1+2 i t_0} c_{1+2 i t_0})+\!\!\sum_{\zeta\in (\frac{1}{2},1)\cap\scrE_k}\!\! c_{\zeta}M^{\zeta}+O_{\Gamma, t_0,\eta}\big( \sfrac{M\log M}{U}+\sqrt{M}U^{\frac{7}{2}+\eta}\big).
\end{align*}
Note now that our choice of $\psi_U$ gives
\begin{equation*}
\sum_{\underset{0<|m|\leq M(1-U^{-1})}{m\in\ZZ}}\!\!\!\! \!\!\!\!|\psi_{m,k}^{(k_0)}(\sfrac 1 2 + i t_0)|^2\leq\sum_{\underset{m\neq 0}{m\in\ZZ}} |\psi_{m,k}^{(k_0)}(\sfrac 1 2 + i t_0)|^2\psi_U\left(\sfrac{|m|}{M} \right)\leq\!\!\!\!\sum_{\underset{0<|m|\leq M(1+U^{-1})}{m\in\ZZ}}\!\!\!\!\!\!\!\! |\psi_{m,k}^{(k_0)}(\sfrac 1 2 + i t_0)|^2.
\end{equation*}
Letting $M'=M(1-U^{-1})$, we then have $M=M'+O(\frac{M'}{U})$ and $M\log  M=M'\log M'+O(\frac{M'\log M'}{U})$ (note that $U\geq 2$), giving
\begin{align*}
\sum_{\underset{0<|m|\leq M'}{m\in\ZZ}}& |\psi_{m,k}^{(k_0)}(\sfrac 1 2 +t_0)|^2\leq \frac{16\cosh(\pi t_0)}{\mu(\GaH)\pi}(M'\log M'+O(\sfrac{M'\log M'}{U})) \\&\qquad+\frac{8\cosh(\pi t_0)}{\mu(\GaH)\pi}\left(c_1\mu(\GaH)+\log (4\pi)-2-2\Re\left(\sfrac{\Gamma'(\frac{1}{2}+i t_0)}{\Gamma(\frac{1}{2}+i t_0)} \right) \right)(M'+O(\sfrac{M'}{U}))\\&\qquad+\Re\big((M'+O(\sfrac{M'}{U}))^{1+2 i t_0} c_{1+2 i t_0}\big)+\!\!\sum_{\zeta\in (\frac{1}{2},1)\cap\scrE}\!\! c_{\zeta}(M'+O(\sfrac{M'}{U}))^{\zeta}\\&\qquad+O_{\Gamma, t_0,\eta}\big( \sfrac{M'\log M'}{U}+\sqrt{M'}U^{\frac{5}{2}+\eta}\big),
\end{align*}
hence
\begin{align*}
\sum_{\underset{0<|m|\leq M'}{m\in\ZZ}}& |\psi_{m,k}^{(k_0)}(\sfrac 1 2 +t_0)|^2\leq \frac{16\cosh(\pi t_0)}{\mu(\GaH)\pi}M'\log M'\\&\qquad+\frac{8\cosh(\pi t_0)}{\mu(\GaH)\pi}\left(c_1\mu(\GaH)+2\log (4\pi)-2-2\Re\left(\sfrac{\Gamma'(\frac{1}{2}+i t_0)}{\Gamma(\frac{1}{2}+i t_0)} \right) \right)M'\\&\qquad+\Re\big(M'^{1+2 i t_0} c_{1+2 i t_0}\big)+\!\!\sum_{\zeta\in (\frac{1}{2},1)\cap\scrE_k}\!\! c_{\zeta}M'^{\zeta}\\&\qquad+O_{\Gamma, t_0,\eta}\big( \sfrac{M'\log M'}{U}+\sqrt{M'}U^{\frac{5}{2}+\eta}\big).
\end{align*}
Similarly, letting $M'=M(1+U^{-1})$ again gives $M=M'+O(\frac{M'}{U})$ and $M\log  M=M'\log M'+O(\frac{M'\log M'}{U})$, hence
\begin{align*}
\sum_{\underset{0<|m|\leq M'}{m\in\ZZ}}& |\psi_{m,k}^{(k_0)}(\sfrac 1 2 +t_0)|^2\geq \frac{16\cosh(\pi t_0)}{\mu(\GaH)\pi}M'\log M'\\&\qquad+\frac{8\cosh(\pi t_0)}{\mu(\GaH)\pi}\left(c_1\mu(\GaH)+2\log (4\pi)-2-2\Re\left(\sfrac{\Gamma'(\frac{1}{2}+i t_0)}{\Gamma(\frac{1}{2}+i t_0)} \right) \right)M'\\&\qquad+\Re\big(M'^{1+2 i t_0} c_{1+2 i t_0}\big)+\!\!\sum_{\zeta\in (\frac{1}{2},1)\cap\scrE_k}\!\! c_{\zeta}M'^{\zeta}\\&\qquad+O_{\Gamma, t_0,\eta}\big( \sfrac{M'\log M'}{U}+\sqrt{M'}U^{\frac{5}{2}+\eta}\big).
\end{align*}
Replacing $M'$ with $M$ now gives
\begin{align*}
\sum_{\underset{0<|m|\leq M}{m\in\ZZ}}& |\psi_{m,k}^{(k_0)}(\sfrac 1 2 +t_0)|^2= \frac{16\cosh(\pi t_0)}{\mu(\GaH)\pi}M\log M\\&\qquad+\frac{8\cosh(\pi t_0)}{\mu(\GaH)\pi}\left(c_1\mu(\GaH)+2\log (4\pi)-2-2\Re\left(\sfrac{\Gamma'(\frac{1}{2}+i t_0)}{\Gamma(\frac{1}{2}+i t_0)} \right) \right)M\\&\qquad+\Re\big(M^{1+2 i t_0} c_{1+2 i t_0}\big)+\!\!\sum_{\zeta\in (\frac{1}{2},1)\cap\scrE_k}\!\! c_{\zeta}M^{\zeta}\\&\qquad+O_{\Gamma, t_0,\eta}\big( \sfrac{M\log M}{U}+\sqrt{M}U^{\frac{5}{2}+\eta}\big).
\end{align*}
We now let $U=M^{\frac{1}{7}}$, giving $\frac{M}{U}=\sqrt{M}U^{\frac{5}{2}}=M^{\frac{6}{7}}$, completing the proof.
 
\hspace{425.5pt}\qedsymbol

\end{document}